\documentclass[12pt]{article}

\newcommand{\fmtext}{}

\newcommand{\maartenbluefont}{}
\newcommand{\fmrtext}{}

\newcommand{\corrtext}{}

\newcommand{\Lambdamap}{X} 
\newcommand{\Lambdamapc}{{\fmtext X}_c}

\newcommand{\randomYvariable}{{\bf y}}
\newcommand{\yvariable}{y}

\newcommand{\randomZvariable}{{\bf z}} 

\newcommand{\newrandomMvariable}{{\bf M}}%
\newcommand{\newrandomXvariable}{{\bf X}}%

\newcommand{\xxtext}{}
\newcommand{\ytext}{}
\newcommand{\ztext}{}
\newcommand{\zztext}{}
\newcommand{\yytext}{}
\newcommand{\qtext}{}
\newcommand{\ptext}{}
\newcommand{\ktext}{}
\newcommand{\ftext}{}
\newcommand{\mtext}{}

\newcommand{\rtext}{}

\newcommand{\commentedinfinal}[1]{}

\usepackage{amsthm}
\usepackage{amsmath}
\usepackage{enumerate,enumitem}

\usepackage{xr}

\usepackage{epsfig,amsfonts,latexsym}
\usepackage{mathrsfs}
\usepackage{color}

\usepackage{amsopn}

\DeclareMathOperator {\Vol} {vol}

\newcommand{\B}{{\mathcal B}}

\newcommand{\Expec}{\mathbb{E}}
\newcommand{\Prob}{\mathbb{P}}
\newcommand{\R}{\mathbb{R}}

\newcommand{\ND}{\mathcal{M}_{ND}}

\newcommand{\diam}{\operatorname{diam}}

\def\p{\partial}
\def\Rec{\mathcal R} 

\def\wavecap{\textnormal{cap}}
\let\bdy\partial
\newcommand\indicator{\mathbf 1}

\newcommand{\norm}[1]{\left\|#1 \right\|}

\newcommand{\e}{\varepsilon}
\newcommand{\dd}{\operatorname{d} \!}

\newcommand{\C}{\mathbb{C}}

\DeclareMathOperator{\relu}{\operatorname{ReLU}} 

\newcommand{\ba}{\begin{eqnarray*}}
\newcommand{\ea}{\end{eqnarray*}}
\newcommand{\baa}{\begin{eqnarray*}}
\newcommand{\eaa}{\end{eqnarray*}}
\def \beq {\begin {eqnarray}}
\def \eeq {\end {eqnarray}}
\def\bra{\langle}
\def\cet{\rangle}                 

\newtheorem{theorem}{Theorem}[section]
\newtheorem{prop}[theorem]{Proposition}
\newtheorem{lemma}[theorem]{Lemma}

\theoremstyle{definition}
\newtheorem{definition}         {Definition}[section]
\newtheorem{remark}{Remark}

\title{Deep learning architectures for nonlinear operator functions and  nonlinear inverse problems}

\author{Maarten V. de Hoop\thanks{Rice University ({\tt
      mdehoop@rice.edu})}
\and Matti Lassas\thanks{University of Helsinki ({\tt
    matti.lassas@helsinki.fi})}
\and Christopher A. Wong\thanks{Rice University ({\tt
    cawong@rice.edu})} }
\date{\today}
\begin{document}

\maketitle


\begin{abstract}
We develop a theoretical analysis for special neural network architectures, termed \emph{operator recurrent neural networks}, for approximating nonlinear functions whose inputs are linear operators. Such functions commonly arise in solution algorithms for
inverse boundary value problems. Traditional neural networks treat input data as vectors, and thus they do not effectively capture the multiplicative structure associated with the linear operators that correspond to the data in such inverse problems. We
therefore introduce a new family that resembles a standard neural network architecture, but where the input data acts multiplicatively on vectors. Motivated by compact operators appearing in boundary control and the analysis of inverse boundary value problems for the wave equation, we promote structure and sparsity in selected weight matrices in the network. After describing this architecture, we study its representation properties as well as its approximation properties. We furthermore show that an explicit regularization can be introduced that can be derived from the mathematical analysis of the mentioned inverse problems, and which leads to certain guarantees on the generalization properties. We observe that the sparsity of the weight matrices improves the generalization estimates. Lastly, we discuss how operator recurrent networks can be viewed as a deep learning analogue to deterministic algorithms such as boundary control for reconstructing the unknown wavespeed in the acoustic wave equation from boundary measurements.
\end{abstract}

\smallskip


\section{Introduction}

In standard deep learning, the input data are represented by vectors,
and each layer of a deep neural network applies an affine
transformation (a matrix-vector product plus a shift) composed with
nonlinear activation functions. However, for functions for which the
input data are linear operators, vectorizing the input destroys the
underlying operator structure. Functions whose inputs are linear
operators, which we term \emph{nonlinear operator functions}, are
present in a broad class of nonlinear inverse problems for partial
differential equations (PDE). That is, the possible reconstructions
associated with such problems involve nonlinear, nonlocal functions
between spaces of data operators and function spaces of
``images''. {\maartenbluefont Optimality of reconstruction algorithms can
  be studied with statistical decision theory; however, machine
  learning offers data-driven approaches that make such studies
  computationally feasible.}

\subsection{Nonlinear operator functions, inverse problems and
  reconstruction}

We focus our attention on \emph{nonlinear operator functions}, meaning
nonlinear functions whose input consists of linear operators, and
whose structure consists of a holomorphic function of an operator
composed with a very regular function. This type of function structure
is found in a variety of existing solution procedures for nonlinear
inverse problems arising from hyperbolic PDEs. The model problem is
reconstruction of, or ``imaging'' the unknown speed {\mtext $c =
  c(x)$} of waves inside a body, based on from boundary measurements. In this
problem, the body is probed by multiple boundary sources, $h$,
generating waves; the waves that come back are measured at the
boundary. The boundary measurements corresponding to an operator ${\fmtext \Lambdamapc}:h
\mapsto {\fmtext \Lambdamapc}(h)$, and the inverse problem of determining
$c$ from ${\fmtext \Lambdamapc}$ is highly nonlinear.  This inverse problem
has been extensively studied, e.g. in \cite{Anderson,belishev1987approach,belishev1992reconstruction,KKL,Katchalov,lassasDuke,LiuOksanen,Oksanen,StU,SU,U}
and the stability of the solution with data containing errors is considered in \cite{AlessandiniSylverster,Anderson,bellassoued2010stability}. The inverse problems for the wave equation with given 
boundary measurements $\Lambdamapc$  corresponds to the case when we observe the complete wave patterns on the 
boundary. This  inverse problem is closely related to the inverse travel time problem where only the first arrival 
times of the waves are observed, see \cite{burago2010boundary,lassas2003semiglobal,PestovUhlmann,stefanov2008boundary,stefanov2016boundary,U}. Even though the underlying physical system, for
  example, the wave equation, is a linear equation, the inverse
  problem of finding the coefficient function of this equation is a
  non-linear problem. {\mtext In general, we consider ${\fmtext
    \Lambdamapc}$ as data given to us and denote it by ${\fmtext
    \Lambdamap} = {\fmtext \Lambdamapc}$.}

Established uniqueness proofs, based on boundary control
\cite{bingham2008, korpela2016, dHkepleyO2018} and scattering control
\cite{caday1, caday2}, for the above mentioned inverse problems lead
to solution procedures that are recursive in the data operator,
$\Lambdamap$. These procedures can be viewed as applying an
operator-valued series expansion in terms of $\Lambdamap$ followed by
some elementary operations such as taking inner products and
divisions. Typically, one starts with a boundary source $h_0$,
measures the wave $\Lambdamap(h_0)$ at the boundary and computes a new
source $h_1$ using both $h_0$ and $\Lambdamap(h_0)$. The process is
iterated to thus produce a sequence of sources that converge to an
optimal source, called a control, which can effectively determine
information about the interior. However, the convergence is typically
very slow while the intrinsic stability of the inverse problem is
poor. Therefore, a natural question is whether the procedures can be
replaced by learned procedures that are adapted to the data, taking
advantage of working on a low-dimensional manifold of linear
operators. The iterative nature of the procedures suggests the
introduction of recurrent neural networks (RNNs). Mathematical
properties of the inverse problems can be used to reduce the number of
weights to be learned. Notably, a crucial feature of boundary control
is that each iteration involves linear operators that smooth source
signals by a finite order, meaning that such operators are compact
operators. The compactness is used in a crucial way in the solution of
the inverse problem. Moreover, when the data operator and operators
appearing in the boundary or scattering control based procedures are
discretized and approximated by finite $n \times n$ matrices, one
obtains good approximations using sparse and low-rank matrices.

The main goal of this paper is to develop a mathematical framework for
supervised learning to solve nonlinear inverse problems, whose
underlying structure is that of nonlinear operator functions. Based on
the structure of known, constructive uniqueness proofs, we introduce
general operator recurrent neural networks that take data in as a
linear operator. We further introduce an explicit
regularization scheme for training such networks based on compactness,
sparsity and rank properties of certain operators embedded in the
network. The result is a principled network architecture for which
crucial analytic features can be controlled tightly. This stands in
contrast to more traditional applications of deep neural networks,
such as computer vision and speech recognition, in which little
mathematical information about the behavior of the underlying
``function'' is known. To highlight the potential of deep learning in
the context of inverse problems, we prove that our type of network,
the weights of which are obtained via training with simulated data,
solves the inverse problems \textit{at least as well as} the
classical, partial-differential-equation based reconstruction
procedures. We analyze the approximation and detailed expressivity
properties of our operator recurrent neural networks, and provide
generalization estimates and rates with increasing training sets to
the best possible network. The universal approximation theorems only
guarantee a small approximation error for a sufficiently large network,
but do not consider the optimization (training) and generalization
errors, which are equally important \cite{Jin2019}.
From the viewpoint of studying inverse problem, the deep learning framework provides a novel integration of analysis and statistics. In this framework, the architecture is derived from the analysis as a domain adaptive ingredient, while statistical decision theory is used to define what is meant by an ``optimal'' solution method involving regularization with a finite set of training ``data''.

{ Formally, we consider inverse problems of the form $X = F(z)$, where $F$ is a direct operator acting on real-valued {vectors $z$} generating linear operators $X$, and are concerned with determining $z$ given $X$. {The vector $z$ 
models a real-valued function that is
 digitized in $m$ points whereas the functions on which $X$ acts are digitized with $n$ points. Thus we view $z$ as a vector in $\R^m$ and $X$ as a matrix in $\R^{n \times n}$ with $m > n$. We will assume uniqueness. In the digitized framework, we let $z \in B^m(\rho_0)$ and $${\color{black}\mathcal X} = F(B^m(\rho_0)) \subset \mathcal{B}^{n \times n}(\rho_1) ;$$ here, $B^m(\rho_0)$ denotes a ball with radius $\rho_0$ in $\R^m$ equipped with the standard Euclidean norm and $\mathcal{B}^{n \times n}(\rho_1)$ a ball with radius $\rho_1$ in $\R^{n \times n}$ equipped with the operator norm of linear operators $\R^n\to\R^n,$
 that is, $\|X\|_{\R^{n\times n}}=\max_{\|v\|_{\R^n}\leq 1}\|Xv\|_{\R^n}$.
 When the map $F$ is injective and one is given (or measures) the matrix $X$ as data, 
 and this data does not contain errors, one can consider a map
 $H$ that is the left inverse of $F$ on $\operatorname{Ran}(F)={\color{black}\mathcal X} = F(B^m(\rho_0))\subset \R^{n \times n}$, that is,
one consider the sequence
\begin{equation*} 
   B^m(\rho_0)\xrightarrow{\ F\ }{\color{black}\mathcal X}   
   \xrightarrow{\ H }\R^m ,\quad
   H(F(z))=z\ \hbox{for }z\in B^m(\rho_0).
\end{equation*}
However, one has to deal with two challenges: Computing the map $H$ may be difficult and the data $X=F(z)$ may contain errors.
 
We consider a strategy that is rooted in the analysis of inverse problems, when the reconstruction is obtained in two steps. In the first step, one constructs an intermediate quantity, $y \in \R^n$, from $X$, which is typically relatively unstable; this construction may have to be repeated for a variable parameter which, upon discretization, yields $(y^1,\ldots,y^{{T}})$. From $(y^1,\ldots,y^{{T}})$ one then obtains $z$, typically in a stable manner. To formalize this, we assume that there are functions
\beq\label{layers 1}
 &&  \boldsymbol{f} = (f^1,\ldots,f^T): \R^{n \times n} \to (\R^n)^T, \text{ where}\ f^t :\ \R^{n \times n} \to \R^n ,
 \eeq
 and
\beq\label{layers 2}
  & & g :\ (\R^n)^T \to \R^m
\eeq  
 having the property that $g\circ \boldsymbol{f}:\R^{n\times n}\to \R^m$
is an extension of the inverse  map $H$ of $F$ defined on $\mathcal X$, that is,
 we consider the sequence
\beq \label{gfF sequence 1}
& &B^m(\rho_0)\xrightarrow{\ F\ } \R^{n \times n} 
\xrightarrow{\ \boldsymbol{f}\ }(\R^{n})^T \xrightarrow{\ g\ }\R^m,
\\ \label{gfF sequence 2}
 & &g(\boldsymbol{f}(F(z))=z\ \hbox{for }z\in B^m(\rho_0).
\eeq
As discussed above, the intermediate quantities are denoted by
 \ba 
 & &f^t(X)=y^t ,\ t = 1,\ldots,T.
 \ea 
 We will approximate $f^t$ by $f^t_{\theta}$ as an operator recurrent network and $g$ by $g_{\theta}$ as a shallow network with fully connected layers motivated, again, by the analysis of inverse problems. {\color{black} We consider the model, where parameter $\theta$ consists of two parts,
$\theta=(\theta',\theta'')$, and $f_\theta$ depends on $\theta'$ and $g_\theta$ depends on $\theta''$, that is, the parameters determining $f_\theta$ and $g_\theta$ are unrelated.} This structure, as well as the regularization introduced in the later analysis, could be viewed as the inductive bias.

We will consider how the functions $f^t$ and $g$ can be approximated by operator recurrent neural network with appropriately chosen weights. We will also analyze the case when the data are contaminated with noise, $\mathcal{E}$ say, such that $X + \mathcal{E}$ no longer belongs to $\operatorname{Ran}(F)$. {Our goal is to use to use recurrent operator neural networks to find a trainable solution algorithm for an inverse problem so that the architecture is informed by the PDE-based solution methods but in which the measurement noise can be take into account in the training. Moreover, we will show that optimal (general) operator recurrent network under the expected loss can be identified as a Bayes estimator.

\begin{remark}
In the above, the direct map $F$ is an approximation of a map $\mathcal F$ that maps 
between infinite-dimensional Banach spaces and the map $H$ is an approximative inverse of the map $\mathcal F$. In practice, $F$ can be obtained using a numerical discretization, such as the finite element method, to approximate solutions of partial differential equations. When the discretization of the model is taken in to account, the sequence \eqref{gfF sequence 2} needs to be replaced by the sequence
\beq \label{gfF sequence 2 continuous}
 & &g(\boldsymbol{f}(F(z))=I_{app}(z),\quad \hbox{for }z\in B^m(\rho_0),
\eeq
where $\|I_{app}(z)-z\|\leq \epsilon_0$.
However, in this paper we assume that the finite-dimensional approximation of function $F$ is so precise that the 
approximation error $\epsilon_0$ is negligible, and assume that the identity
\eqref{gfF sequence 2} is valid.
\end{remark}
}}

\subsection{Related work}

There has been a substantial amount of progress concerning applying
machine learning techniques to linear or linearized inverse problems,
particularly in the domain of natural image processing.  However,
nonlinear hyperbolic inverse problems are an entirely different class
of problems, see e.g. \cite{caday1, KKL, KLU, Lassas, U} and
references therein. A closely related recent work is
\cite{gilton2019}, in which a neural network is trained as an additive
term to regularize each iteration of a truncated Neumann series as a
way to solve a linear reconstruction task. Our paper also uses
truncated Neumann series as an approximation to the holomorphic
operator function, but the introduced deep learning architecture is
directly adapted from the Neumann series structure rather than
regularizing it. There have been other prior works in the area of
nonlinear inverse problems. In \cite{jin2017}, a deep neural network
is constructed mimicking the structure of the filtered back projection
algorithm for computerized tomography. In \cite{li2018}, neural
networks are used for learning a nonlinear regularization term, also
in the context of tomography. Deep neural networks have further been
employed for inverse scattering problems, such as in
\cite{li2018invscat, khoo2018, wei2019} and other related inverse
problems in \cite{Antholzer, Arridge2019, Arridge1, Bubba, jin2017,
  Lunz}.

{\maartenbluefont Unrolled deep neural network architectures were first
  used to solve optimization problems \cite{gregor2010}, in
  particular, the iterative shrinkage algorithm (ISTA)
  \cite{Daubechies2004}; for a recent review, see \cite{monga2021}. Unrolling is a way to select a domain
  specific architecture for deep neural networks that approximates an
  operator given implicitly by an iterative scheme \cite[Sections~
    4.9.1 and 4.9.4]{Arridge1}. Usage of such architectures for
  solving inverse problems was outlined in \cite{adler2017} and
  \cite{putzky2017},
  while further developments came in \cite{adler2018}.
Our work has some similarities to unrolling, as we take an existing
iterative algorithm and use it as the basis for developing a deep
learning strategy.}

A crucial feature of our approach is that properties derived from the
mathematical analysis provide insight as to how to efficiently and
sparsely parametrize the neural network that learns the inverse map.
Such sparsity bounds are important because fully general neural
network models are heavily overparametrized, making them both
difficult to analyze as well as computationally resource intensive.
Reducing the parameter space as a way to improve learning also has
connections to nascent information-theoretic formulations of deep
learning, such as through the information bottleneck method
\cite{tishby2015}. There is a wide array of existing literature on
studying sparsity in neural networks. One popular technique to achieve
sparsity is to take a pre-trained dense network and prune parameters
with low importance; an early example of this technique is
\cite{lecun1990}, with later examples studying pruning including
\cite{frankle2019, han2015, liu2019}. However, it is desirable to
achieve sparsity without needing to first train a dense network.
{\maartenbluefont Indeed, in our work, sparsity bounds are directly
  imposed for the network parameters that encode the linear
  transformations across layers.} Studies of sparsity promotion either
before or during network training include \cite{bellec2018, Bruna,
  mocanu2018, mostafa2019}. {\maartenbluefont As will be seen later,
  sparsity in the network parameters has an interpretation in terms of
  low-rank approximation of the compact operators appearing in the
  original iterative scheme that is unrolled.} The use of low-rank
weight matrices in deep learning has become popular for a variety of
applications; see for example \cite{jaderberg2014, lebedev2014,
  xjli2017, yaguchi2019}. However, these works all exploit low-rank
structure that is empirically found rather than mathematically
derived. {\maartenbluefont Our sparsity bounds also provide improved
  generalization bound. This is independent of (regularization)
  techniques employed to improve upon training
 \cite{kukacka2017}.}

\section{Principled architecture}
\label{sec.DNN}

{ In this section, we focus on the network architecture for $f_{\theta}$ while suppressing the parameter $t$.} This network represents $f$, which is the main component of the inverse map, $H$. The design is domain adapted in the sense that it utilizes structure that the inverse map may possess. We exemplify this with the inverse boundary value problem associated with the wave equation and boundary control in the final section of this paper; however, we expect the architecture to adapt equally well to electrical impedance tomography and the $\bar{\partial}$ method.

\subsection{Operator recurrent architecture}

We define a specialized neural network architecture, the \emph{operator recurrent network}, that we propose as a suitable architecture for learning certain classes of nonlinear functions whose inputs are linear operators and whose outputs are functions. As mentioned in the introduction, we invoke a discretization turning operators into matrices and functions into vectors.

\subsubsection{Standard deep neural network}

To draw a comparison with the operator recurrent architecture we will
introduce shortly, we first define the standard neural network. This
is a function $f_{\theta}: \R^{d_0} \rightarrow \R^{d_L}$ with depth
$L$ and set of weights $\theta$ defined by
\begin{align}
\label{eqn.stdNNdef1}
   f_{\theta}(x) & = h_L ,
\\
\label{eqn.stdNNdef2}
   h_{\ell} & = A^{\ell,0}_{\theta} h_{\ell-1}
     + \phi_{\ell}\left[ b^{\ell}_{\theta}
                 + A^{\ell,1}_{\theta} h_{\ell-1} \right] ,
\\
\label{eqn.stdNNdef3}
   h_0 & = x .
\end{align}
The index $\ell = 0,\ldots,L$ indicates the layer of the neural
network. Each vector $h_{\ell} \in \R^{d_{\ell}}$ is the output of
layer $\ell$, where $d_{\ell}$ is the width of that layer. For each
layer $\ell$, the functions $\phi_{\ell}: \R^{d_{\ell}} \rightarrow
\R^{d_{\ell}}$ are the activation functions, which apply a scalar
function to each component, that is, for $x=(x_j)_{j=1}^{d_{\ell}}\in \R^{d_{\ell}}$,
$\phi_{\ell}(x)=(\phi_{\ell}(x_j))_{j=1}^{d_{\ell}}\in \R^{d_{\ell}}$.

The matrices $A^{\ell,0}_{\theta} \in \R^{d_{\ell} \times
  d_{\ell-1}}$, which typically have an identity matrix as a sub-block,
encode skip connections by passing outputs from layer $\ell-1$ to
layer $\ell$ without being operated on by any activation
functions. The $\R^{d_{\ell}}$-vectors $b_{\theta}^{\ell}$ are the
bias vectors and the $d_{\ell} \times d_{\ell-1}$ matrices
$A^{\ell,1}_{\theta}$ are the weight matrices.  Each of
$b^{\ell}_{\theta}, A^{\ell,0}_{\theta}, A^{\ell,1}_{\theta}$ are
dependent (in a context-specific way) on parameters $\theta$ to be
learned. For example, in the case of convolutional neural networks,
$A^{\ell,1}_{\theta}$ is a block-sparse matrix whose blocks are
Toeplitz matrices, and the parameters $\theta$ determine the values of
the diagonals and off-diagonals of these blocks.

\subsubsection{Operator recurrent network}

While standard neural networks have enjoyed widespread success in
many applications, they are not efficient at approximating functions
that are mathematically known to have a multiplicative
and highly nonlinear structure.
This is because a standard neural network with rectifier activations
is a form of a multivariate linear spline.
For example, approximating even a univariate polynomial to high accuracy
requires a fairly deep neural network \cite{yarotsky2016}.
In nonlinear inverse problems, the situation is even more problematic,
since their structure includes operator polynomials where the
polynomial is of high degree and the operator is discretized
as a large matrix. This situation motivates our new construction.

An \emph{operator recurrent} network has an internal structure
reflecting the linear operator nature of the input by performing
matrix-matrix multiplications, rather than vectorizing the input and
then performing matrix-vector multiplications. To this end, 
{we consider following neural networks.}

\begin{definition}
\label{defn.basicORN}
A \emph{basic operator recurrent network} with depth $L$, width $n$,
and set of weights (or parameters) $\theta$ is defined as a function
$f_{\theta} : \R^{n \times n} \rightarrow {\mtext \R^{n}}$ given by
\begin{align}
\label{eqn.NN1def1}
    f_{\theta}(\Lambdamap) & = h_L,
\\
\label{eqn.NN1def2}
    h_{\ell} & = b_{\theta}^{\ell,0} + A^{\ell,0}_{\theta} h_{\ell-1}
        + B^{\ell,0}_{\theta}\, \Lambdamap \,h_{\ell-1} + 
    \phi_{\ell} \left[ b^{\ell,1}_{\theta}
                       + A^{\ell,1}_{\theta} h_{\ell-1}
        + B^{\ell,1}_{\theta} \, \Lambdamap \,h_{\ell-1} \right],
\end{align}
where $h_0 \in {\mtext \R^{n}}$ is an initial vector not explicitly
given by the data, the quantities
$b^{\ell,0}_{\theta},~b^{\ell,1}_{\theta} \in {\mtext \R^{n}}$ and
$A^{\ell,0}_{\theta}, A^{\ell,1}_{\theta}, B^{\ell,0}_{\theta},
B^{\ell,1}_{\theta} \in \R^{n \times n}$ are dependent on the
parameters $\theta$, and the $\phi_{\ell}$ are the activation
functions.
\end{definition}

We note that $h_{\ell}$ should be viewed as a \emph{hidden state}. The typical initialization of the hidden state is $h_0 = 0$,
though it could be learned as well. This naturally applies in the context of inverse problems; in Section~\ref{sec.IP} the hidden states take the role of boundary controls.

\begin{remark}
{\it We may consider $h_0$ not as part of the initial layer, but instead as
the output of an initial layer whose value is entirely determined by a
bias vector $b_{\theta}^{0,0}$ set to be equal to $h_0$, with all
other terms set to zero.}
\end{remark}

\begin{remark}
The data matrix, $\Lambdamap$, is a digitized counterpart of an
operator. In Section~\ref{sec.IP} we realize this as the outcome of
numerical discretization. However, the digitization may be obtained
through composition with a data acquisition operator, which may be
viewed as a pre-processing operator that can be learned. Learning a
data acquisition scheme has been considered in different contexts 
\cite{Baldassarre-2016, Dardikman-Yoffe-2020, Luijten-2019, Sanchez-2020}.
\end{remark}

\subsubsection{Activation function}

In general, the activation functions $\phi_{\ell}$ may differ at each
layer $\ell$. We choose the form of $\phi_{\ell} :\ \R^n \rightarrow
\R^n$ to be a rectifier (or ReLU). That is, $\phi_{\ell}$ is given by
\begin{equation} \label{eqn.leakyrelu}
   (\phi_{\ell}({\fmtext y}))_j
     = {\fmtext \phi_{\eta}({\fmtext y}_j)=}
     \max({\fmtext y}_j,\eta {\fmtext y}_j), \quad j = 1,\ldots,n ,
\end{equation}
where $0 \le \eta \le 1$ is either a hyperparameter that is chosen in
advance (the ``leaky'' ReLU) or could be a parameter that is learned
during optimization (the ``parametric'' ReLU). In either case, this
choice of activation function is a piecewise-linear function on each
vector component.

The choice of the rectifier as the activation function has both
pragmatic and mathematical reasons. Indeed, in the case of standard
deep neural networks with $\eta = 0$, there is significant empirical
evidence indicating that the use of the rectifier activation function
promotes sparsity and accelerates training \cite{glorot2011,maas2013}.
Rectifier networks are also closely connected with piecewise-linear
splines, which are known to interpolate data points while minimizing
the second-order total variation \cite{unser2017, unser2018}. In
Section~\ref{ssec.piecewisepoly}, we will show that in our case such
activations induce piecewise (operator) polynomial behavior.

We note that a network of the form
\eqref{eqn.NN1def1}-\eqref{eqn.NN1def2} with activation functions
being rectifiers with leaky parameter $\eta > 0$ can have its
activation functions replaced, without loss of generality, by standard
rectifier activation functions ($\eta = 0$). We let $\phi_{\eta}$ be the
activation function in \eqref{eqn.leakyrelu}. Then we can write
\begin{equation}
    \phi_{\eta} = \eta \mathrm{Id} + (1 - \eta) \phi_0,
\end{equation}
where $\mathrm{Id}$ is the identity map and {\mtext the activation function $\phi_0$ is the standard rectifed linear unit (relu)}. Then, starting with \eqref{eqn.NN1def2}, we
have
\begin{align}
   h_{\ell} & = b_{\theta}^{\ell,0} + A^{\ell,0}_{\theta} h_{\ell-1}
      + B^{\ell,0}_{\theta}\, \Lambdamap \,h_{\ell-1} + 
        \phi_{\eta} \left[ b^{\ell,1}_{\theta}
      + A^{\ell,1}_{\theta} h_{\ell-1}
      + B^{\ell,1}_{\theta}\, \Lambdamap \,h_{\ell-1} \right]
\nonumber\\
   & = (b_{\theta}^{\ell,0} + \eta b_{\theta}^{\ell,1})
      + (A^{\ell,0}_{\theta} + \eta A^{\ell,1}_{\theta}) h_{\ell-1}
      + (B^{\ell,0}_{\theta}
                     + \eta B^{\ell,1}_{\theta})\, \Lambdamap\, h_{\ell-1}
\nonumber\\
   &\quad + (1 - \eta) \phi_0 \left[b^{\ell,1}_{\theta}
      + A^{\ell,1}_{\theta} h_{\ell-1}
      + B^{\ell,1}_{\theta}\, \Lambdamap \,h_{\ell-1} \right],
\label{eq:ORNNl}
\end{align}
and thus an operator recurrent network with $\eta > 0$ can be replaced
by another one with $\eta = 0$ by relabeling some of the biases and weights.

\subsubsection{Recurrence}

By inspecting \eqref{eqn.NN1def1}--\eqref{eqn.NN1def2}, we observe
that the input data $\Lambdamap$ is inserted multiplicatively into the
network at every layer, so that each computed intermediate output
$h_{\ell}$ depends both on $\Lambdamap$ and previous intermediate
outputs $h_{\ell-1}, h_{\ell-2}, \ldots$ in an identical fashion for
each $\ell$. In the finite-dimensional setting, such expansions can be
viewed as matrix polynomials, and each layer can be thought of as
performing another stage of an iteration in which the degree of the
polynomial is raised through multiplication by the matrix
variable. Thus the neural network learns nonlinear perturbations of
this process at each iteration.

There may be a reason to expect that every iteration not only has the same structure, but is in fact identical. For example, this holds true for nonlinear operator functions given by a truncated Neumann series. Thus, operator recurrent networks can also be interpreted as the unrolling of an iterative nonlinear process, where the recurrence refers to the fact that the output of each layer is fed back into another layer that may have the same weights.

\begin{figure}
\hspace*{1.0cm}
\includegraphics[width = 0.65\textwidth, angle=0]{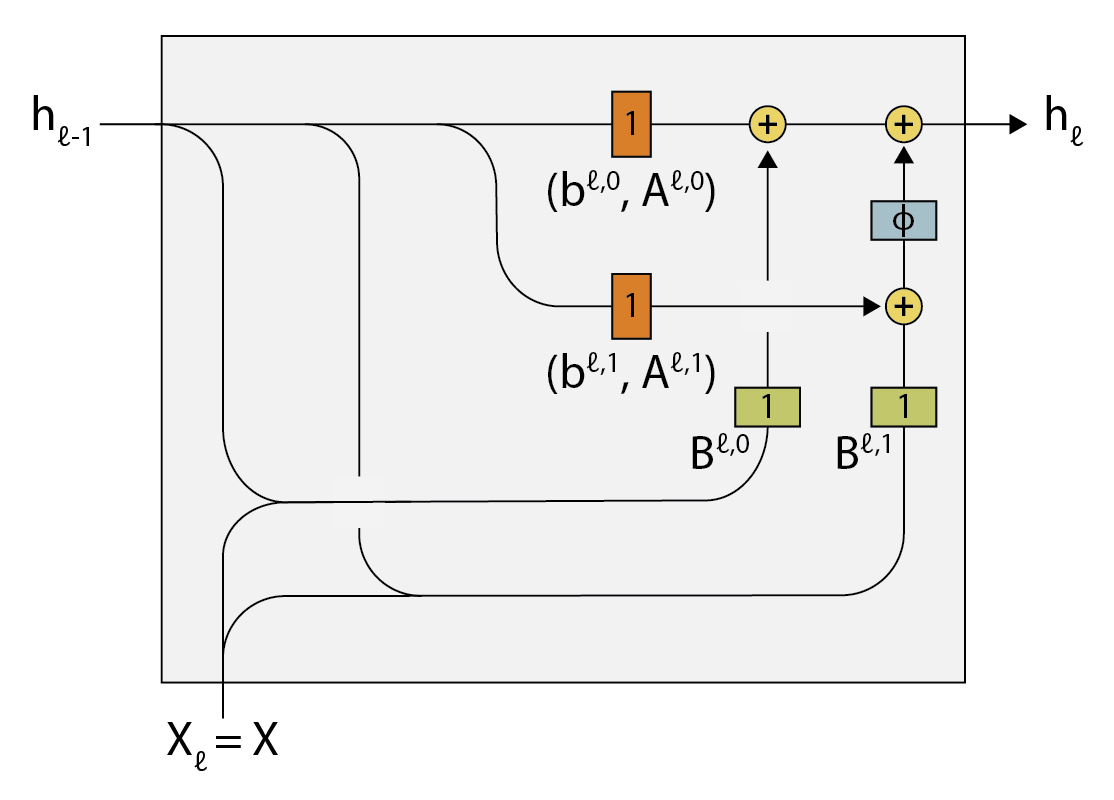} \\[-4.5cm]
\hspace*{1.0cm}
\includegraphics[width = 0.85\textwidth, angle=0]{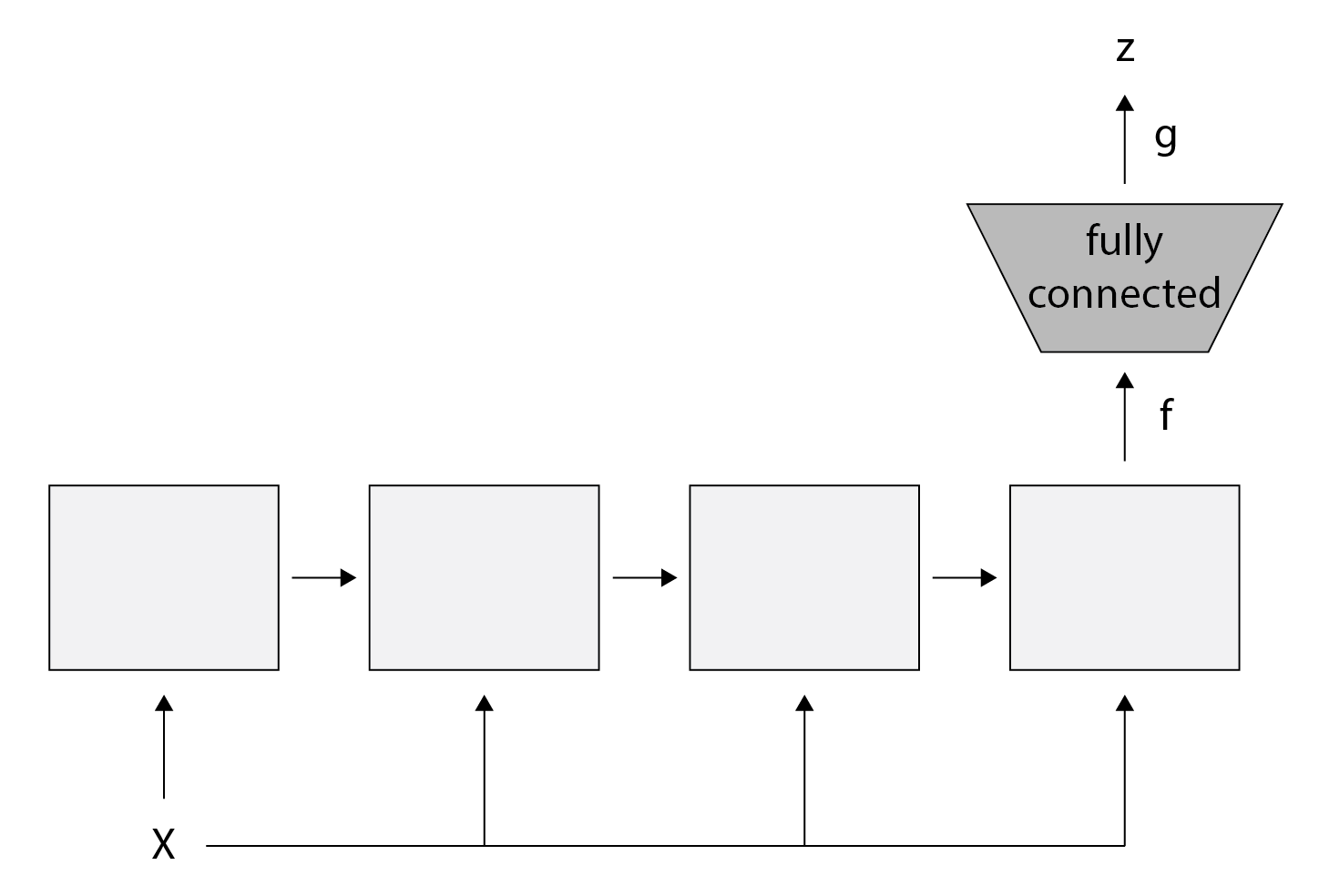}
\caption{Cell (top) of the operator recurrent network (bottom) architecture. When concatenated with a feed-forward network consisting of a few fully connected layers, the network adapts to inverse problems. The data operator $\Lambdamap$ is inserted multiplicatively into the network at each cell. The initial hidden state $h_0$ is typically chosen to be zero.}
\label{fig.DNN}
\end{figure}

\subsection{General operator recurrent networks}

A general operator recurrent network is obtained from a basic operator recurrent network merely by adding memory.

\medskip

\begin{definition}
A \emph{general operator recurrent network} { of level $K$} is an extension of the
basic operator recurrent network, including terms that contain
$h_{\ell-k}$ in the expression for $h_{\ell}$, that is,
\begin{eqnarray} \label{eqn.NN1def1 exte3dnde}
   f_{\theta}(\Lambdamap) &=& h_L ,
\\
\label{eqn.NN1def2 exte3dnde}
   h_{\ell} &=& b_{\theta}^{\ell,0} +
\sum_{k=1,...,K;i=0}(
     A^{\ell,k,i}_{\theta} h_{\ell-k}
        + B^{\ell,k,i}_{\theta}\, \Lambdamap\, h_{\ell-k}) 
\\
   & &\hspace*{2.5cm}
   +  \phi_{\ell} \left[ b^{\ell,1}_{\theta}+ \sum_{k=1,...,K;i=1}
                       (A^{\ell,k,i}_{\theta} h_{\ell-k}
        + B^{\ell,k,i}_{\theta}\, \Lambdamap \,h_{\ell-k}) \right],
\nonumber
\end{eqnarray}
for $\ell \ge 1$, where $h_0 \in {\mtext \R^{n}}$ is some initial
vector not explicitly given by the data, that is, { the initial hidden state}, $h_{-k}=0$ for $-k<0$, and the quantities $b^{\ell,0}_{\theta},~b^{\ell,1}_{\theta} \in {\mtext \R^{n}}$ and $A^{\ell,k,i}_{\theta}, 
B^{\ell,k,i}_{\theta} \in \R^{n \times n}$ are dependent on the parameters $\theta$, and the $\phi_{\ell}$ are the activation functions.
\end{definition}

\medskip

\noindent
Basic and general operator recurrent networks can be further generalized upon replacing vectors $h_{\ell}$ and biases in $\R^n$ by sets of $r$ vectors, that is, matrices in $\R^{n \times r}$. We will not consider this in the analysis.

In the general operator recurrent network \eqref{eqn.NN1def1 exte3dnde}-\eqref{eqn.NN1def2 exte3dnde}, the dependency of $h_\ell$ on previous
outputs $h_{\ell-m}$ for $m > 1$ is an explicit way to encode skip connections, which feature prominently in applications of standard
neural networks \cite{ronneberger2015,he2016}. In standard neural networks, however, similar generalizations are fully included in the
basic definition since they can be implemented by increasing the width of the network. However, in operator recurrent networks, the width is
fixed and so this generalization must be explicitly included. In the following discussions, however, the basic definition \eqref{eqn.NN1def1}-\eqref{eqn.NN1def2} is sufficient as discussed in the example below.

A general operator recurrent network can be written as a basic operator recurrent network by extending the width of the network. We show this
explicitly starting from \eqref{eqn.NN1def1 exte3dnde}-\eqref{eqn.NN1def2 exte3dnde}. Let $\tilde h_\ell=(h_{\ell},\dots,h_{\ell-K-1})^T\in \R^{nK}$ where $h_{-i}=0$ for $i>0$. Also, let
\beq
  \tilde A^{\ell,i}_{\theta}  =\left(\begin{array}{ccccc}   A^{\ell-1,1,i}_{\theta} & A^{\ell-1,2,i}_{\theta}&\dots 
  & A^{\ell-1,K-1,i}_{\theta}&A^{\ell-1,K,i}_{\theta} \\
  I  & 0& \dots&0 &0 \\
  0  & I& \dots&0 &0 \\
  \vdots &\vdots & &\vdots &\vdots\\
  0  & 0& \dots  &I&0 \\
  \end{array}\right)
\eeq
and
\beq
  \tilde B^{\ell,i}_{\theta}  =\left(\begin{array}{ccccc}   B^{\ell-1,1,i}_{\theta} & B^{\ell-1,2,i}_{\theta}&\dots 
  &B^{\ell-1,K-1,i}_{\theta}&B^{\ell-1,K,i}_{\theta} \\
  I  & 0& \dots&0 &0 \\
  0  & I& \dots&0 &0 \\
\vdots &\vdots & &\vdots &\vdots\\
0  & 0& \dots  &I&0 \\
\end{array}\right),
\eeq
for $i=1,2$ and $\tilde b_{\theta}^{\ell,i}=(b_{\theta}^{\ell-1,i},\dots,b_{\theta}^{\ell-K,i})^T\in \R^{nK}$, for $i=1,2$.
Also, let $\tilde \Lambdamap=\hbox{diag}( \Lambdamap,\dots, \Lambdamap)\in  \R^{nK\times nK}$. 
Then  the general operator recurrent network $f_{\theta}$ given in \eqref{eqn.NN1def1 exte3dnde}-\eqref{eqn.NN1def2 exte3dnde} can be written as a basic operator recurrent network $\tilde f_{\theta}: \R^{nK\times nK}\to  \R^{nK}$ given by
\begin{align} \label{eqn.NN1def1 B}
   \tilde f_{\theta}(\tilde\Lambdamap) & = \tilde h_{L},
\\
\label{eqn.NN1def2 B}
   \tilde  h_{\ell} & = \tilde b_{\theta}^{\ell,0} + \tilde A^{\ell,0}_{\theta}\tilde  h_{\ell-1}
        + \tilde B^{\ell,0}_{\theta} \tilde \Lambdamap\tilde h_{\ell-1} + 
    \phi_{\ell} \left[ \tilde b^{\ell,1}_{\theta}
                       + \tilde A^{\ell,1}_{\theta} \tilde h_{\ell-1}
        +\tilde B^{\ell,1}_{\theta} \tilde \Lambdamap \tilde h_{\ell-1} \right],
\end{align}
and setting $f_{\theta}(\Lambdamap)=\Pi_n(\tilde
f_{\theta}(\tilde\Lambdamap))$. Here, $\Pi_n:\R^{nK}\to \R^{n} $ is the
operator \newline $\Pi_n(y_1,y_2,\dots,y_{nK})= (y_1,y_2,\dots,y_{n})$. Also,
we observe that $\|\tilde \Lambdamap\|_{\R^{nK}\to \R^{nK}}=\|
\Lambdamap\|_{\R^{n}\to \R^{n}}$.

We can contrast this construction with the standard neural network
definitions \eqref{eqn.stdNNdef1}-\eqref{eqn.stdNNdef3}. In the
standard neural network, a vector $x$ is the input, and the
intermediate outputs $h_{\ell}$ at each layer $\ell$ are produced by
repeatedly applying matrix-vector products as well as activation
functions in some order. In contrast, in the operator recurrent
network, the input is a matrix $\Lambdamap$, and it is multiplied on both
the left and right by matrices. At the first layer, this is still
equivalent to a standard neural network, since the action of a matrix
on another matrix is linear. However, at all subsequent layers, this
is no longer equivalent, since the matrix $\Lambdamap$ is re-introduced
at each layer and is multiplied with the previous output $h_{\ell-1}$.
\medskip

\begin{remark} \label{rem: other networks}
{\it A standard ``additive'' neural network (cf.~(\ref{eqn.stdNNdef1})-\eqref{eqn.stdNNdef3}) with input $x \in \R^n$ can be written as a general operator recurrent network \eqref{eqn.NN1def1 exte3dnde}-\eqref{eqn.NN1def2 exte3dnde} as follows. We set
$\Lambdamap = \hbox{diag}(x_1, \dots, x_n)$, $h_0 = {\bf 1} = (1, 1, \ldots, 1)^T$ and let $K = 1$. For $\ell = 1$ we choose the weight matrices to be
$B^{1,1,0}_{\theta} = A^{1,0}_{\theta}$ and 
$B^{1,1,1}_{\theta} = A^{1,1}_{\theta}$ and $A^{1,1,i}_{\theta} = 0$, $i=0,1$;
for $2 \le \ell \le L$ we choose the weight matrices to be
$A^{\ell,1,0}_{\theta} = A^{\ell,0}_{\theta}$,
$A^{\ell,1,1}_{\theta} = A^{\ell,1}_{\theta}$ and $B^{\ell,1,i}_{\theta} = 0$, $i = 0,1$.}
\end{remark}

To simplify notation, in particular the indexing of variables, we consider mostly basic operator recurrent networks, that is, the case $K = 1$. However, the results can be straightforwardedly generalized to general operator recurrent networks. The general operator recurrent network will play a fundamental role in Theorem~\ref{not linear operator approximation theorem} only.

\subsection{Sparse representation of trained matrices}

Next, we specify how the biases and weights depend on the parameters
$\theta$. In a typical fully-connected layer for a standard neural
network, $\theta$ determines the entries of the biases and
weights. More precisely,
\begin{equation} \label{alternative 1}
    b_{\theta}^{\ell} = \theta^{\ell}_0, \quad 
    A^{\ell,1}_{\theta} = \begin{bmatrix} \theta^{\ell}_1 &
    \theta^{\ell}_2 & \ldots & \theta^{\ell}_{d_{\ell}}
                        \end{bmatrix} ,
\end{equation}
where
\begin{equation*}
   \theta = \left\{ \theta^{\ell}_p{\mtext \in \R^{d_{\ell+1}}} :
     \ell = 1, \ldots, L,
                       p = 0, \ldots, d_{\ell} \right\}
\end{equation*}
stands for a set of column vectors. We consider the parametrization of basic operator recurrent networks in terms of $\theta$. The matrices $A^{\ell,{}i}_{\theta}, B^{\ell,{}i}_{\theta}$ in \eqref{eqn.NN1def2} could depend on $\theta$ similarly to \eqref{alternative 1}. However, in our analysis, it is beneficial to provide an alternative quadratic dependence: For each $\ell$ and $i
= 0, 1$ there are {\mtext $4n$} column vectors $\theta^{\ell,{}i}_1, \ldots, \theta^{\ell,{}i}_{4n} \in \R^n$ within the parameter set $\theta$ such that for $i = 0, 1$,
\begin{equation} \label{eqn.Cmatrixdecomp}
 {\mtext  A^{\ell,{}i}_{\theta} =A^{\ell,{}i,(0)}+A^{\ell,{}i,(1)}_\theta,\quad
   A^{\ell,{}i,(1)}_\theta= \sum_{p=1}^n
           \theta^{\ell,{}i}_{{\qtext 2p-1}} (\theta^{\ell,{}i}_{2p})^T,}
\end{equation}
{\mtext and similarly for $B^{\ell,{}i}_{\theta}$, 
\begin{equation} \label{eqn.Cmatrixdecomp B}
{\mtext B^{\ell,{}i}_{\theta} = B^{\ell,{}i,(0)} + B^{\ell,{}i,(1)}_\theta,
        \quad
   B^{\ell,{}i,(1)}_\theta= \sum_{p=n+1}^{2n}
           \theta^{\ell,{}i}_{{\qtext 2p-1}} (\theta^{\ell,{}i}_{2p})^T.}
\end{equation}
Each $A^{\ell,{}i,(0)}$ and $B^{\ell,{}i,(0)}$ is} a fixed operator
that does not depend on parameter $\theta$ and is ``handcrafted''. The resulting deep neural network is illustrated in Figure~\ref{fig.DNN}. The { fixed} operators are typically the zero operator or the identity operator, but they can be also other operators that are chosen depending on the specific application. Examples of such operators suitable for solving the inverse problem for the wave equation are considered later in Section~\ref{sec.IP}, in particular the discussion below \eqref{alternatives for the fixed operators}.

\begin{remark}
{\it Following Remark~\ref{rem: other networks}, choosing for $2 \le \ell \le L$ the weight matrices to be $A_{\theta}^{\ell,1,i,(0)} = I$ and $B_{\theta}^{\ell,1,i} = 0$, $i = 0,1$, we obtain a residual network \cite{he2016}.}
\end{remark}

We now assume that the matrices $A^{\ell,{}i,(0)}$ and $B^{\ell,{}i,(0)}$  and the bias vectors satisfy
\beq \label{norm estimate for fixed}
   { \sum_{i=0}^1 (\|A^{\ell,{}i,(0)}\|_{\R^n\to\R^n} + \|B^{\ell,{}i,(0)}\|_{\R^n\to\R^n}
   + |b_{\theta}^{\ell,i}|)}
   \le  c_0 ,
\eeq
for some $c_0 \ge 1$. The lower bound, $1$, arises as we allow 
the relevant non-learned matrices to be identity matrices. This makes possible the ResNet-type architectures that contain layers $h\to \phi(h+A^{\ell,(1)}_\theta  h+B^{\ell,(1)}_\theta X h)$ or $h\to h+\phi(A^{\ell,(1)}_\theta  h+B^{\ell,(1)}_\theta X h)$.

We parametrize the bias vectors by $b_{\theta}^{\ell,i} = \theta^{\ell,1,i}_0 \in \R^n$, $i=0,1$. With these notations $f_\theta$ is determined by the set of parameters $\theta$ that is given as an ordered sequence
\begin{multline} \label{thetavector}
   \theta = [\theta^{\ell,{}i}_p\in \R^n :\
       \ell=1,2,\dots,L,\ p=1,2,\dots,4n,\  
         i=0,1]
\\
   \cup [\theta^{\ell,{}i}_0 \in \R^n: \ \ell=1,2,\dots,L,\ i=0,1].
\end{multline}
We denote the index set in the above sequence by 
\beq\label{index set P}
& & P=P_1\cup P_2,\\ \nonumber
& &P_1=
\{(\ell,{}i,p):\   \ell=1,2,\dots,L,\
 i=0,1,\ p=1,2,\dots,4n \},\\
& &P_2=  \nonumber
\{(\ell,{}i,p):\   \ell=1,2,\dots,L,\ 
i=0,1,\ p=0\}.
\eeq
We note that the indices in $P_1$ are related to the learnable weight matrices, $A_\theta^{\ell,{}i,(1)}$ and $B_\theta^{\ell,{}i,(1)}$ of the basic operator recurrent network, and the indices in $P_2$ are related to bias vectors. Below, we use the fact that $P_1$ has $\# P_1\leq 4n{}L$ elements, and $P_2$ has $\# P_2 \leq 2L$ elements. We note that $\# P_2$ is significantly smaller than $\# P_1$ and that $\# P_2$ is independent of $n$.

{ For the general recurrent operator  networks we add the index $k=1,\dots, K$ and 
replace the above parameters by 
the ordered sequences
\begin{multline} \label{generalthetavector}
\tilde    \theta = [\tilde  \theta^{\ell,i,k}_p\in \R^n :\
       \ell=1,2,\dots,L,\ p=1,2,\dots,4n,\  k=1,2,\dots,K,\
         i=0,1]
\\
   \cup [\tilde  \theta^{\ell,i,0}_0 \in \R^n: \ \ell=1,2,\dots,L,\ i=0,1].
\end{multline}
Also, for the general recurrent operator networks 
we denote the index set in the above sequence by 
\beq\label{generalindex set P}
& & \tilde P=\tilde  P_1\cup \tilde  P_2,\\ \nonumber
& &\tilde  P_1=
\{(\ell,i,p,k):\   \ell=1,2,\dots,L,\
 i=0,1,\ p=1,2,\dots,4n,\ k=1,2,\dots,K\},\\
& &\tilde P_2=  \nonumber
\{(\ell,i,p,k):\   \ell=1,2,\dots,L,\ 
i=0,1,\ p=0,\ k=0\}.
\eeq

Next, we return to considering the basic recurrent operator networks.} From a (numerical) linear algebra viewpoint, the decomposition
\eqref{eqn.Cmatrixdecomp} expresses the matrix
$A_{\theta}^{\ell,{}i,(1)}$ as a sum of rank $1$ matrices, similar to
a singular value decomposition. This structure is valuable for our
analysis, since it means that we essentially learn a factorization of
these matrices rather than the explicit matrix elements. We will exploit that in Subsection~\ref{ssec:convexreg} while introducing low-rank structures.

\begin{remark}
{\it In the above, the parameters in each layer are allowed to be different
and independent. However, it is natural to consider the case that a
subset of parameters is shared across layers. We will analyze the
impact of shared weights in various estimates below.}
\end{remark}

\subsection{Approximation properties}

\subsubsection*{Estimates for nonlinear operator functions in the holomorphic calculus}

Here, we establish the approximation power of operator recurrent networks, within a certain space of general nonlinear operator functions.  We begin by studying the approximation of functions mapping the linear operator $\Lambdamap: \R^n \to \R^n$ to another linear operator $q(\Lambdamap) :\ \R^n \to \R^n$. This map $q$ is holomorphic in $\Lambdamap$  { and it is defined by the fundamental formula of holomorphic operator calculus \textup{\cite{rudin1991}},
\beq 
q(\Lambdamap)=\frac 1{2\pi i}\int_\gamma q(z)\,(\Lambdamap-z)^{-1}dz,
\eeq
where $\gamma \subset \mathbb C$ is a circle having radius larger than the norm of $\Lambdamap$, oriented in the positive direction.} We contract the operator $q(\Lambdamap)$ with a vector $v \in \bar B^n(1)$, where $\bar B^n(R)$ is the closed ball of radius $R > 0$. In the context of inverse problems and reconstruction, $q(\Lambdamap)$ is often polynomial. To emphasize this context, we write $f(\Lambdamap)$ for $q(\Lambdamap) v$. {An example of a neural network based on holomorphic operator calculus is considered in Section \ref{sec: Example of matrix inversion}.}
We then consider a map, $H$, obtained from the composition with a nonlinear, smooth function $g :\ \R^n \to \R^m$.

Below, we use the norms
\beq
   \|g\|_{\C^k(\bar B^n(r);\R^m)}=
   \max_{y \in \bar B^n(r)}\max_{|\alpha|\leq k}\| D^\alpha g (y)\|_{\R^m}, \quad
   \|f\|_{C({\fmtext\mathcal B^{n\times n}};\R^m)}=
   \max_{\Lambdamap \in \mathcal B^{n\times n}}\|f(\Lambdamap)\|_{\R^m} ,
\eeq
where $ D^\alpha g (y)=(\frac {\p}{\p y_1})^{\alpha_2}(\frac {\p}{\p y_2})^{\alpha_2} \dots (\frac {\p}{\p y_n})^{\alpha_n})g(y)$, $\alpha = (\alpha_1,\alpha_2,\dots,\alpha_n) \in \mathbb N^n$  and $|\alpha| = \alpha_1+\alpha_2 + \dots + \alpha_n$. We denote the linear operator norm by $\|\Lambdamap\|_L = \|\Lambdamap\|_{\mtext \R^n \to \R^n}$ and recall that $\mathcal B^{n\times n} = \{\Lambdamap\in \R^{n\times
  n}:\ \|\Lambdamap\|_{\mtext \R^n \to \R^n}\le 1\}$ is the closed unit ball in the set of matrices.
  
\begin{theorem} \label{thm.apprx1}
Consider a nonlinear operator function $H :\ \mathbb R^{n \times n}
\to \R^m$, defined by
\begin{equation} \label{eqn.smooth+hol}
   H :\ \Lambdamap \mapsto g({ v}^\top q(\Lambdamap) ) ,
\end{equation}
where $q$ is obtained using the holomorphic operator calculus, $v \in \bar B^n(1)$
and $g : \R^n \to \R^m$, satisfying
\begin{itemize}
\item
$q$ is a holomorphic function whose domain contains a complex
  disk ${\fmtext\mathbb D}_{1+r_1}$ having radius $1 + r_1 > 1 + r$
  centered at the origin, for some $r \in (0,1)$,
\item
$g \in C^k(\bar{B}^n(2);\R^m)$ for some $k \geq 1$, and
\item
$q$ and $g$ are both bounded by $1$.
\end{itemize}
Let $\varepsilon \in (0,1)$. Then there exists a {general} operator
recurrent network, $H_{\theta}$, which depth $L \le L_0$, level $K = 2$, and width $W
\le W_0$, and constant $C = C(k,n,r)$ such that
\begin{equation}
   \| H - H_{\theta} \|_{C({\ktext
      {\fmtext\mathcal B^{n \times n}}};\R^m)} \le \varepsilon
\end{equation}
with
\begin{equation} \label{eqn.Ldepth}
   L_0 = C \bigg( \log \bigg(
   \frac{4 \|g\|_{C^1({\ktext \bar B^n{\ptext (2)}}) }}{r \varepsilon}
           \bigg)
   + \log \bigg( \frac{{\qtext 4}^{k+1}
     \|g\|_{ C^k({\ktext \bar B^n{\ptext {\ptext (2)}}})}}{\varepsilon}
     \bigg) + 1 \bigg)
\end{equation}
and
\begin{equation} \label{eqn.Wwidth}
   W_0 = C {\ftext m}n \bigg(\frac {\varepsilon}{{\qtext 4}^{k+1}
     \|g\|_{ C^k({\ktext \bar B^n{\ptext (2)}})}} \bigg)^{-n/k} \bigg(
   \log \bigg( \frac {{\qtext 4}^{k+1}
     \|g\|_{ C^k({\ktext \bar B^n{\ptext (2)}})}}{\varepsilon} \bigg)
               + 1 \bigg) .
\end{equation}
\end{theorem}

\begin{proof}
{In the proof we will first estimate how to approximate a holomorphic function of an operator by a polynomial and
represent the obtained polynomial as a general operator recurrent network. After this we adapt Yarotsky's results on quantified approximation of a function pointwise by a deep neural network and represent the obtained network as a recurrent operator network.

To prove the claim, we first}
approximate
$q(\Lambdamap)$ locally by a polynomial $P(\Lambdamap)$.  As $q$ is
holomorphic on {\ytext some disc ${\fmtext\mathbb D}_{1+r_1}$, where
  $r_1 > r > 0$} and bounded by $1$, its derivatives at zero satisfy
\begin{equation}
   q^{({j})}(0) = \frac{{j}!}{2\pi i}
               \int_{|z|=1+r} \frac {q(z)}{z^{{j}+1}}dz
\end{equation}
and, hence, {\maartenbluefont using that $\| q \| \le 1$}, its Taylor
coefficients at zero satisfy
\begin{equation}\label{eq ak}
   a_{j} = \frac 1{{j}!}q^{({j})}(0) ,\quad
         |a_{j}| \leq \frac{1}{(1+r)^{{j}+1}}.
\end{equation}
Thus we have the Taylor polynomial
\begin{equation}\label{P Taylor}
   P(z) = \sum_{{j}=0}^\ell a_{j} z^{j},
\end{equation}
which satisfies for $|z| \le 1$
\begin{equation}
   |q(z) - P(z)| \leq \sum_{{p}=\ell+1}^\infty \frac{1}{(1+r)^{{p}+1}}
                 \leq \frac {(1+r)^{\ktext -\ell-1}}{r}.
\end{equation}
Hence, if $q(\Lambdamap)$ is defined using the holomorphic functional
calculus, then it can be approximated by the matrix polynomial
$P(\Lambdamap)$, with
\begin{equation}
\label{eqn.thmapprx1-polyapprx}
   \|{q}(\Lambdamap) - P(\Lambdamap)\|_{\C^n \to \C^n }
              \leq \frac {(1+r)^{-\ell}}{r} = \varepsilon_0.
\end{equation}
{\ktext Given $\varepsilon_0<r$, we   choose  $\ell$ to be 
\begin{equation} \label{net 2 i}
   \ell = 1 + \left\lfloor \frac{\log ( ({r}{\varepsilon_0})^{-1})}{\log (1+r)} \right\rfloor ,
\end{equation}
where $\lfloor\ \cdot\ \rfloor$  is the integer part of the real number $s$.
From the discussion in the previous section, it is thus possible to
exactly represent the map
\begin{equation}
   \Lambdamap \mapsto { v^\top P(\Lambdamap)=}
      P(\Lambdamap) v,
\end{equation}
see {\color{black} \eqref{P Taylor}, that approximates
the map $\Lambdamap \mapsto  v^\top q(\Lambdamap)$,} using a general operator recurrent network 
\ba
& &P(\Lambdamap) v=h_{2\ell+3},\\
& &h_0=0,\\
& &h_{1}=v,\\
& &h_{2j}=a_{j-1}h_{2j-1}+h_{2j-2}.\quad j=1,2,\dots,\ell+1\\
& &h_{2j+1}=Xh_{2j-1},\quad j=1,2,\dots,\ell+1
\ea
(so that $h_{2j+1}=X^jv$ and $h_{2j+1}=\sum_{{p}=0}^{j-1} a_{p} X^{p}v$)
of depth $2\ell+3$ and level 2, and whose hidden states are
vectors in $\R^n$ and weight matrices are $n \times n$.}

Next, we consider the network approximation of $g$. First, we note that in the exact nonlinear function $f$, the function $g$ takes in a vector $q(\Lambdamap){ v}$ whose norm is bounded by $1$, since $|q(z)| \le 1$ on the closed disk {\ktext of radius $1+r$}, and $\|{ v}\| \le 1$. Thus we are in the setting of approximating a nonlinear function ${\ftext g:}\R^n \rightarrow \R^n$ uniformly by neural networks on a bounded domain. By Remark~\ref{rem: other networks}, the standard neural network (\ref{eqn.stdNNdef1})-\eqref{eqn.stdNNdef3} can be written as a general operator recurrent network \eqref{eqn.NN1def1 exte3dnde}-\eqref{eqn.NN1def2 exte3dnde}, and thus to consider the approximation of $g$  we can use the results for standard neural networks. Such approximation problems have been studied in a wide variety of settings. Here, we use the results of Yarotsky \cite{yarotsky2016} applied to the function
\begin{equation}
  g_1(y) = \frac{g(4y)}{4^k \|g\|_{C^k({\ktext \bar B^n(2)})}}.
\end{equation}
This normalization is
such that $\|g_1\|_{C^k({\ptext \bar B^n(\ytext 1/2}))} \le 1$,
where the domain of $g_1$ is a ball in $\R^n$ of {\ytext radius $1/2$}.
With this normalization, then by Theorem~1
of \cite{yarotsky2016}, there exists a constant $C =
C(k,n)$ such that a standard {\ftext additive} neural network $G$ exists, satisfying
\begin{equation}\label{first approximation}
   \|g - G\|_{L^\infty({\ptext \bar B^n(2}))} \leq \varepsilon_1,
\end{equation}
and the depth $L'$ and width $W'$ of $G$ satisfy
\begin{eqnarray}\label{net 1} 
& &L' \leq C\bigg( \log \bigg(\frac {{\qtext 4}^{k}
           \|g\|_{ C^k({\ktext \bar B^n{\ptext (2)}})}}{\varepsilon_1} \bigg) + 1 \bigg),
\\ \label{net 2}
& &W' \leq C {\ftext m} n \bigg(\frac {\varepsilon_1}{{\qtext 4}^{k}
           \|g\|_{ C^k({\ktext \bar B^n{\ptext (2)}})}}\bigg)^{-n/k} 
    \log \bigg( \frac {{\qtext 4}^{k}\|g\|_{ C^k({\ktext \bar B^n{\ptext (2)}})}}{\varepsilon_1}
           + 1 \bigg).
\end{eqnarray}

Concatenating the previous two networks, we can construct an operator
recurrent network $f_{\theta}(\Lambdamap)={G(v^\top P(\Lambdamap)) =} G(P(\Lambdamap)
v )$. {\maartenbluefont By abuse of earlier notation we absorbed the
  weights of $G$ in $\theta$.} We then prove our main estimate:
\begin{eqnarray}
 \|g({q}(\Lambdamap){ v})-G(P(\Lambdamap){ v})\|_{\R^n}
&\le& \|g({q}(\Lambdamap){ v})-g(P(\Lambdamap){ v})\|
                + \|g(P(\Lambdamap){ v})-G(P(\Lambdamap){ v})\|
\nonumber\\
&\le& \|g\|_{C^1}\|({q}(\Lambdamap)-P(\Lambdamap)){ v}\| + \|g-G\|_{\ktext C^0(\bar B^n(1+r))}
\nonumber\\
&\le&  \|{g}\|_{C^1} \varepsilon_0 + \varepsilon_1.
\label{A estimate}
\end{eqnarray}
We choose $\varepsilon/2 = \|{g}\|_{C^1}\varepsilon_0 =
\varepsilon_1$. Then we set
\begin{equation}
   \ell = \frac{ \log(2 \|g\|_{C^1({\ktext \bar B^n{\ptext (2)}})} / (r \varepsilon))}{ \log(1 + r)},
\end{equation}
and redefine $C$ to include dependencies on $r$, to find the full
depth bound for the network
\begin{equation}
   L \le C \bigg( \log \bigg(\frac{4 \|g\|_{C^1(({\ktext \bar B^n{\ptext (2)}}))}}{r \varepsilon}\bigg)
      + \log \bigg(\frac{{\qtext 4}^{k+1}
              \|g\|_{ C^k({\ktext \bar B^n{\ptext (2)}})}}{\varepsilon}\bigg) + 1 \bigg),
\end{equation}
while {\ftext the width $W$ satisfies
\begin{equation} \label{eqn.Wwidth 2}
   W \le C  {\ftext m}n \bigg(\frac {\varepsilon}{{\qtext 4}^{k+1}
                        \|g\|_{ C^k({\ktext \bar B^n{\ptext (2)}})}}\bigg)^{-n/k} \bigg(
   \log \bigg( \frac {{\qtext 4}^{k+1} \|g\|_{ C^k({\ktext \bar B^n{\ptext (2)}})}}{\varepsilon} \bigg)
               + 1 \bigg).
\end{equation}}
\end{proof}

Because neural networks are naturally compositional, it is straightforward to extend Theorem~\ref{thm.apprx1} to the case where the function $f$ being approximated is given by a composition of functions of the form \eqref{eqn.smooth+hol}.

\begin{theorem} \label{thm.apprx2}
Let $J\in \mathbb Z_+$ and $\varepsilon\in (0,1)$. Suppose there is a sequence of holomorphic functions $q_j$ and smooth
functions $g_j$ for $j = 1, \ldots, J$, where the $q_j$ and $g_j$ satisfy the same assumptions {\ytext as functions $q$  and $g$} in Theorem~\ref{thm.apprx1} with $m=n$, and $v \in \bar B^n(1)$. Consider a nonlinear operator function, $H$, defined by
\begin{equation}
   \Lambdamap \mapsto g_J({q}_J(\Lambdamap)g_{J-1}({q}_{J-1}(\Lambdamap)
        \dots g_2({q}_2(\Lambdamap)g_1({v^\top} {q}_1(\Lambdamap)))) \dots) .
\end{equation}
There exists an operator recurrent network $H_{\theta}$ with depth $JL$ and width $W$, with $W$ and $L$ given by \eqref{eqn.Wwidth} and \eqref{eqn.Ldepth}, respectively, such that
\begin{equation}\label{deterministic estimate 1}
   \| H - H_{\theta} \|_{\mtext C(\mathcal B^{n \times n};\R^n)} \le C' \varepsilon
\end{equation}
with constant $C' = C'(k,n,r,J)$.
\end{theorem}

We will introduce a function $\mathcal{R}$ that measures the norms of the network parameters, $\theta$, and provide an upper bound on the value of this function in an approximation of such maps $f$ or $H$ by operator recurrent networks. This additional control over the norms of the weight parameters will later be used to bound the derivatives of $f_{\theta}$ or $H_{\theta}$, ultimately leading to a generalization bound.

{
\subsubsection*{Universal approximation by general operator recurrent networks}

Next we show an universal approximation result for general operator recurrent networks. We recall that the general operator recurrent networks can be written also as a basic operator recurrent network with an increased width, as shown in formulas \eqref{eqn.NN1def1 B}-\eqref{eqn.NN1def2 B}.

\begin{theorem}\label{not linear operator approximation theorem}
Let $n,K\in \mathbb Z_+$ and $\mathcal F^{(n,K)}=\{f^{(L,K)}_\theta :\R^{n\times n}\to \R^n\ |\ L\in \mathbb Z_+,\ \theta\in  (\R^n)^{\#\tilde P}\}$ be the space of general operator recurrent networks  $f^{(L,K)}_\theta$ of the form (\ref{eqn.NN1def1 exte3dnde})-(\ref{eqn.NN1def2 exte3dnde}) that have  the level $K$, the length $L$ and the width $n$. Let  $\mathcal Z\subset \R^{n\times n}$ be a compact set. Then for $K= 2n+1$, the set  $\mathcal F^{(n,K)}$ is dense in the space $C(\mathcal Z;\R^n)$.  
\end{theorem}

\begin{proof}
{In the proof we will first consider the matrix $X$ as a vector in $R^{n^2}$  and approximate a 
function $X\to g(X)$, where $g\in C(\mathcal Z;\R^n)$, using a standard neural network. After this we represent the obtained neural network as a general recurrent neural network that has the level $K=2n+1$.}

Let $\Lambdamap=(x_{jk})_{j,k=1}^n\in \R^{n\times n}$. We can consider $\Lambdamap$ as a vector consisting of 
 $n^2$ elements and define a single layer neural network $G:\R^{n\times n}\to \R^n$ of form 
  \beq
 {\mathcal G}(\Lambdamap)=({\mathcal G}_p(\Lambdamap))_{p=1}^n 
 \eeq  
 where
 \beq\label{standard}
 {\mathcal G}_p(\Lambdamap)=  \sum_{m=1}^M  b^{(m)}_{p}   \phi( (\sum_{j,k=1}^n a^{(km)}_{pj}x_{jk})+c_p^{(m)}),
 \eeq   
 $p=1,2,\dots,n$ and $b^{(m)}_{p},a^{(km)}_{pj},c_p^{(m)}\in \R$. 

Let now $g:\R^{n\times n}\to \R^n$,  $g(\Lambdamap)=(g_p(\Lambdamap))_{p=1}^n$  be a continuous function, $\e>0$ and $\mathcal Z\subset \R^{n\times n}$ be a compact set. By universal approximation results for standard neural networks \cite{hornik1991approximation,pinkus1999approximation},
for each $p$ there is a neural network ${\mathcal G}_p:\R^{n\times n}\to \R^n$ of the form \eqref{standard}
(having a  sufficiently large width $M$) such that
\beq\label{eps error 2}
\|{\mathcal G}_p(\Lambdamap)-g_p(\Lambdamap)\|_{\R^n}\leq \frac{\e}n,\quad \hbox{for all }\Lambdamap\in \mathcal Z.
\eeq
We can write $ {\mathcal G}(\Lambdamap)$ using  matrix notation as 
 \beq\label{matrix notation}
 {\mathcal G}(\Lambdamap)=  \sum_{m=1}^M B^{(m)} \phi(c^{(m)}+ \sum_{k=1}^nA^{(km)}\Lambdamap v_k),
 \eeq  
 where
 \ba
 & &B^{(m)}=\hbox{diag}(b^{(m)}_{1},\dots,b^{(m)}_{n})\in \R^{n\times n},\quad
  A^{(km)}=(a^{(km)}_{pj})_{p,j=1}^n\in \R^{n\times n},\\ & &c^{(m)}=
  (c^{(m)}_{p})_{p=1}^n\in \R^{n},\quad v_k=(\delta_{pk})_{p=1}^n\in \R^{n}.
 \ea
  Moreover, we can write  $ {\mathcal G}(\Lambdamap )$ in \eqref{matrix notation} as 
 \beq \nonumber
 {\mathcal G}(\Lambdamap )&=& S_M,
 \\[0.25cm] \label{eq: NN special 1}
 S_0&=&0,
 \\ \nonumber
 S_{m}&=& \phi( B^{(m)}c^{(m)}+ \sum_{k=1}^nB^{(m)}A^{(km)}\Lambdamap v_k)+S_{m-1} .
 \eeq  
Writing for $m=0,1,\dots,M$ and $k=1,2,\dots,n,$
\beq \label{eq recurrent 1}
& &h_{0}=0,\\[0.25cm]
& & h_{m(2n+1)+2k-1}= v_k,\\[0.25cm]
& &h_{m(2n+1)+2k}=A^{(km)}\Lambdamap h_{m(2n+1)+2k-1},\\
& &h_{m(2n+1)+2n+1}=  \phi(B^{(m)}c^{(m)}+ \sum_{k=1}^nB^{(m)} h_{m(2n+1)+2k})+
h_{m(2n+1)}, \label{eq recurrent 4}
\eeq
so that $h_{m(2n+1)}=S_{m-1}$, we see that ${\mathcal G}(\Lambdamap )= S_M=h_{M(2n+1)}$.
Thus ${\mathcal G}(\Lambdamap )$ can be written as a {general operator recurrent network} 
(\ref{eqn.NN1def1 exte3dnde})-(\ref{eqn.NN1def2 exte3dnde}) having depth
$L = (M+1)(2n+1)$, level $K = 2n+1$ and width $n$, and parameters
\ba
& &\ h_{0}=0,\\
& & b_\theta^{\ell,0}= v_k,\ b_\theta^{\ell,1}=0,\ A_\theta^{\ell,k,i}=0,\ \ B_\theta^{\ell,k,i}=0\
\quad \hbox{for }\ell={m(2n+1)+2k-1},\\
& & b_\theta^{\ell,i}\,=0,\ A_\theta^{\ell,k,i}=0,\ \ B_\theta^{\ell,k,0}=A^{(km)},\
\ B_\theta^{\ell,k,1}=0\quad \hbox{for }\ell={m(2n+1)+2k},
\ea
and for $\ell=m(2n+1)+2n+1$,
\ba
& &b_\theta^{\ell,0}=0,\
b_\theta^{\ell,1}=B^{(m)}c^{(m)},\
 A_\theta^{\ell,k,0}=0,\  A_\theta^{\ell,k,1}=B^{(m)},
 \ B_\theta^{\ell,k,i}=0
 \quad \hbox{for }k\le K-1,\\
 & &b_\theta^{\ell,i}=0,\
 A_\theta^{\ell,k,0}=I,\  A_\theta^{\ell,k,1}=0,
 \ B_\theta^{\ell,k,i}=0
 \quad \hbox{for }k=K .
\ea
Moreover, by \eqref{eps error 2}, the inequality 
\beq\label{eps error 1}
\|{\mathcal G}(\Lambdamap)-g(\Lambdamap)\|_{\R^n}\leq \e\quad \hbox{for all }\Lambdamap\in \mathcal Z.
\eeq
is satisfied.
 \end{proof}
 
 \begin{remark}\label{Remark on density}
 Let $\mathcal K_{L,K}=\mathcal K_{L,K}^n$ be the space of functions $\Lambdamap\to f_\theta(\Lambdamap)$ where $f_\theta$ is a  general recurrent neural network of depth $L$, level $K$ and width $n$ and vanishing non-learned parts of the weight matrices, $A^{\ell,i,k,(0)}$ and $B^{\ell,i,k,(0)}$. Moreover, let $\mathcal K_{L,K}^{(sp)}\subset \mathcal K_{L,K}$ be the space of (special) general recurrent neural network $f_\theta(\newrandomMvariable)\in \mathcal K_{L,K}$ that of the form \eqref {eq: NN special 1} and that can be written in the form (\ref{eqn.NN1def1 exte3dnde})-(\ref{eqn.NN1def2 exte3dnde}) with $\theta\in {\tilde \Theta}_{L,K}$. We note that $\theta\in {\tilde \Theta}_{L,K}$ implies that that the learned parts of the weigh matrices, $A^{\ell,i,k,(1)}_\theta$ and $B^{\ell,i,k,(1)}_\theta$ are bounded.
  
 We observe first that in formula \eqref{standard} we can multiply  numbers $a^{(km)}_{{p}j}$, $b^{(m)}_{{p}}$, and $c_{p}^{(m)}$ by  $0<\lambda<1$ then the function $g(\Lambdamap)$ is changed to $\lambda^2{g}(\Lambdamap)$. Second, we observe that if ${g}^{(1)}(\Lambdamap)$ and ${g}^{(2)}(\Lambdamap)$ are two neural networks in $\mathcal K_{L,K}^{(sp)}$, then their sum, 
${g}^{(1)}(\Lambdamap)+{g}^{(2)}(\Lambdamap)$, can be written as 
a function ${g}(\Lambdamap)\in \mathcal K_{2L,K}^{(sp)}$ by replacing
in definition \eqref{standard} of ${g}_{p}^{(2)}(\Lambdamap)$ the initial value
$S_0=h_0=0$ of ${g}_{p}^{(2)}(\Lambdamap)$ by the output of the neural
network ${g}_{p}^{(1)}(\Lambdamap)$. Note that then the sum ${g}_{p}^{(1)}(\Lambdamap)+{g}^{(2)}(\Lambdamap)$ of the two neural networks of length $L$ is  represented 
as  a neural network of the double length $2L$.

 By combining the above two observations, we conclude that the union
\beq
   \mathcal K^{(sp)}_{\infty,K}=\bigcup_{L=1}^\infty K^{(sp)}_{L,K}
\eeq
is a linear subspace  that is equal to the space of neural networks $\bigcup_{L=1}^\infty \mathcal K_{L,K}$ considered in Theorem \ref{not linear operator approximation theorem}. {\color{black}The fact that $\mathcal K^{(sp)}_{\infty,K}$ is a linear subspace will be essential in Subsection~\ref{subsec: Bayes}, where we consider Bayes estimators and the orthogonal projection to the subspace $\mathcal K^{(sp)}_{\infty,K}$.} Moreover, Theorem \ref{not linear operator approximation theorem} implies that for $K\ge 2n+1$ the set $ \mathcal K^{(sp)}_{\infty,K}$ is dense in $L^\infty(\B_{n\times n}(1);\R^n)$.
\end{remark}}

\subsection{Expressivity}
\label{ssec.piecewisepoly}

{\maartenbluefont One way to assess the representational power of a
  network architecture is to study its range. More precisely, suppose
  that one can partition the output space into regions and also
  locally characterize the network as it is restricted to each of
  these regions. Networks that partition the output space to a larger
  number of regions are considered to be more complex, or in other
  words, possess better representational power. In regular, deep
  rectifier networks, which are linear splines that can be written as
  a composition of max-affine spline operators \cite{balestriero2018},
  each application of the rectifier activation partitions the output
  space into regions bounded by hyperplanes. In contrast, we will here
  show that the corresponding regions for a recurrent operator network
  have algebraic varieties as their boundaries and within each region,
  the network is a polynomial operator function.} To make this
description precise, we introduce and motivate several definitions.

\begin{definition}
\label{defn.operator-polynomial}
An \emph{operator polynomial} of degree $d$ on $\R^{n \times n}$ is
a function $P : \R^{n \times n} \to \R^{n \times n}$ defined as
\begin{align}
   P(\Lambdamap) & = (A_{00} + A_{10} \Lambdamap A_{11} + 
     A_{20} \Lambdamap A_{21} \Lambdamap A_{22} + 
     \ldots + A_{d0} \Lambdamap \ldots \Lambdamap A_{dd} )
     \\
     & = A_{00} + \sum_{j=1}^d A_{j0}
                 \prod_{k=1}^j (\Lambdamap A_{jk}).
\end{align}
{\maartenbluefont with matrix-valued coefficients $A_{ij} \in \R^{n \times
    n}$.}
\end{definition}

The definition of an operator polynomial generalizes the usual
definition of a polynomial $\R \to \R$, and is equivalent when $n =
1$. We will prove in Theorem~\ref{thm.piecewisepoly} that locally, all
operator recurrent networks behave like operator polynomials. This is
analogous to the result that locally, all deep rectifier networks
behave like linear functions.

We next introduce the concept of a \emph{polynomial region}.
To motivate this definition, let us recall that in an
operator recurrent network we have activation function terms
of the form
\begin{equation}
\phi_{\ell}(b^{\ell,1}_{\theta} + B^{\ell}_{\theta} \Lambdamap h_{\ell}),
\end{equation}
where $\phi_{\ell}$ is a leaky rectifier activation. Then the first
vector component of this expression is equal to
\begin{equation}
\begin{cases}
(b^{\ell,1}_{\theta} + B^{\ell}_{\theta} \Lambdamap h_{\ell})_1, &
    (b^{\ell,1}_{\theta} + B^{\ell}_{\theta} \Lambdamap h_{\ell})_1 > 0 \\
\eta (b^{\ell,1}_{\theta} + B^{\ell}_{\theta} \Lambdamap h_{\ell})_1, &
    (b^{\ell,1}_{\theta} + B^{\ell}_{\theta} \Lambdamap h_{\ell})_1 \le 0.
\end{cases}
\end{equation}
Therefore, the activation function partitions the first vector
component of the output into two regions, depending on the sign of
$(b^{\ell,1}_{\theta} + B^{\ell}_{\theta} \Lambdamap h_{\ell})_1$.
If we assume that $h_{\ell}$ is a continuous function of $\Lambdamap$,
then the resulting output above will also be continuous in $\Lambdamap$,
and therefore the boundary between these two regions is given by
\begin{equation}
\label{eqn.polyboundary}
(b^{\ell,1}_{\theta} + B^{\ell}_{\theta} \Lambdamap h_{\ell})_1 = 0
\end{equation}
under the assumption that the two regions are nonempty and the quantity
in \eqref{eqn.polyboundary} does not vanish identically in an open set.
This is expected behavior for all neural networks using rectifier activations.
In the case of operator recurrent networks, however, this partition
is highly nonlinear due to the presence of a multiplication term.
Assume that $h_{\ell} = Q(\Lambdamap)v$ where $Q(\Lambdamap)$ is an
operator polynomial and $v \in \R^n$ is a vector. Then 
one can observe that
$b^{\ell,1}_{\theta} + B^{\ell}_{\theta} \Lambdamap h_{\ell}$
can also be written as a polynomial $P(\Lambdamap)v$, with $P$ having
degree one higher than $Q$. Thus the boundaries
separating the regions of the output of an activation function
in an operator recurrent network are subsets of zero sets of
multivariate polynomials (such sets are also called
\emph{algebraic varieties}). These observations motivate
the following definition and theorem.

\begin{definition}
A \emph{polynomial region} is an open subset
$U \subset \R^{n \times n}$ such that for any boundary point
$x_0 \in \partial U$, there exists an open set $V$ containing $x_0$,
a finite index set $J$, operator polynomials $P_j$ and vectors $v_j \in \R^n$
for $j \in J$, such that
\begin{equation}
\label{eqn.polyregioncondition}
V \cap U = \{ \Lambdamap \in V:
    (P_j(\Lambdamap) v_j)_1 > 0 \text{ for all } j \in J \}.
\end{equation}
\end{definition}

\begin{remark}
Since we can always compose with a permutation matrix, the coordinate
index $1$ can be replaced by another index without loss of generality,
for example, $(P_j(\Lambdamap) v_j)_k > 0$ for any $k$.
\end{remark}

The set $\{ \Lambdamap \in \R^{n \times n}:(P(\Lambdamap) v)_1 = 0\}$, if
nonempty, is a submanifold of codimension 1 in $\R^{n \times n}$,
since the map $\Lambdamap \mapsto (P(\Lambdamap) v)_1$ can be viewed as a
real multivariate polynomial $\R^{n^2} \rightarrow \R$.  Thus we can
consider a polynomial region as being a high-dimensional
generalization of a domain in Euclidean space bounded between a
collection of polynomial surfaces. As mentioned above, in operator
recurrent networks activation functions partition the output space
nonlinearly according to zero sets of polynomials.  The partitions are
precisely described by the polynomial regions defined above. The
analogous behavior in standard deep rectifier networks appears in the form of simplices,
which originate from activation functions partitioning the output
space along hyperplanes. We have

\begin{theorem} \label{thm.piecewisepoly}
Let $f_{\theta}$ be an operator recurrent network on $\R^{n \times n}$
with layerwise outputs $h_{\ell}$, $\ell = 0, \ldots, L$.
Then, for each $\ell$, there exists a countable collection of
polynomial regions $\{U^{\ell}_i\}$ in $\R^{n \times n}$ satisfying:
\begin{enumerate}
\item This collection partitions $\R^{n \times n}$;
    that is, $U^{\ell}_i \cap U^{\ell}_j = \emptyset$ for every $i \neq j$,
    and $\bigcup \overline{U^{\ell}_i} = \R^{n \times n}$.
\item Every open ball $B \subset \R^{n \times n}$ only nontrivially
    intersects $U^{\ell}_i$ for finitely many $i$.
\item The restriction of $h_{\ell}$ to each $U^{\ell}_i$ is an 
    operator polynomial of degree at most $\ell$, applied to $h_0$.
\end{enumerate}
\end{theorem}

\begin{proof}
The result of the theorem characterizes $f_{\theta}$ as a piecewise operator polynomial
whose domain is partitioned into polynomial regions, on each of which $f_{\theta}$ is
exactly an operator polynomial.
Since operator polynomials form a vector space, then this characterization is also
closed under addition and scalar multiplication.
In particular, if $f_{\theta}$ and $g_{\theta}$ are two such functions,
then any linear combination of the two functions also satisfies the result of the theorem,
except with a new partition of polynomial regions which is the mutual refinement
of those of each of the original two functions.
Therefore, to prove this theorem, it suffices to consider a slightly simplified
version of an operator recurrent network, in which
\begin{align}
   f_{\theta}(\Lambdamap) & = h_L, 
\\
   h_{\ell+1}(\Lambdamap) & = b^{\ell,0}_{\theta}
            + \phi_{\ell} (b^{\ell,1}_{\theta}
            + B^{\ell}_{\theta} \Lambdamap h_{\ell}),
\label{eqn.NNsimple}
\end{align}
where $\Lambdamap \in \R^{n \times n}$ is the input, $h_0 \in \R^n$ is
given, and $\phi_{\ell}$ is a leaky rectifier activation function with
$\eta > 0$.  The network given in \eqref{eqn.NNsimple} is derived by
taking \eqref{eqn.NN1def2} and setting several of the weight matrices
to zero.  This is done to highlight the fact that the nonlinearity is
derived by the matrix-vector multiplication term $B^{\ell}_{\theta}
\Lambdamap h_{\ell}$. By our above argument, if the theorem holds for
this simplified version, then since the general form of an operator
recurrent network in \eqref{eqn.NN1def2 exte3dnde} is merely a sum
of terms of the simplified form, then the result will hold in general.

On this simplified case, we proceed by induction,
and in our inductive step we construct a new collection of polynomial
regions based on the previous collection of polynomial regions.
For the base case, the result of the theorem holds for $\ell = 0$, since the
output of the neural network is $h_0$, which is independent of
$\Lambdamap$. Now, for the induction, suppose the claim is true at output layer $\ell$. 
Then there exists some collection of polynomial regions $\{U^{\ell}_m \}$
that partitions $\R^{n \times n}$ (that is, disjoint sets such that
the union of their closures is $\R^{n \times n}$), such that for any
given region $U^{\ell}_m$ and {\maartenbluefont for every $\Lambdamap \in
  U^{\ell}_m$, $h_{\ell}(\Lambdamap)$ is expressible as an operator
  polynomial $P_{\ell,m} (\Lambdamap)$ as applied to $h_0$,} that is,
\begin{equation}
\label{eqn.oppolyexpansion}
   h_{\ell}(\Lambdamap) = {\xxtext    P_{\ell,m} (\Lambdamap):=}A_{00} h_0 + \sum_{i=1}^{\ell} 
   A_{i0} \left[\prod_{j = 1}^i ( \Lambdamap A_{ij}) \right] h_0 ,
\end{equation}
where $A_{ij}$ are the matrix-valued polynomial coefficients of
$h_{\ell}(\Lambdamap)$ {\xxtext in the region $U^{\ell}_m$}.
Now we apply the iteration \eqref{eqn.NNsimple} to produce the next layer.
{\maartenbluefont We first construct the regions and then prove that these
  partition the matrix space and are polynomial regions.} We define
\begin{align}
\label{eqn.Dell1j}
  D^{\ell}_{1,j} & =
  \{ \Lambdamap\in \R^{n \times n}: (b^{\ell,1}_{\theta} + B^{\ell}_{\theta} \Lambdamap
{\xxtext    P_{\ell,m}} {\xxtext (\Lambdamap)})_j > 0\}, \\
D^{\ell}_{2,j} & = \mathrm{int}\left((D^{\ell}_{1,j})^c \right),
\end{align}
meaning that $D^{\ell}_{2,j}$ is the interior of the complement of
$D^{\ell}_{1,j}$ in $\R^{n \times n}$. Note that 
 that a polynomial $P:\R^{n \times n} \to \R$, or more generally, a
 real-analytic function, cannot vanish in an open set unless it is
 identically zero (see \cite{rudin1991}). Thus, if $D^{\ell}_{1,j}$
is nonempty but also not all of $\R^{n \times n}$, then
\begin{equation}
D^{\ell}_{2,j} = {\xxtext \{ \Lambdamap\in \R^{n \times n}: (-b^{\ell,1}_{\theta} - B^{\ell}_{\theta} \Lambdamap
{\xxtext    P_{\ell,m}} {\xxtext (\Lambdamap)})_j> 0}\}.
\end{equation}
We note that $D^{\ell}_{1,j}, D^{\ell}_{2,j}$ are polynomial regions for every $\ell,j$.
The significance of these regions is that they are formed
by the application of the rectifier activation function $\phi_{\ell}$.
We aim to show that if at the inductive step, $h_{\ell}(X)$ is a piecewise
operator polynomial on the partition $\{ U^{\ell}_m\}$,
then we can use the regions $D^{\ell}_{1,j}, D^{\ell}_{2,j}$ to produce
a refinement $\{U^{\ell+1}_m\}$ that satisfies the theorem
for the case $\ell +1$. 

To this end, we explicitly construct the new collection of polynomial regions
$\{U^{\ell+1}_m\}$, checking that they are indeed polynomial
regions. We define this collection of subsets as the collection of all
such nonempty sets $U^{\ell+1}_m$ that can be written as
\begin{equation} \label{eqn.newpartition}
   U^{\ell+1}_m = U^{\ell}_{m'} \cap \bigcap_{j=1}^n  D^{\ell}_{k_j,j} ,
\end{equation}
for some index $m'$, and some $k_j \in \{1,2\}$.
The collection of sets in \eqref{eqn.newpartition} are thus a refinement of the original
partition $\{U^{\ell}_m\}$, where the refinement is produced by intersecting
with the sets $D^{\ell}_{1,j}, D^{\ell}_{2,j}$. 

It may be useful to observe that computing the $j$-th vector component of the next layer gives {\xxtext
  for $\Lambdamap \in U^{\ell}_{m'}$}
\begin{align}
\label{eqn.jcomponentnextlayer}
(h_{\ell+1}(\Lambdamap))_j & = \begin{cases}
(b^{\ell,0}_{\theta})_j + (b^{\ell,1}_{\theta} + B^{\ell}_{\theta} \Lambdamap {\xxtext    P_{\ell,m}} {\xxtext (\Lambdamap)})_j, &
\Lambdamap \in D^{\ell}_{1,j}{\xxtext \cap U^{\ell}_{m'}}\\
(b^{\ell,0}_{\theta})_j + \eta (b^{\ell,1}_{\theta} + B^{\ell}_{\theta} \Lambdamap {\xxtext    P_{\ell,m}} {\xxtext (\Lambdamap)})_j, &
\Lambdamap \in D^{\ell}_{2,j}{\xxtext \cap U^{\ell}_{m'}}.
\end{cases}
\end{align}

Having given our explicit construction of the new partition $\{U^{\ell+1}_m\}$,
we now must verify that they satisfy the conditions in the statement of the theorem.
In particular, we must show that the collection is a finite partition of the domain,
that each member is a polynomial region, and that $h_{\ell+1}$ restricted to each such region
is an operator polynomial. 
First we check that this new set $\{U^{\ell+1}_m\}$ partitions the space.
Note $D^{\ell}_{1,j} \cap D^{\ell}_{2,j} = \emptyset$, and
furthermore $\overline{D^{\ell}_{1,j}} \cup
\overline{D^{\ell}_{2,j}} = \R^{n \times n}$. Since
$\{D^{\ell}_{1,j}, D^{\ell}_{2,j}\}$ partitions $\R^{n \times n}$,
each element $U^{\ell+1}_m$ is constructed by picking one element from
each of $n+1$ different partitions of $\R^{n \times n}$ and taking
their intersection. Then it is clear that any two such sets have
empty intersection, and
\begin{equation}
\bigcup_m \overline{U^{\ell+1}_m} \subseteq 
\bigcup \Big( \overline{U^{\ell}_{m'}} \cap \bigcap_j \overline{D^{\ell}_{k_j,j} } \Big)
= \R^{n \times n} \cap \bigcap_{j, k_j} \overline{D^{\ell}_{k_j,j}} = \R^{n \times n}.
\end{equation}
Furthermore, since the $D^{\ell}_{k,j}$ are finite, and any open ball
only finitely intersects $\{U^{\ell}_m\}$ by induction hypothesis,
then the same must hold of $\{U^{\ell+1}_m\}$.

Next we show that each set $U^{\ell+1}_m$ is a polynomial region.  For
any given $m$ let $x \in \partial U^{\ell+1}_m$. Since $U^{\ell+1}_m$
can be expressed by \eqref{eqn.newpartition}, then there are indices
$m', k_j$ such that
\begin{equation}
   \Lambdamap \in \partial U^{\ell}_{m'} \cup
                   \bigcup_j \partial D^{\ell}_{k_j,j} .
\end{equation}
Since $U^{\ell}_{m'}$ and $D^{\ell}_{k_j,j}$ are polynomial regions,
then there exists a finite collection of open sets containing $x$
satisfying the polynomial region definition
\eqref{eqn.polyregioncondition} for each of the sets $U^{\ell}_{m'}$
and $D^{\ell}_{k_j,j}$.  Therefore, taking the intersection of these
open sets yields a new open set satisfying the conditions for
\eqref{eqn.polyregioncondition} for the set $U^{\ell+1}_m$.

Lastly, we check that $h_{\ell+1}$ is an operator polynomial applied
to $h_0$ when restricted to each such set. Suppose $h_{\ell+1}$ is
restricted to one such polynomial region $U^{\ell+1}_m$. Using the
index notation $m'$ and $k_j$ from the decomposition
\eqref{eqn.newpartition}, define a vector $b \in \R^n$ by $b =
(b_j)_{j=1}^n$ and
\begin{equation}
   b_j = (b^{\ell,0}_{\theta})_j + \gamma_j (b^{\ell,1}_{\theta})_j ,
\end{equation}
where
\begin{equation}
   \gamma_j = \begin{cases} 1, &\hbox{\fmtext for } k_j = 1, \\
     \eta, &\hbox{\fmtext for } k_j = 2 .
              \end{cases}
\end{equation}
Similarly, we define a matrix $B \in \R^{n \times n}$ by $B =
(B_{ij})$ and
\begin{equation}
   B_{ij} = \gamma_j (B^{\ell}_{\theta})_{ij} .
\end{equation}
Then, restricted to $\Lambdamap \in U^{\ell+1}_m$, we can write
\begin{equation}
   h_{\ell+1}(\Lambdamap) = b + B \Lambdamap {\xxtext P_{\ell,m}}
   {\xxtext (\Lambdamap)} .
\end{equation}
Combining the above with the induction hypothesis, it is clear then that $h_{\ell+1} {\xxtext (\Lambdamap)}$ can be expressed as an operator polynomial applied to $h_0$ when restricted to each $U^{\ell+1}_m$.
\end{proof}

The polynomial regions that emerge from an operator recurrent network can have very
nonlinear boundaries and thus have a much more complicated geometry compared to the linear
regions in standard rectifier networks.
In particular, because each polynomial region can be bounded by a number of high-degree polynomial submanifolds, they can be highly irregular and highly non-convex.
This behavior enables the resulting networks to potentially approximate
highly nonlinear functions with fewer layers compared with traditional rectifier networks,
which must approximate nonlinear behavior through piecewise linear behavior.
However, due to this additional complexity, it is nonetheless likely best to employ
these networks for nonlinear problems that naturally have operator polynomial or
operator analytic behavior, such as hyperbolic inverse problems.

We should also note that since every operator recurrent network has a rectifier network as a special case, then by utilizing the results of \cite{montufar2014}, we can construct a particular operator recurrent network of depth $L$ and width $n$ that possesses at least $2^{(L+1)n/2}$ distinct polynomial regions. Thus, the expressivity of the network, as measured by the size of the polynomial region partition, increases exponentially with depth $L$. In the example of representing matrix inversion, the expressivity on a special set of real symmetric matrices is detailed in Theorem~\ref{thm.ORNpolyforposneg}.

\section{Regularization function and basic estimates}

In this section, we will introduce a sparsity promoting regularization function. This function will later be used as a penalty term in optimization and is employed in training a network; it naturally arises in the analysis of inverse boundary value problems such as the one presented in Section~\ref{sec.IP} where weight matrices, $A^{\ell,i,(1)}$ and $B^{\ell,i,(1)}$ correspond to compact operators that are in a Schatten class. The regularization yields essentially improved generalization bounds in the later analysis. We note that regularization nowadays is commonly incorporated through the choice of method for non-convex optimization \cite{kukacka2017}.

\subsection{Convex regularizing function}
\label{ssec:convexreg}

We introduce { convex regularization functions, all denoted by $\mathcal{R}$ that, with a slight abuse of notations, are 
given by
\beq \nonumber
   \mathcal{R}(\theta) &=&
           \frac 12 \sum_{\mtext (\ell,{}i,p) \in P_1}
           \|\theta^{\ell,i}_{p{}}\|_{\R^n} ,\\
          \label{eqn.Rdef}
              \mathcal{R}(\theta^\ell) &=&
           \frac 12 \sum_{(i,p) \in I^{\ell}}
           \|\theta^{\ell,i}_{p{}}\|_{\R^n} 
           ,\\
           \nonumber
              \mathcal{R}(\theta^{\ell,i}) &=&
           \frac 12 \sum_{p\in I^{\ell,i}}
           \|\theta^{\ell,i}_{p{}}\|_{\R^n} ,
\eeq
where $I^\ell=\{(i,p): \ \exists (\ell,{}i,p) \in P_1\}$ and
$I^{\ell,i}=\{p: \ \exists (\ell,{}i,p) \in P_1\}$ 
and} $\theta$ is a set of parameters for neural network $f_{\theta}$; the index notation was introduced in (\ref{thetavector}) with $i = 0,1$ and $P_1 \subset P$ the index set (\ref{index set P}) corresponding to the weight matrices. We will use this function as part of an explicit regularization. \emph{The function $\mathcal{R}$ measures the sum of the Schatten seminorms of the learnable weight matrices of the network}, which we will show below.

We consider the value of $\mathcal{R}(\theta)$ for a neural network $f_\theta$ when the matrices \newline $A^{\ell,{}i,(1)}{\mtext :\R^n\to \R^n}$ and $B^{\ell,{}i,(1)}{\mtext :\R^n\to \R^n}$ satisfy
\begin{equation} \label{eqn.Aschattenbd}
   A^{\ell,{}i,(1)}_\theta,B^{\ell,{}i,(1)}_\theta\in \mathcal{S}_q ,
\end{equation}
where $\mathcal{S}_{\fmtext q}$ is the Schatten $q$ class of matrices; here, $q = 1/2$. The Schatten seminorm ${\fmtext q}$, denoted by $\| \,\cdotp\|_{\mathcal{S}_{\fmtext q}},$ is the $\ell^q$-seminorm of the vector of singular values of a matrix. {\ptext We note that for $0<{\fmtext q}<1$ the $\ell^q$-seminorms $\|\,\cdotp\|_{q}$ are not norms but satisfy $\|x+y\|_q^q\leq \|x\|_q^q+\|y\|_q^q$. 

If $A^{\ell,{}i,(1)} $ is an $n \times n$ matrix with singular values $\sigma_p^{\ell,{}i} $ and corresponding singular vectors $u_p^{\ell,{}i} , v_p^{\ell,{}i} $ then we can choose parameters (cf.~(\ref{eqn.Cmatrixdecomp})-(\ref{eqn.Cmatrixdecomp B})) 
\beq
   \theta_{{\qtext 2p-1}}^{\ell,{}i}
   = (\sigma_p^{\ell,{}i} )^{1/2} u_p^{\ell,{}i} ,\quad
   \theta_{2p}^{\ell,{}i}  = (\sigma_p^{\ell,{}i} )^{1/2} v^{\ell,{}i} _p ,
\eeq
so that
\beq\label{A-model}
 A^{\ell,{}i,(1)}_\theta= \sum_{p=1}^n 
           \theta^{\ell,{}i}_{{\qtext 2p-1}} (\theta^{\ell,{}i}_{2p})^T=\sum_{p=1}^n {\ytext \sigma_p^{\ell,{}i}}
u_p^{\ell,{}i}     (   v^{\ell,{}i} _p)^T .
\eeq
We note that the singular values $\sigma_p^{\ell,{}i} $ are bounded by the norm of $A^{\ell,{}i,(1)} $ and that the singular vectors $u_p^{\ell,{}i}$ and $v_p^{\ell,{}i} $ are orthonormal vectors. We also note that 
generally the vectors $\theta^{\ell,i}_p$ that parametrize the neural network
are not assumed to be orthonormal, but it is possible to choose those to be
parallel to the orthogonal vectors that define the  singular value decompositions
of the weight matrices. Moreover, we have
 \begin{equation}\label{cost function}
 \sum_{i=0}^1 \sum_{p=1}^n (\|\theta_{2p-1}^{\ell,{}i} \|_{\mtext \R^n} +
    \| \theta_{2p}^{\ell,{}i}  \|_{\mtext \R^n})
    =  \sum_{i=0}^1 \sum_{p=1}^n 2 (\sigma_p^{\ell,{}i} )^{1/2}
    = 2 \sum_{i=0}^1 \|  A^{\ell,{}i,(1)}_\theta \|_{\mathcal{S}_{1/2}}^{\ktext 1/2}.
\end{equation}
A similar analysis applies to $B^{\ell,{}i,(1)}_\theta$, $i = 0,1$ with $p = n+1, \ldots, 2n$. Thus the function $\mathcal{R}$ measures the sum of the $\mathcal{S}_{1/2}$ seminorms of the matrices of the network as announced above.

Furthermore, we observe that when $\|\theta_p^{\ell,{}i} \|_{\R^n} \leq 1$ for all $p$, we have
\beq\label{cost function and norm}
   \|  A^{\ell,{}i,(1)}_\theta \|_{\R^n\to \R^n}+\|  B^{\ell,{}i,(1)}_\theta \|_{\R^n\to \R^n}
&\leq&  \sum_{p=1}^{2n}  \|\theta_{2p-1}^{\ell,{}i} \|_{\mtext \R^n}\,\cdotp 
    \| \theta_{2p}^{\ell,{}i}  \|_{\mtext \R^n}
\nonumber\\
&\leq& \frac 12 \sum_{p=1}^{2n} ( \|\theta_{2p-1}^{\ell,{}i} \|_{\mtext \R^n}+
    \| \theta_{2p}^{\ell,{}i}  \|_{\mtext \R^n}) \leq \mathcal{R}(\theta^{\ell,i}) .
\eeq
We will use this estimate in the proof of Lemma~\ref{lemma.lipf1} below.

\subsection{Truncated network}
\label{ssec:truncnet}

For our later generalization results, it is important to guarantee that
the output of any given network is bounded. 
Indeed, our goal is to construct an operator recurrent network
$f_\theta: {\mathcal B^{n \times n}} \to \R^n$,
where ${\mathcal B^{n \times n}} = {\mathcal B^{n\times n}}(1) = \{ \Lambdamap \in \R^{n \times n}:\ \|\Lambdamap\|_{\mtext \R^n \to \R^n}\le 1\}$ is the closed unit ball in the set of matrices, that approximates a bounded, {\fmtext continuous} function $f: {\fmtext\mathcal B^{n\times n}}\to \R^n$.
As we a priori know that the function we approximate is bounded
by 
\beq
\|f\|_\infty=\|f\|_{L^\infty({\fmtext\mathcal B^{n\times n}};\R^n)}=\sup_{\Lambdamap\in {\fmtext\mathcal B^{n\times n}}}\|f(\Lambdamap)\|_{\R^n},
\eeq
we can add to the network $f_\theta$ two additional layers that cut off
any coordinates values that are too large. That is, {we introduce a new parameter, $m \in \R_+$, satisfying
\beq\label{G net parameter}
m \geq \|f\|_\infty
\eeq 
and}
add two layers that
implement the function
$T_m:\R^{n}\to \R^n$ where for $x=(x_j)_{j=1}^n\in \R^n$ 
\beq\label{G net}
& &T_m(x)=-b+\phi_0(b+y),\quad y=b-\phi_0(b-x),\quad \\
& &\hbox{where $b=(m,m,\dots ,m)^T\in \R^n$},\nonumber
\eeq
and $\phi_0$ is the standard rectifier function ``ReLU''. We note that
then  $T_m(x) = (T_m(x_j))_{j=1}^n$ where $T_m(x_j)=\max(-m,\min(x_j,m))$.

\begin{definition}
\label{def.truncatednetwork}
We say that ${\bar f}_\theta:{\fmtext\mathcal B^{n\times n}}\to \R^n$ is a \emph{truncated basic or, respectively, a truncated general) operator recurrent network} of depth $L+2$ and width $n$ if
\beq \label{truncated network}
   {\bar f}_\theta = T_m \circ f_\theta ,
\eeq 
where $T_m$ is of the form \eqref{G net} and $f_\theta$ is a  basic (or, respectively, general) operator recurrent network with depth $L$ and width $n$.
\end{definition}

Truncated neural networks make it possible to effectively use Hoeffding's inequality in studying their generalization properties. Below, we use that for a truncated  general operator recurrent network
${\bar f}_\theta$ we have 
\beq\label{new L infty bound}
\| {\bar f}_\theta\|_{L^\infty({\fmtext\mathcal B^{n\times n}};\R^n)}\leq {n^{1/2}m} 
\eeq
and as the map $T_m$ has Lipschitz constant 1, the  Lipschitz constant
of $\theta \mapsto T_m( \tilde f_\theta(\Lambdamap))$ is bounded by the   Lipschitz constant
of $\theta \mapsto \tilde f_\theta(\Lambdamap)$.
We note that in \eqref{new L infty bound} the factor $n^{1/2}$ appears due to the fact that we use
the Euclidean norm $\|\,\cdotp\|_2$ in $\R^n$. If the norm of $x=(x_j)_{j=1}^n\in \R^n$ is replaced by 
{ the norm  $\|x\|_{\max}=\|x\|_\infty=\max_j |x_j|$, that is, if we replace the Euclidean space $(\R^n,\|\cdot\|_{\R^n})$ by $(\R^n,\|\cdot\|_{\max})$ and 
use $m=\sup_{\Lambdamap\in {\fmtext\mathcal B^{n\times n}}} \| f(\Lambdamap)\|_{\max}$, we obtain
\beq \label{new L infty bound 2}
\sup_{\Lambdamap\in {\fmtext\mathcal B^{n\times n}}} \| {\bar f}_\theta(\Lambdamap)\|_{\max} \leq 
m.
\eeq}

\subsection{Intermediate function and regularization determining the loss functions}
\label{subsec: regulirized loss}

Here, we assume that the network ${\bar f}_{\theta}$ is a truncated basic recurrent operator
  network that satisfies \eqref{new L infty bound}. To guarantee a
generalization error bound, one needs to avoid the problem of
overfitting, in which ${\bar f}_{\theta}$ accurately approximates
  $f:\R^{n\times n}\to \R^n$ on the training set $S$, but poorly approximates $f$ away from $S$. To this end, we introduce a regularizing penalty term using $\mathcal{R}$.

\begin{definition}
For parameter $\theta$ and the pair $(\Lambdamap,\yvariable)$, we let ${\fmtext\mathcal{L}}$ be given by 
\begin{equation}
    \label{defn.lossfunction1}
   \mathcal{L}(\theta,\Lambdamap,\yvariable)
   = \|{\bar f}_{\theta}(\Lambdamap)-\yvariable\|^2_{\R^n}
\end{equation}
Moreover, we let ${\fmtext\mathcal{L}_{r}}$ with regularization parameter $\alpha \in (0,1)$ be given by 
\begin{equation}
    \label{defn.lossfunction2}
   {\fmrtext\mathcal{L}_{r}(\theta,\Lambdamap,\yvariable)
   = \|{\bar f}_{\theta}(\Lambdamap)- y\|^2_{\R^{\ktext n}} + \alpha \mathcal R(\theta).}
\end{equation}
\end{definition}

\medskip\medskip

\noindent
We denote by $\Theta$ the set of all parameters $\theta$ that the weight matrices of the network ${\bar f}_{\theta}$ depend upon; more precisely  
\beq\label{eq: parameter space}
   \Theta = {\color{black} \Theta_{(L)}} = \{( \theta^{\ell,{}i}_{p})_{(\ell,{}i,p)\in P} \in (\R^n)^{\#P}:\ \|\theta^{\ell,{}i}_{p}\|_{\R^n}\leq 1\} ,
\eeq
where $P$ is the index set \eqref{index set P}. The regularization term shows up explicitly and independently in the estimate for the Lipschitz constant of the network as well as in the estimate for its derivatives with respect to the weights in $P_1$ in the next subsection. { Also, for the general recurrent operator neural networks we denote the parameter space by
\beq\label{eq: generalparameter space}
   \tilde \Theta = {\color{black} \tilde \Theta_{(L,K)}} = \{(\tilde \theta^{\ell,i,k}_{p})_{(\ell,i,p,k)\in \tilde P} \in (\R^n)^{\#\tilde P}:\ \|\theta^{\ell,i,k}_{p}\|_{\R^n}\leq 1\} .
\eeq}

\subsection{Basic estimates of the recurrent operator neural network}

\subsubsection{Derivative with respect to weights}
\label{subsec.weightderiv}

{We show how controlling the norms of the parameter $\theta$}  provides an upper bound on directional derivatives in a local neighborhood for the neural network $f_{\theta}$, given by
(\ref{eqn.NN1def1})--(\ref{eqn.NN1def2}), as a function of $\theta$. Such a bound is crucial to controlling the behavior of the neural network during training. The key intuition here is that estimates of the derivative, which also give upper bounds on the local Lipschitz constant of $\theta \to f_\theta(X)$, provide some knowledge concerning the behavior of the { \color{black} regularized loss function in a neighborhood of its minimum.}

In the lemma below we consider a basic operator recurrent network.  Recall from \eqref{eqn.Cmatrixdecomp} that the weight matrices $A^{\ell,k}_{\theta}$ and $B^{\ell,k}_{\theta}$ depend quadratically on the column vectors contained within any parameter set $\theta$. This lemma generalizes easily to general operator recurrent networks and is used in the proofs of Theorem~\ref{thm.sparsity} and Lemma~\ref{lemma.thetafinite}.

\begin{lemma} \label{lemma.lipf1}
Let $f_{\theta}{\ktext :\R^{n\times n}\to \R^n}$ be a {\ptext basic} operator recurrent network with leaky rectifier activations, and $h_0$ satisfy $\|h_0\| \le 1$. 
{\ytext Let $\|\Lambdamap\|\leq 1$.}
Then, for 
$(\ell,{}i,p)\in P_1$, see \eqref{index set P}, the {local Lipschitz constant (or the derivative, if it exists)} of
$f_{\theta}(\Lambdamap)$ with respect to $\theta^{{\ktext \ell,{}i}}_p$ is bounded by
$K^{{\ktext \ell,{}i}}_p$ with
\begin{equation}
    K^{{\ktext \ell,{}i}}_p \le {\ptext {\qtext {}\,c_0^{L+1}}} \| \theta^{{\ktext \ell,{}i}}_{{\ptext (p)'}} \|
                              \exp(\mathcal{R}(\theta)),
\end{equation}
where ${\ptext (p)'} = p + 1$, if $p$  is {\ytext odd and ${\ptext (p)'} = p - 1$, if $p$ is even}.
For 
$(\ell,{}i,p)\in P_2$, see \eqref{index set P}, the derivative of
$f_{\theta}(\Lambdamap)$ with respect to $\theta^{{\ktext \ell,{}i}}_p$ is bounded by
$K^{{\ktext \ell,{}i}}_p$ with 
\begin{equation}\label{c0 bnd}
    K^{{\ktext \ell,{}i}}_p \le {\ptext {\qtext {}\,c_0^{L+1}}}  
                              \exp(\mathcal{R}(\theta)),
\end{equation}
{that is, $\hbox{Lip}(\theta^{{\ktext \ell,{}i}}_p\to f_{\theta}(\Lambdamap))\leq c_0^{L+1} \exp(\mathcal{R}(\theta))$} {for all $(\ell,{}i,p)\in P$}.
\end{lemma}

\begin{proof}
{\color{black} In the proof we estimate the derivatives of the output of $\ell$th layer and the results are combined in a way that is analogous to the backpropagation algorithm.}
{\ktext We consider $(\ell,{}i,p)\in P_1$; that is,
we consider derivatives with respect to parameters that
determine the weight matrices. 

To compute $K^{\ell,i}_p$ we differentiate using the chain rule. We consider the
intermediate outputs by $h_{\ell}$.} 
{\ptext At every point $\theta$, where $h_{\ell}$  and $ f_{\theta}(\Lambdamap)$  are  differentiable with respect to $\theta$,} we have for 
$\ell' > \ell$ 
\begin{align}
   \left\| \frac{\partial h_{\ell'}}{\partial \theta^{{\ytext  \ell,i}}_p}
           \right\|
   & \le \left\| \frac{\partial h_{\ell'-1}}{
           \partial \theta^{{\ytext  \ell,i}}_p} \right\|
   \left (\| A^{\ell',0}_{\theta} \| + \| B^{\ell',0}_{\theta} \|
   + \| A^{\ell',1}_{\theta} \| + \| B^{\ell',1}_{\theta} \| \right).
\end{align}
Since $h_L$, the output at layer $L$, is the same as
$f_{\theta}(\Lambdamap)$, then iterating the above starting from $\ell' =
L$ down to $\ell$, we obtain
\begin{align}
   \left\| \frac{\partial f_{\theta}(\Lambdamap)}{
                  \partial \theta^{{\ytext  \ell,i}}_p} \right\|
   & \le \left\| \frac{\partial h_{\ell}}{\partial \theta^{{\ytext  \ell,i}}_p}
                 \right\|
   \prod_{\ell' = \ell + 1}^L \left (\| A^{\ell',0}_{\theta} \|
         + \| B^{\ell', 0}_{\theta}  \| +
   \| A^{\ell', 1}_{\theta} \| + \| B^{\ell', 1}_{\theta} \| \right).
\end{align}
{\ptext We recall that the largest singular value $\sigma_1(A)$ of a  matrix $A\in \R^{n\times n}$  satisfies
$\|A\|_{\R^n\to \R^n}\le \sigma_1(A)\le \|A\|_{\mathcal{S}_{1/2}}$.
Using \eqref{cost function and norm}, we find that 
\beq\label{cost function and norm2}
\|  A^{\ell,i,(1)}_\theta \|_{\R^n\to \R^n}+\|  B^{\ell,i,(1)}_\theta \|_{\R^n\to \R^n}
\leq    \mathcal{R}(\theta^{\ell,i}).
\eeq
With this inequality,} we relate the matrix norms $\|A_{\theta}^{\ell',i}\|,
\|B_{\theta}^{\ell',i}\|$ to the
regularization terms $\mathcal{R}(\theta^{\ktext \ell',k})$. 
{\ptext By our assumptions 
\begin{multline} \label{A-inequality}
 { \sum_{i=1}^2} \|A^{\ell',i}_{\theta}\|+\|B^{\ell',i}_{\theta}\|\leq   \sum_{i=1}^2\|A^{\ell',i{\ktext,(0)}}\|+\|A^{\ell',i{\ktext,(1)}}_{\theta}\|+ \|B^{\ell',i{\ktext,(0)}}\|+\|B^{\ell',i{\ktext,(1)}}_{\theta}\|
\\
   \leq
c_0+  \sum_{i=1}^2\|A^{\ell',i{\ktext,(1)}}_{\theta}\|+\|B^{\ell',i{\ktext,(1)}}_{\theta}\|\leq
c_0+\mathcal{R}(\theta^{\ell'}),
\end{multline}
cf. \eqref{norm estimate for fixed}.}
We find that 
\begin{align}
   \left\| \frac{\partial f_{\theta}(\Lambdamap)}{
                    \partial \theta^{{\ytext  \ell,i}}_p} \right\|
   & {\ptext \le 
   \left\| \frac{\partial h_{\ell}}{\partial \theta^{{\ytext  \ell,i}}_p} \right\|
         \prod_{\ell' = \ell + 1}^L ( {\ktext c_0+\mathcal{R}(\theta^{\ell'})})}
\nonumber\\
   &  {\ptext\le 
 {\ptext { c_0^{L-\ell}}}  \left\| \frac{\partial h_{\ell}}{\partial \theta^{{\ytext  \ell,i}}_p} \right\|
         \exp\Big(\sum_{\ell'=\ell+1}^L \mathcal{R}(\theta^{\ell'}) \Big),}
\label{eqn.lipf1bd1}
\end{align}
where we used the simple inequality 
{\ptext $c_0+x \le  c_0e^{x}$ for $x\ge 0$.}
Viewing
$\theta^{{\ytext  \ell,i}}_p$ as a column vector,
\begin{align} \label{eqn.hlderiv2}
    \left\| \frac{\partial h_{\ell}}{\partial \theta^{{\ytext  \ell,i}}_p}
         \right\| & \le \| h_{\ell-1} \| \| \theta^{{\ytext  \ell,i}}_{{\ptext (p)'}} \|,
\end{align}
where {\ptext where ${\ptext (p)'} = p + 1$, if $p$ is odd and $(p)' = p - 1$, if $p$ is even and the weight matrices are written in terms of the parameters as a sum of rank-1 matrices
as given in \eqref{eqn.Cmatrixdecomp}. This means that every
column vector $\theta^{{\ytext  \ell,i}}_p$ is ``paired'' with an adjacent column vector, thus justifying \eqref{eqn.hlderiv2}. For $h_{\ell}$ we find in a similar fashion that
\beq
    \| h_{\ell} \| &\le& \| b_{\theta}^{\ell,0} \|
       + \| b_{\theta}^{\ell,1} \|
       + ( \| A^{{\ytext \ell}, 0}_{\theta}  \|
       + \| B^{{\ytext \ell}, 0}_{\theta}  \|
       + \| A^{{\ytext \ell}, 1}_{\theta} \|
       + \| B^{{\ytext \ell}, 1}_{\theta} \| ) \|h_{\ell-1}\|
\\ \nonumber
       &\le& \big(c_0+\| b_{\theta}^{\ell,0} \|
       + \| b_{\theta}^{\ell,1} \|
\\ \nonumber
& &\qquad {
       + \| A^{{\ytext \ell}, 0,{(1)}}_{\theta}  \|
       + \| B^{{\ytext \ell}, 0,{(1)}}_{\theta}  \|
       + \| A^{{\ytext \ell}, 1,{(1)}}_{\theta} \|
       + \| B^{{\ytext \ell}, 1,{(1)}}_{\theta} \| \big) (1+\|h_{\ell-1}\|)} .
\eeq
Here, when $b_{\theta}^{\ell,i}=\theta^{\ell,i}$ and $A^{\ell,i,(1)}_\theta= \sum_{p=1}^n \theta^{\ell,i}_{{\qtext 2p-1}} (\theta^{\ell,i}_{2p})^T$, $B^{\ell,i,(1)}_\theta = \sum_{p=n+1}^{2n} \theta^{\ell,i}_{{\qtext 2p-1}} (\theta^{\ell,i}_{2p})^T$ { and \eqref{norm estimate for fixed} is satisfied,}
we find that as in \eqref{cost function and norm2} and (\ref{A-inequality}) 
\beq
%
       c_0+{\ytext \| \theta^{\ell,0} \|+\| \theta^{\ell,1} \|+}
       \| A^{{\ytext \ell}, 0,{(1)}}_{\theta}  \|
       + \| B^{{\ytext \ell}, 0,{(1)}}_{\theta}  \|
       + \| A^{{\ytext \ell}, 1,{(1)}}_{\theta} \|
       + \| B^{{\ytext \ell}, 1,{(1)}}_{\theta}\|
       \leq c_0+\mathcal{R}(\theta^{{\ytext \ell}}) ,
       \eeq
       { where we recall that $c_0\ge 1$.}
As the initial vector $h_0$ is in the closed unit ball, using the above 
and that $x \leq e^x$ and $c_0 + x \le c_0 e^x${} , it follows that
\begin{align} \label{eqn.lipf1bd2}
    \| h_{\ell} \| & \le  
{\ktext c_0^\ell   }  \exp\Big(\sum_{\ell'=1}^{\ell} \mathcal{R}(\theta^{\ell'}) \Big).
\end{align}
Using \eqref{eqn.lipf1bd1} and
\eqref{eqn.lipf1bd2}, we therefore find that {\ytext if $K^{\ell,i}_p$ is the 
local Lipschitz constant of $f_{\theta}(\Lambdamap)$ in a neighborhood of
$\theta$ (when considering only $\theta^{\ell,i}_p$ as a variable),}
then 
{\ptext\begin{align}
   {\ytext K^{\ell,i}_p}& \le 
 {\ptext c_0^{L-\ell}}  \left\| \frac{\partial h_{\ell}}{\partial \theta_p^{\ytext \ell,i}} \right\|
   \exp\Big(\sum_{\ell' = \ell + 1}^L 
   \mathcal{R}(\theta^{\ell'}) \Big)
\nonumber\\
   & \le 
           {\ptext c_0^{L-\ell}}   \| h_{\ell-1} \| \| \theta_{{\ptext (p)'}}^{\ytext \ell,i} \|
            \exp\Big(\sum_{\ell' = \ell + 1}^L 
            \mathcal{R}(\theta^{\ell'}) \Big)
\nonumber\\
   & \le 
 {\ptext c_0^{L-\ell}}   {\ptext c_0^\ell} 
            \| \theta_{{\ptext (p)'}}^{\ytext \ell,i} \| \exp \Big(\sum_{\ell' \neq \ell}
            \mathcal{R}(\theta^{\ell'}) \Big)
\nonumber\\
   & \le 
 {\ptext c_0^{L}}   
        \| \theta_{{\ptext (p)'}}^{\ytext \ell,i} \| \exp({\mathcal{R}(\theta)}).
\end{align}}
{\ktext This yields the claim for $p\in P_1$. 

To compute derivatives with respect to bias parameters,
in which case $(\ell,{}i,p)\in P_2$, the result follows similarly { to the above by using \eqref {eqn.lipf1bd1} and replacing}
(\ref{eqn.hlderiv2}) by
\begin{align} \label{eqn.hlderiv2 bias}
    \left\| \frac{\partial h_{\ell}}{\partial \theta^{\ell}_p}
         \right\| & \le {\ytext 1}.
\end{align}}
This completes the proof.}
\end{proof}

\noindent
We point out that the factor $c_0^{L+1}$ in inequality \eqref{c0 bnd} that grows {exponentially} in $L$ is natural as the non-learnable parts of the weight matrices of a neural network $f_\theta$ are linear operators which norms are bounded by $c_0$, see \eqref{norm estimate for fixed}. Hence, even when the trained parts of the weight matrices vanish, each layer of the neural network can increase the Lipschitz constant of the function $f_\theta$ by a multiplicative factor $c_0$.

\subsubsection{Lipschitz constant in $X$ variable}

Obtaining sharp Lipschitz constants for networks is essential to assess their robustness against perturbation in their inputs. Such constants were recently derived for feed-forward neural networks in { \cite{combettes2020}} using advanced tools from nonlinear analysis. Here, we provide an upper bound to the Lipschitz constant for a basic operator recurrent network. This bound will play a role in the forthcoming section on generalization.

\begin{lemma} \label{Lem: Lip in x}
Let the set of parameters or weights $\theta$ belong to $\Theta$ defined in (\ref{eq: parameter space}). Then
the Lipschitz norm of the map $\Lambdamap \to f_\theta(\Lambdamap)$ satisfies
\beq
   \hbox{Lip}(f_\theta) \leq { L} c_0^{L} \exp(\mathcal{R}(\theta)) .
\eeq
\end{lemma}

\begin{proof}
We recall that by \eqref{A-inequality},
\beq \nonumber
& & \sum_{i=1}^2  \|A^{\ell',i}_{\theta}\|+\|B^{\ell',i}_{\theta}\|\leq { \sum_{i=1}^2} \|A^{\ell',i{\ktext,(0)}}\|+\|A^{\ell',i{\ktext,(1)}}_{\theta}\|+ \|B^{\ell',i{\ktext,(0)}}\|+\|B^{\ell',i{\ktext,(1)}}_{\theta}\|
\\ \label{A-inequality copied}
  & &\quad\leq
c_0+ \sum_{i=1}^{2}\|A^{\ell',i{\ktext,(1)}}_{\theta}\|+\|B^{\ell',i{\ktext,(1)}}_{\theta}\|\leq
c_0+\mathcal{R}(\theta^{\ell',i})
\eeq
cf. \eqref{norm estimate for fixed}. As $\Lambdamap \in \mathcal B_{n\times n}$ we have $\|\Lambdamap\| \leq 1$. In the definition of a basic recurrent operator network
(cf.~(\eqref{eqn.NN1def1})-(\eqref{eqn.NN1def2})) we introduced the notation $h_\ell = h_\ell(\Lambdamap) = h_\ell(\Lambdamap;h_{\ell-1})$.

Using \eqref{eqn.lipf1bd2}, we obtain
\begin{align} \label{eqn.lipf1bd2 copied}
    \| h_{\ell} \| & \le
    c_0^\ell \exp\Big(\sum_{\ell'=1}^{\ell} \mathcal{R}(\theta^{\ell'}) \Big).
\end{align}
Moreover,
\ba
\|h_\ell(\Lambdamap_1)-h_\ell(\Lambdamap_2)\|&=&
\|h_\ell(\Lambdamap_1;h_{\ell-1}(\Lambdamap_1))-h_\ell(\Lambdamap_2;h_{\ell-1}(\Lambdamap_2))\|\\[0.45cm]
&=&
\|h_\ell(\Lambdamap_1;h_{\ell-1}(\Lambdamap_1))-h_\ell(\Lambdamap_2;h_{\ell-1}(\Lambdamap_1))\|\\[0.25cm]
& &+\|h_\ell(\Lambdamap_2;h_{\ell-1}(\Lambdamap_1))-h_\ell(\Lambdamap_2;h_{\ell-1}(\Lambdamap_2))\|\\[0.25cm]
&\leq&\left(\sum_{i=0}^1\|B^{\ell,i{\ktext}}_{\theta}\|\right)\|\Lambdamap_1-\Lambdamap_2\|\,
\|h_{\ell-1}(\Lambdamap_1)\|
\\
& &+ \left(\sum_{i=0}^1\|B^{\ell,i{\ktext}}_{\theta}\|\right) \|\Lambdamap_2\|\,
\|h_{\ell-1}(\Lambdamap_1)-h_{\ell-1}(\Lambdamap_2)\|
\\
&\leq& \left(c_0+\sum_{i=0}^1\|B^{\ell,i{\ktext,(1)}}_{\theta}\|\right)\|\Lambdamap_1-\Lambdamap_2\|\,
\|h_{\ell-1}(\Lambdamap_1)\|\\
& &+ \left(c_0 + \sum_{i=0}^1\|B^{\ell,i{\ktext,(1)}}_{\theta}\|\right) \|\Lambdamap_2\|\,
\|h_{\ell-1}(\Lambdamap_1)-h_{\ell-1}(\Lambdamap_2)\|
\\
&\leq& (c_0+\mathcal{R}(\theta^{\ell}))\|\Lambdamap_1-\Lambdamap_2\|\,
{\ktext c_0^{\ell-1}   }  \exp\Bigg(\sum_{\ell'=1}^{\ell-1} \mathcal{R}(\theta^{\ell'}) \Bigg)
\\
& &+(c_0+\mathcal{R}(\theta^{\ell}))
\|h_{\ell-1}(\Lambdamap_1)-h_{\ell-1}(\Lambdamap_2)\|
\\
&\leq&\|\Lambdamap_1-\Lambdamap_2\|\,
c_0^{\ell}   \exp\Bigg(\sum_{\ell'=1}^{\ell} \mathcal{R}(\theta^{\ell'}) \Bigg)
\\
& &+ c_0 \exp(\mathcal{R}(\theta^{\ell}))
 \|h_{\ell-1}(\Lambdamap_1)-h_{\ell-1}(\Lambdamap_2)\|
\ea
We observe that $h_{0}(\Lambdamap_1)-h_{0}(\Lambdamap_2)=0$. Using induction, we see from this  that 
\ba
\|h_\ell(\Lambdamap_1)-h_\ell(\Lambdamap_2)\|&\leq &
\|\Lambdamap_1-\Lambdamap_2\|\,
 { \ell} c_0^{\ell}   \exp\Bigg(\sum_{\ell'=1}^{\ell} \mathcal{R}(\theta^{\ell'}) \Bigg)
\ea
Then the Lipschitz norm of the map $\Lambdamap \to  f_\theta(\Lambdamap)$  satisfies 
\beq\label{a c0 estimate}
\hbox{Lip}(f_\theta)\leq c_0^L  \exp\Bigg(\sum_{\ell'=1}^{L} \mathcal{R}(\theta^{\ell'}) \Bigg)= { L}c_0^L   \exp(\mathcal{R}(\theta)) .
\eeq
\end{proof}

\begin{remark}
We observe that the Lipschitz constant of a truncated neural network ${\bar f}_\theta = T_m \circ f_\theta$ satisfies $\hbox{Lip}({\bar f}_\theta) \leq \hbox{Lip}f_\theta)$. This holds both with respect to the $X$ and $\theta$ variables.
\end{remark}

\begin{remark}
The Lipschitz constant grows with the number of layers $L$. This seemingly indicates that deeper networks are expected to generalize more poorly even though they reduce training error. The responsible factor, $c_0^L$, however, arises from the inequality with \emph{non-learnable} matrices in the network through $\|A^{\ell,i,(0)}\| + \|B^{\ell,i,(0)}\| + |b_{\theta}^{\ell,0}| + |b_{\theta}^{\ell,1}|$ (cf.~(\ref{norm estimate for fixed})) with $c_0 \ge 1$. 
\end{remark}

{
\section{Deep learning and inverse problem from a common statistical viewpoint}

\subsection{Formulation}

The learning problem is finding an approximation to a continuous nonlinear operator function $f$ by an operator recurrent network ${\bar f}_{\theta}$ given training data. We recall that the inverse operator had the form $H = g \circ \boldsymbol{f}$; that is, $H$ stands for $F^{-1}$ on its range, ${\color{black}\mathcal X}$. As laid out in the introduction,
\beq \label{eq: reconstruction}
   z = g(f^{1}(\Lambdamap), f^{2}(\Lambdamap),\ldots,f^{T}(\Lambdamap)),
\eeq
A key objective is to train recurrent operator networks
$f^{t}_\theta$ that approximate the functions $f^{t}$.
To simplify notations, we below drop the index $t$ and consider just one function $f$. 
\subsubsection{Definitions of random variables}

We let $(\Omega,\Sigma,\mathbb{P})$ be a complete probability space and let $\randomZvariable :\ \Omega \to \R^m$ be a random variable that models a random object, 
and $\newrandomXvariable = F(\randomZvariable)$ be a random variable modelling the noiseless measurement obtained from the object $\randomZvariable$
and $\randomYvariable$ be a random variable modelling the intermediate quantity, {\color{black} $\randomYvariable=f(\newrandomXvariable)$. We denote
\beq\label{p map}
p=f\circ F,
\eeq
so that $\randomYvariable=p(\randomZvariable)$.}
Next, we add the measurement error $\mathcal E$ to the  noiseless measurement. To this end, we consider a random variable $\mathcal E:\Omega\to \R^{n \times n}$ having the variance  
\beq
   \Expec \|\mathcal E\|_{\R^{n}\to \R^{n}}^2\le   \Expec \|\mathcal E\|_{\R^{n^2}}^2 ={\color{black}n^2\sigma^2} .
\label{variance sigma2} 
\eeq
 \color{black} Here, $ \|\mathcal E\|_{\R^{n^2}}^2=\sum_{i,j=1}^n|\mathcal E_{ij}|^2$ is the square of the Frobenius norm of the matrix $\mathcal E$.}
The noisy measurement is then defined to be
\beq\label{M map}
   \newrandomMvariable = \newrandomXvariable + \mathcal E .
\eeq
This defines a random variable $\newrandomMvariable :\ \Omega \to \R^{n \times n}$. We assume that $\newrandomXvariable$ and $\mathcal E$ are independent.

We denote by $\pi_\randomZvariable$ the distribution of $\randomZvariable$, by $\pi_\newrandomXvariable = F_*\pi_\randomZvariable$ the distribution of $\newrandomXvariable$, and by $\pi_{\mathcal E}$ the distribution of ${\mathcal E}$.  Also, we define that
\beq\label{def tau}
& &\hbox{$\tau_0$ is the distribution of the pair $(\newrandomXvariable,\randomYvariable)$ and}\\ \nonumber
& &\hbox{$\tau$ is the distribution of the pair $(\newrandomMvariable,\randomYvariable)$.}
\eeq
When $\mathfrak{F}$ is the map $z \to (F(z),p(z))$ and $\widehat{\mathfrak{F}}$ is the map $(z,\e) \to (F(z) + \e, p(z))$, then the distribution $\tau_0$ is given by $\tau_0 = \mathfrak{F}_*\pi_\randomZvariable $ and $\tau$ is given by $\tau = \widehat{\mathfrak{F}}_*(\pi_\randomZvariable \times \pi_{\mathcal E})$.

Our aim below is to approximate the map $\boldsymbol{f}$ by a recurrent operator neural network. We assume that we are given samples of the {\it measurement-property} pairs that are samples of the pair $(\newrandomMvariable ,\randomYvariable)$. We note  that adding noise $\mathcal E$ to the noiseless data $\newrandomXvariable$ gives us noisy data $ \newrandomMvariable$ that may be outside the range ${\color{black}\mathcal X}$ of the direct map $F$ and hence outside the domain of definition of the map $f$. However, the domain of the (trained) neural network is not restricted to the range of the direct map.

To consider neural networks that are defined in a ball of radius $1$, we assume that 
\begin{itemize}
\item [(i)] the distribution $\pi_{\mathcal E}$ of $\mathcal E$ is supported in the ball of radius $\frac 12$, that is, $\mathcal E \in \mathcal B_{n \times n}(\frac 12)$ a.s.,
\item [(ii)]
the distribution $\pi_\randomZvariable$ of $\randomZvariable$ is such that $F_*\pi_\randomZvariable$ is supported in
the ball of radius $\rho_1 = \frac 12$, that is, the noiseless measurements satisfy $\newrandomXvariable=F(\randomZvariable) \in \mathcal B_{n \times n}(\frac 12)$ a.s.,
\end{itemize}
Under these assumptions, $\newrandomMvariable=F(\randomZvariable) + \mathcal E \in \mathcal B_{n \times n}(1)$ a.s.

\subsubsection{Expected loss and regularization}

Given a network with parameters $\theta$, the expected loss { for
noisy and noiseless measurements are} defined by
\begin{equation} \label{Exploss}
   \mathcal{L}(\theta,\tau) =
   \Expec_{(\newrandomMvariable,\randomYvariable) \sim \tau}
   \left[ \mathcal{L}(\theta,\newrandomMvariable,\randomYvariable) \right],
   \quad
   \mathcal{L}(\theta,\tau_0) =
   \Expec_{(\newrandomXvariable,\randomYvariable) \sim \tau_0}
   \left[ \mathcal{L}(\theta,\newrandomXvariable,\randomYvariable) \right]
\end{equation}
(cf.~(\ref{defn.lossfunction1})) and the expected regularized loss for noisy and noise-free measurements are defined by
\begin{equation}
\label{Exploss2}
 \mathcal{L}_{r}(\theta,\tau) =
       \Expec_{(\newrandomMvariable,\randomYvariable)\sim \tau}
                \left[ \mathcal{L}_{r}(\theta, \newrandomMvariable,\randomYvariable) \right] ,
                   \quad
   \mathcal{L}_r(\theta,\tau_0) =
   \Expec_{(\newrandomXvariable,\randomYvariable) \sim \tau_0}
   \left[ \mathcal{L}_r(\theta,\newrandomXvariable,\randomYvariable) \right]
\end{equation}
(cf.~(\ref{defn.lossfunction2})).

We remark that many other regularizers have been studied. For example, regularizers that measure the Lipschitz norm of the neural network with respect to their inputs have
been shown to give good approximations \cite{oberman1}.

{\fmrtext
\subsection{Optimal network subject to sparsity bound}

First we consider the case when it is a priori known that some recurrent operator network approximates the target function $f$ with some reasonable accuracy. We formalize this case in

\begin{definition} \label{def: approximation}
{\it We say that the function $f :\ {\mathcal{X}} \to \R^n$ can be approximated with accuracy $\e_0$, that is, in the range of $F$, by a neural network $f_{\theta_0}$ with sparsity bound $R_0$ if there is $\theta_0 \in (\R^n)^P$ with
\beq \label{R0 condition}
   \mathcal R(\theta_0) =  R_0 ,
\eeq 
such that the neural network ${\bar f}_{\theta_0}$ corresponding to the parameter $\theta_0$ satisfies
\begin{equation}\label{deterministic estimate 2}
 \sup_{\Lambdamap \in \mathcal{X}} \| f(\Lambdamap) - f_{\theta_0}(\Lambdamap) \|_{\R^{\ktext n}} \le  \varepsilon_0.
\end{equation}}
\end{definition}

We observe that when $f_{\theta_0}$ satisfies 
\eqref{deterministic estimate 2}, and $m= \|f\|_\infty$, then the truncated neural network, $\bar f_{\theta_0} = T_m \circ f_{\theta_0}$, satisfies
\begin{equation}\label{deterministic estimate 2 with truncation}
 \sup_{\Lambdamap \in \mathcal{X}} \| f(\Lambdamap) - {\bar f}_{\theta_0}(\Lambdamap) \|_{\R^{\ktext n}} \le  \varepsilon_0.
\end{equation}}

When $f$ satisfies the assumptions of Theorem~\ref{thm.apprx1}, then by Theorem~\ref{thm.apprx2} and inequality \eqref{deterministic estimate 1} we have that \eqref{deterministic estimate 2} holds with $\varepsilon_0 = C' \varepsilon$ and parameters $\theta_0$ that satisfy \eqref{R0 condition} for some value $R_0$. We note that below it is not necessary to assume that \eqref{R0 condition}-\eqref{deterministic estimate 2} hold. Our aim  is to find a neural network ${\bar f}_\theta$ that is a better approximation of $f$ than the neural network ${\bar f}_{\theta_0}$ considered in Definition \ref{def: approximation}.

 \subsubsection{Optimal neural network when the measurements are noise free}

We introduce the following definition of $\theta^*_0$ that minimizes the regularized loss function in the noise-free case

 \begin{definition} {\it The parameters of the optimal network in the noise-free case, $\theta^*_0$, { are a solution of
 \begin{eqnarray} \label{theta star minimization def}
   \theta^*_0 &=&
\underset{\theta\in {\fmtext  {{\Theta}}}}{\operatorname{argmin}}\,
 {\fmrtext\mathcal{L}_{r}}(\theta,\tau_0)
   = \underset{\theta\in {\fmtext  {{\Theta}}}}{\operatorname{argmin}}\, 
   \Expec_{(\newrandomXvariable,\randomYvariable) \sim \tau_0}
 (\|{\bar f}_{\theta}(\newrandomXvariable)-\randomYvariable\|^2 + \alpha \mathcal R(\theta))
 \nonumber\\
    &=&
  \underset{\theta\in { {{\Theta}}}}{\operatorname{argmin}}\, \Expec_{\fmrtext \randomZvariable\sim \pi_{\randomZvariable}}  (\|{\bar f}_{\theta}(F(\randomZvariable))-p(\randomZvariable)\|^2 + \alpha \mathcal R(\theta)).
\eeq
Here,}} {\color{black} ${\bar f}_\theta$ are truncated basic recurrent operator networks of depth $L+2$ with truncation parameter $m\ge \|f\|_\infty$, see \eqref{G net parameter}-\eqref{G net}.}
\end{definition}                            
               
{\color{black} Below, for simplicity, we assume  that 
\beq
m= \|f\|_\infty.
\eeq
In the case when we do not know the norm $ \|f\|_\infty$ but are only given an upper bound $m$ for the norm, all estimates below are valid when $ \|f\|_\infty$ are replaced by $m$.
}
                                
\begin{remark}
Minimizers to \eqref{theta star minimization def} necessarily exist because the loss function is continuous in $\theta$ and the constraint $\mathcal{R}(\theta) \le R_0$ restricts the allowable set to a compact one. The minimizer may not be unique.
\end{remark}
   
\begin{lemma}
The optimal parameter for noise-free measurements, $\theta^*_0$, and the noise-free expected loss satisfy
\beq\label{on optimal performance copied}
 {\fmrtext\mathcal{L}_{r}}(\theta^*_0,{ \tau_0})=
 {\Expec_{(\newrandomXvariable,\randomYvariable) \sim \tau_0}
 (\|{\bar f}_{\theta^*_0}(\newrandomXvariable)-\randomYvariable\|^2)}
+\alpha \mathcal R(\theta^*_0)
\le { \mathcal{K}_0}
\eeq
where
\begin{equation}\label{L0 equation copied} 
{ \mathcal{K}_0}
  :=  \begin{cases}\min({\fmrtext  {4n\|f\|_\infty^2}},\varepsilon_0^2+\alpha R_0),\quad \hbox{if {\qtext \eqref{R0 condition}-\eqref{deterministic estimate 2}} holds,}\\
{\fmrtext{{4n\|f\|_\infty^2}}},\quad \hbox{if {\qtext \eqref{R0 condition}-\eqref{deterministic estimate 2}} does not hold.}\end{cases}
\end{equation}
\end{lemma}

\begin{proof}
If {\qtext \eqref{R0 condition}-\eqref{deterministic estimate 2}} holds, we find that 
\beq
 \mathcal{L}_{r}(\theta_0, \tau_0)=
\Expec_{\newrandomXvariable \sim \pi_\newrandomXvariable}
 (\|{\bar f}_{\theta_0}(\newrandomXvariable)-f(\newrandomXvariable)\|^2
+\alpha \mathcal R(\theta_0)
\le  \varepsilon_0^2+ \alpha R_0.
\eeq
Moreover, both in the case when 
\eqref{R0 condition}-\eqref{deterministic estimate 2} hold or do not hold,
we can take $\theta = 0$, in which case by \eqref{new L infty bound}
{ $\|{\bar f}_\theta\|\leq {n^{1/2}\|f\|_\infty}$ and}
$\mathcal{R}(\theta) = 0$ and thus
\beq \nonumber
{\fmrtext\mathcal{L}_{r}}(\theta,{ \tau_0})&=&\Expec_{\newrandomXvariable \sim \pi_\newrandomXvariable}
 (\|{\bar f}_{\theta}(\newrandomXvariable)-f(\newrandomXvariable)\|^2+\alpha\mathcal{R}(\theta)\\
 &\leq&  
{ \Expec_{\newrandomXvariable \sim \pi_\newrandomXvariable}
 \bigg((\|{\bar f}_{\theta}(\newrandomXvariable)\|+\|f(\newrandomXvariable)\|)^2\bigg)+\alpha\mathcal{R}(\theta)}
 \\ \label{4n inequality 1}
  &\leq& {(n^{1/2}+1)^2 \|f\|^2_{\infty}+0\le 4n \|f\|^2_{\infty} }.
 \eeq
We conclude that
${\fmrtext\mathcal{L}_{r}}(\theta^*_0,{ \tau_0})=\min_{\theta\in \Theta} {\fmrtext\mathcal{L}_{r}}(\theta,{ \tau_0})$  satisfies \eqref {on optimal performance copied}.
\end{proof}

Next we consider how adding the measurement error $\mathcal E$ changes the behavior of the neural network ${\bar f}_{\theta^*_0}$. To this end, we return to considering the random variable $\newrandomMvariable$.

 \subsubsection{Estimate of expected loss for noisy measurements}

Next, we consider the expected loss in the case when measurements contain errors. We recall that adding noise to the data brings us outside the range of the direct map but the domain of the (trained) neural network is not restricted to the range of the direct map.

We introduce the notation, 
\beq \label{L0 equation pre-formula1}
 \mathcal{L}_{r,0}: =
\min\Big( {4n\|f\|_\infty^2},2\varepsilon_0^2+ 2 \alpha  R_0+
2L^2c_0^{2L}   \exp\Big(2\varepsilon_0^2/\alpha + 2R_0 \Big)
 {\cdotp {\color{black}n^2\sigma^2}}  \Big)
\eeq     
if \eqref{R0 condition}-\eqref{deterministic estimate 2} hold, and
\beq\label{L0 equation pre-formula2}
   \mathcal{L}_{r,0} = 4 n \|f\|_\infty^2
\eeq
if \eqref{R0 condition}-\eqref{deterministic estimate 2} do not hold. Sometimes, to indicate the parameter $\alpha$, we denote ${\fmtext\mathcal{L}_{r,0}} = {\fmtext\mathcal{L}_{r,0}}(\alpha)$. We also write
\begin{equation}\label{R0 equation} 
   \mathcal{R}_0 
   :=\frac 1\alpha \mathcal{L}_{r,0},
\end{equation}
that is,
\beq\label{R0 equation pre-formula1}
 \mathcal{R}_0   =\frac 1\alpha \min\Big( {4n\|f\|_\infty^2},2 \varepsilon_0^2+ 2 \alpha  R_0+
2L^2c_0^{2L}   \exp\Big(2\varepsilon_0^2/\alpha + 2R_0 \Big)
 {\cdotp {\color{black}n^2\sigma^2}}  \Big),
\eeq     
if \eqref{R0 condition}-\eqref{deterministic estimate 2} hold, and
\beq\label{R0 equation pre-formula2}
   \mathcal{R}_0  = \frac 1\alpha 4 n \|f\|_\infty^2,
\eeq
if \eqref{R0 condition}-\eqref{deterministic estimate 2} do not hold.


\begin{lemma} \label{lem: adding noise}
The optimal parameters for noise-free measurements, $\theta^*_0$, and the noisy expected loss satisfy
\beq\label{on optimal performance updated}
\mathcal{L}_{r}(\theta^*_0,{\fmrtext \tau})\leq
{\color{black}
 \mathcal{L}_{r,0}.}
\eeq
\end{lemma}

\begin{proof}
{\color{black}First, we consider the case when \eqref{R0 condition}-\eqref{deterministic estimate 2} hold.} We have 
\ba
 {\fmrtext\mathcal{L}_{r}}(\theta^*_0,{\fmrtext \tau}) 
&=&
                \Expec_{ (\newrandomMvariable,\randomYvariable)\sim \tau}(
               \|{\bar f}_{\theta^*_0}(\newrandomMvariable)-\randomYvariable\|^2_{\R^{\ktext n}})
              +\alpha \mathcal R(\theta^*_0)\\
              &=&
         {          \Expec_{ (\newrandomXvariable,{\mathcal E})}(
            \|{\bar f}_{\theta^*_0}(\newrandomXvariable+\mathcal E)-f(\newrandomXvariable)\|^2_{\R^{\ktext n}})} +\alpha \mathcal R(\theta^*_0).
                 \ea
 Equation \eqref{on optimal performance copied} implies that
 \beq\label{on optimal performance implies}
  \mathcal R(\theta^*_0)
\le { \mathcal{K}_0}/\alpha.
\eeq
Furthermore,
\beq \nonumber  
   \Expec_{(\newrandomXvariable,{\mathcal E})}(
               \|{\bar f}_{\theta^*_0}(\newrandomXvariable+\mathcal E)-f(\newrandomXvariable)\|^2_{\R^{\ktext n}})
 &\le& 2 \Expec_{(\newrandomXvariable,{\mathcal E})}(            
\|  {\bar f}_{\theta^*_0}(\newrandomXvariable+\mathcal E)               -{\bar f}_{\theta^*_0}(\newrandomXvariable))\|^2_{\R^{\ktext n}})
\\ \label{new term1}
& &\quad + 2 \Expec_{\newrandomXvariable}(
               \|{\bar f}_{\theta^*_0}(\newrandomXvariable))-
                f(\newrandomXvariable))\|^2_{\R^{\ktext n}}) .
                                \eeq
Using \eqref{on optimal performance copied}, the second term in \eqref{new term1} satisfies the inequality,
\beq
2  \Expec_{\newrandomXvariable}(
               \|{\bar f}_{\theta^*_0}(\newrandomXvariable)-
                f(\newrandomXvariable)\|^2_{\R^{\ktext n}})+2\alpha \mathcal R(\theta^*_0)
\le 2{ \mathcal{K}_0},
\eeq
For the first term in \eqref{new term1}, we observe that as ${\bar f}_{\theta^*_0}:\mathcal B_{n \times n}(1)\to \R^n$
is in the space $C^{0,1} (\mathcal B_{n \times n}(1))$, we have
\beq
\|  {\bar f}_{\theta^*_0}(\newrandomXvariable+\mathcal E) -{\bar f}_{\theta^*_0}(\newrandomXvariable)\|_{\R^{\ktext n}}\le
\|{\bar f}_{\theta^*_0}\|_{C^{0,1}}\|\mathcal E\|.
\eeq
Hence,
\beq        
2 \Expec_{(\newrandomXvariable,{\mathcal E})}(            
\|  {\bar f}_{\theta^*_0}(\newrandomXvariable+\mathcal E)               -{\bar f}_{\theta^*_0}(\newrandomXvariable)\|^2_{\R^{\ktext n}})
\le
2\|{\bar f}_{\theta^*_0}\|_{C^{0,1}}^2
 {\cdotp {\color{black}n^2\sigma^2}} . 
\eeq
Combining these two estimates, we obtain
\beq
  {\fmrtext\mathcal{L}_{r}}(\theta^*_0,{\fmrtext \tau}) 
&\le&2{ \mathcal{K}_0}+
 2\|{\bar f}_{\theta^*_0}\|_{C^{0,1}}^2
 {\cdotp {\color{black}n^2\sigma^2}}  .
\eeq
As $\mathcal{R}(\theta^*_0) \leq \mathcal K_0/ {\alpha}$
we obtain
\beq
  {\fmrtext\mathcal{L}_{r}}(\theta^*_0,{\fmrtext \tau}) &=&2{ \mathcal{K}_0}+
 2\|{\bar f}_{\theta^*_0}\|_{C^{0,1}}^2
 {\cdotp {\color{black}n^2\sigma^2}}  
 \\
 &\le&
2 { \mathcal{K}_0}+ {2L^2c_0^{2L}}
  \exp\Big( 2\mathcal{R}(\theta^*_0) \Big)
 {\cdotp {\color{black}n^2\sigma^2}}, 
\eeq
which proves the claim when
\eqref{R0 condition}-\eqref{deterministic estimate 2} hold.

 Second, {we consider the case when
\eqref{R0 condition}-\eqref{deterministic estimate 2} do not hold.
In this case, as above, 
we take $\theta = 0$ and see similarly to \eqref{4n inequality 1} that
\beq \nonumber
\mathcal{L}_{r}(\theta,\tau)&=&\Expec_{(\newrandomXvariable,\mathcal E)}
 (\|\bar f_{\theta}(\newrandomXvariable+\mathcal E)-f(\newrandomXvariable)\|_{\R^n}^2+\alpha\mathcal{R}(\theta)\\
 &\leq&  
\Expec_{\newrandomXvariable \sim \pi_\newrandomXvariable}
 \bigg((\|\bar f_{\theta}(\newrandomXvariable)\|_{\R^n}+\|f(\newrandomXvariable)\|_{\R^n})^2\bigg)+\alpha\mathcal{R}(\theta)
 \\ \label{4n inequality 2}
  &\leq&  (n^{1/2}+1)^2 \|f\|^2_{\infty}+0\le 4n \|f\|^2_{\infty} .
 \eeq
This proves the claim when
\eqref{R0 condition}-\eqref{deterministic estimate 2} does not hold.}
\end{proof}

\subsubsection{Intrinsic error estimate}

In this section we analyse the intrinsic error,
that is, the 
the expected error that comes from using the optimal (truncated) recurrent operator network
to solve the inverse problem. We consider the case when the data is given with random noise.

\begin{definition}
{\it The intrinsic error for parameters $\theta \in \Theta$ is given by 
\beq \nonumber
\mathcal G_{intrinsic}(\theta)&=&  \mathcal L(\theta,\tau) =  \,
   \Expec_{(\newrandomMvariable,\randomYvariable) \sim \tau}
 \|{\bar f}_{\theta}(\newrandomMvariable)-\randomYvariable\|^2 \\
    &=&  \Expec_{(\newrandomXvariable,{\mathcal E})\sim \pi_{\newrandomXvariable}\times \pi_{\mathcal E}}
            \|{\bar f}_{\theta}(\newrandomXvariable+\mathcal E)-f(\newrandomXvariable)\|^2_{\R^{\ktext n}} .
\eeq}
\end{definition}
\noindent
The optimal recurrent operator network is defined through

\begin{definition}
{\it The optimal parameters for noisy measurements, $\theta^*$, are a solution of 
\begin{eqnarray} \label{theta star minimization}
   \theta^*
   &=&
\underset{\theta\in {\fmtext  {{\Theta}}}}{\operatorname{argmin}}\,
 {\fmrtext\mathcal{L}_{r}}(\theta,\tau)
   = 
\underset{\theta\in {\fmtext  {{\Theta}}}}{\operatorname{argmin}}\,
(\mathcal G_{intrinsic}(\theta)
 + \alpha \mathcal R(\theta)) .  
\eeq
Again, here {\color{black} ${\bar f}_\theta$ are truncated basic recurrent operator networks of depth $L+2$ with the truncation parameter $m= \|f\|_\infty$, see \eqref{G net parameter}-\eqref{G net}.}}
\end{definition}

Our above considerations yield the following

\begin{lemma} 
The optimal parameters for noisy measurements, $\theta^*$, 
and the noisy expected loss satisfy
\beq
     \fmrtext\mathcal{L}_{r}(\theta^*,{\fmrtext \tau})= {          \Expec_{ (\newrandomXvariable,{\mathcal E})\sim \pi_{\newrandomXvariable}\times \pi_{\mathcal E}}(
            \|{\bar f}_{\theta^*}(\newrandomXvariable+\mathcal E)-f(\newrandomXvariable)\|^2_{\R^{\ktext n}})}
               \leq 
\mathcal{L}_{r,0}
\eeq
and  
\beq\label{eq: R0 condition}
 \mathcal{R} (\theta^*)\leq \mathcal{R}_0 .
\eeq
In particular, the intrinsic error with parameter $\theta^*$ satisfies
\beq
\mathcal G_{intrinsic}(\theta^*)
               \leq 
\mathcal{L}_{r,0}.
\eeq
\end{lemma}

{
\begin{proof} 
{As $\theta^*$ satisfies the minimization problem \eqref{theta star minimization},
we see using  Lemma \ref{lem: adding noise} that 
\ba
\mathcal{L}_{r}(\theta^*,{\fmrtext \tau})= \mathcal{L}(\theta^*,\tau) + \alpha \mathcal R(\theta^*) 
\le \mathcal{L}_{r}(\theta^*_0,\tau)\le 
\mathcal{L}_{r,0}
\ea
and $\mathcal R(\theta^*) \leq \mathcal{L}_{r,0}/\alpha=\mathcal R_0$.}
\end{proof}
}

\subsection{Optimal operator recurrent network as Bayes estimator}\label{subsec: Bayes}

In this subsection, we discuss how the optimal neural networks can be considered a Bayesian estimators.

Below, we consider conditional expectations using $\sigma$-algebras. We recall the properties of the conditional expectations in Appendix B.} 
 Let $(\Omega,\Sigma,\Prob)$ be an complete probability space and
 $\newrandomMvariable :\ \Omega \to \R^{n\times n}$ and
  ${\bf y}  :\ \Omega \to \R^{n}$ be  random variables. 
We denote  by $\tau$  the distribution of the pair $(\newrandomMvariable,\randomYvariable)$.

Let $\mathcal{B}_{\newrandomMvariable} \subset \Sigma$  be the $\sigma$-algebra generated by the random variable $\newrandomMvariable :\ \Omega \to \R^{n\times n}$. 
When ${\bf y}\in L^1(\Omega;\Sigma,d\Prob)^n$ is a random variable, 
we denote the conditional expectation of ${\bf y}$ with respect to $\sigma$-algebra $\mathcal{B}_{\newrandomMvariable}$ by $\Expec({\bf y} |\mathcal B_{\newrandomMvariable})$. Roughly speaking, $\Expec({\bf y} \,| \mathcal{B}_{\newrandomMvariable})$ denotes the expectation of  a random variable ${\bf y}$ under the condition that  ${\newrandomMvariable}$ is known and $\newrandomMvariable \to \Expec({\bf y} \,| \mathcal{B}_{\newrandomMvariable})$ can be considered as deterministic, measurable function of $\newrandomMvariable$, see Appendix B. Below, we also use the notation
\baa
   \Expec({\bf y} \,| \mathcal{B}_{\newrandomMvariable})=\Expec({\bf y} \,| \,{\newrandomMvariable}).
\ea
{Below, 
let  $\newrandomMvariable$  and $\randomYvariable$  be as above in \eqref{p map} and \eqref{M map}, that is,
$\newrandomMvariable = F(\randomZvariable)+\mathcal E$
and 
 $\randomYvariable=p(\randomZvariable)$
 where $p=f\circ F.$ Thus, we have that
 $\randomYvariable\in L^\infty(\Omega;\Sigma,d\Prob)^n$
 and $\|\randomYvariable\|_{L^\infty(\Omega;\Sigma,d\Prob)^n}\leq \|f\|_{L^\infty}$.
  Hence,
$T_m(\randomYvariable)=\randomYvariable$ where $T_m$ is the truncation operator
 defined in \eqref{G net} with
 $m= \|f\|_{L^\infty}$.
}

Let $\mathfrak{H} = L^2(\Omega;\mathcal{B}_{\newrandomMvariable},d\Prob)^n$ be the set of $\R^n$-valued functions that lie in $L^2(\Omega;\Sigma,d\Prob)^n$ and
are $\B_{\newrandomMvariable}$-measurable. We note that $\mathfrak{H}$ is a closed subspace of the Hilbert space $L^2(\Omega;\Sigma,d\Prob)^n$. For ${\bf y}  \in L^2(\Omega;{\Sigma},d\Prob)^n$, we define 
\beq
   P_{\mathfrak{H}} {\bf y}  = \ \underset{{\bf u} \in \mathfrak{H}}{\operatorname{argmin}}  \|{\bf y} -{\bf u}\|_{L^2(\Omega;\Sigma,d\Prob)^n}^2 ,
\eeq
which is the orthogonal projector onto the set $\mathfrak{H}$. As discussed in Appendix B,
\beq\label{PH formula}
   P_{\mathfrak{H}} {\bf y}  = \Expec({\bf y} \,| \,{\newrandomMvariable}) .
\eeq
{As 
$\|\randomYvariable\|_{L^\infty(\Omega;\Sigma,d\Prob)^n}\leq m$, 
we see that 
\beq \label{L infty bound}
\|P_{\mathfrak{H}} {\bf y} \|_{L^\infty(\Omega;\Sigma,d\Prob)^n}
=\|\Expec({\bf y} \,| \,{\newrandomMvariable}) \|_{L^\infty(\Omega;\Sigma,d\Prob)^n}
\leq  \|\randomYvariable\|_{L^\infty(\Omega;\Sigma,d\Prob)^n}\leq
m,
\eeq 
and thus $T_m(P_{\mathfrak{H}} {\bf y})=P_{\mathfrak{H}} {\bf y}$, too.}

We now consider the neural networks. {We fix $K = 2n + 1$ 
%
and define the optimal truncated general operator recurrent operator network with depth $L$ and level $K$ to be $\bar f_{ \theta^*_{(L,K)}}(\newrandomMvariable)$ with 
\beq
  \theta^*_{(L,K)}
   = \underset{{\tilde \theta}\in {\tilde \Theta}_{L,K}}{\operatorname{argmin}}
     \, \Expec_{({\bf y},\newrandomMvariable)\sim \tau} (|{\bf y} -{\bar f}_{\tilde \theta}({\newrandomMvariable})|^2),\quad
     {\bar f}_{\tilde \theta}({\newrandomMvariable})=T_m({f}_{\tilde \theta}({\newrandomMvariable})).
\eeq
Observe that here the minimized function is the non-regularized loss function for the truncated general recurrent operator network ${\bar f}_{\tilde \theta}$.

\begin{prop}
Let $n\in \mathbb Z_+$ and $K = 2n + 1$. Then the optimal truncated general operator recurrent operator network with depth $L$ and level $K$, denoted by $\bar f_{ \theta^*_{(L,K)}}=T_m\circ f_{\bar \theta^{(L,K)}}$ satisfies
\beq \label{limit of neural networks}
 \lim_{L\to \infty} \bar f_{ \theta^*_{(L,K)}}(\newrandomMvariable) =\Expec({\bf y} \,| \,{\newrandomMvariable})\quad\hbox{in } L^2(\Omega;\Sigma,d\Prob)^n.
\eeq 
\end{prop}

\begin{proof}
From the functional analytic viewpoint, the conditional expectation is a projector onto a suitable function space, namely $\mathfrak H$ introduced above. Theorem~\ref{not linear operator approximation theorem} will imply that deep operator recurrent networks are dense in this space. We will combine these facts with an analysis of truncated operator recurrent networks to prove the claim.

Let $\mathcal K_{L,K}$ be the space of functions $f_{\tilde \theta}(\newrandomMvariable)$ where $f_{\tilde\theta}$ is a general operator recurrent neural network of depth $L$, level $K$ and width $n$ with $\tilde\theta \in {\tilde \Theta}_{L,K}$. Note that these neural networks are not truncated. We denote $\mathfrak L^2=L^2(\Omega;\Sigma,d\Prob)^n$. Using Theorem~\ref{not linear operator approximation theorem}, we observe that
\beq\label{eq density}
   \mathcal K = \hbox{cl}\left(\bigcup_{L=1}^\infty \mathcal K_{L,K}\right)
\eeq
is equal to $\mathfrak{H}$ with cl the closure in
$\mathfrak L^2$, 
see Remark~\ref{Remark on density}. 

Let ${\bf u}_L\in \mathcal K_{L,K}$
be the nearest point in the set $T_m(\mathcal K_{L,K})$ to ${\bf y}$ and ${\bf u}_\infty\in \mathfrak{H}$
be the nearest point in the set $\mathfrak{H}$ to ${\bf y}$,
that is,
\beq
{\bf u}_L=  Q_{K,L}({\bf y})
   \in  \underset{{\bf u}\in T_m(\mathcal K_{L,K})}{\operatorname{argmin}}
     \, (\|{\bf y} -{\bf u}\|_{\mathfrak L^2}^2),\quad
     {\bf u}_\infty=P_{\mathfrak{H}} {\bf y}.
\eeq
Recall, that by \eqref{L infty bound} we have $P_{\mathfrak{H}} {\bf y}=T_m(P_{\mathfrak{H}} {\bf y})\in T_m(\mathfrak{H}).$

We emphasize that here ${\bf u}_L$ may not be uniquely determined as $\mathcal K_{L,K}$ are not linear subspaces.

As $T_m(\mathcal K_{L,K})\subset \mathfrak{H}$,
we have
\beq
d_L=\|{\bf u}_L-{\bf y}\|_{\mathfrak L^2}\geq 
\|{\bf u}_\infty-{\bf y}\|_{\mathfrak L^2}=d_\infty.
\eeq
On the other hand,  \eqref{eq density} implies that
there are elements ${\bf w}_L\in \mathcal K_{L,K}$
such that ${\bf w}_L\to {\bf u}_\infty$ in ${\mathfrak L^2}$ as $L\to \infty$. 
As ${\bf u}_\infty\in T_m(\mathfrak{H})$,
it is easy to see that 
$\overline{\bf w}_L=T_m\circ {\bf w}_L\in T_m(\mathcal K_{L,K})\subset {\mathfrak H}$ and 
$\overline{\bf w}_L\to {\bf u}_\infty$ in ${\mathfrak L^2}$ as $L\to \infty$. 

Then, 
 $\|\overline{\bf w}_L-{\bf y}\|_{\mathfrak L^2}\to d_\infty$
as $L\to \infty$ and as ${\bf u}_L$ are nearest points in $T_m(\mathcal K_{L,K})\subset {\mathfrak H}$
to ${\bf y}$, we have
$\|\overline{\bf w}_L-{\bf y}\|_{\mathfrak L^2}\ge d_L\ge d_\infty$.
These imply that $d_L\to d_\infty$
as $L\to \infty$. As 
${\bf u}_\infty-{\bf y} \perp \mathfrak{H}$ in
$\mathfrak{L}^2$,
we have
\ba
\|{\bf u}_L-{\bf u}_\infty\|_{\mathfrak L^2}^2+d_\infty^2=\|{\bf u}_L-{\bf u}_\infty\|_{\mathfrak L^2}^2+\|{\bf u}_\infty-{\bf y}\|_{\mathfrak L^2}^2=\|{\bf u}_L-{\bf y}\|_{\mathfrak L^2}^2=d_L^2,
\ea 
the limit $d_L\to d_\infty$ implies that 
$\|{\bf u}_L-{\bf u}_\infty\|_{\mathfrak L^2}\to 0$
as $L\to \infty.$ This shows that
$ Q_{K,L}({\bf y})
  \to  P_{\mathfrak{H}} {\bf y}$ 
as $L\to \infty$. 
This and formula \eqref{PH formula}
yield the claim.
%
%
\end{proof}}

This implies that the optimal general operator recurrent network which minimizes the expected loss, $\mathcal{L}(\tilde \theta,\tau)$ (represented by a formula analogous to (\ref{Exploss}) but with general operator recurrent networks) approximate a Bayes estimator for the inverse problem without regularization. Essentially, the deep neural network, here, is used to parametrize the set of decision rules considered in the Bayes estimator.}

\begin{remark}
 When $f$ is a random function having 
a suitable prior distribution 
it is possible to prove posterior consistency and contraction rates, which give theoretical guarantees that the posterior mean converges to the true solution (determined by $f$) as the amount of data becomes larger and data error tends to zero \cite{abraham2019, giordano2020, monard2020, nickl2017}. Thus one expects $f_{\theta^*_{(L,K)}}$ to approximate $f$. 
\end{remark}


\section{Trained operator recurrent network and generalization}
\label{sec.reg}

We employ the convex function $\mathcal{R}$ as an explicit regularizer in the loss function for training the network, and we show that this regularizer controls the regularity of the resulting local minimizer. This regularizer also provides a form of norm control, which in conjunction with a concentration inequality allows us to produce a generalization bound based on bounding the difference between the expected loss and the empirical loss. Theoretical bounds for generalization and other regularity properties by controlling the norms of parameters have been studied extensively in the literature for neural networks in many different contexts \cite{neyshabur2015, bartlett2017, neyshabur2017, li2018b}. We perform a similar analysis, but still distinct from the above works, since operator recurrent networks are different from standard deep neural networks in an essential way. {To clarify the presentation, we consider only (truncated) basic operator recurrent neural networks
$\bar f_\theta$. The generalization for general operator recurrent neural networks and for the additional layers $g_\theta$ is possible but we omit these details.} 

\subsubsection*{Training data and empirical loss}

The training data is the set 
\beq\label{eq: training data S}
   S = \{ (\Lambdamap_j,y_j)\ ;\ j=1,2,\dots,s\}, 
\eeq
where $s \in \mathbb N$ and $(\Lambdamap_j,y_j)$ are independent samples of the random variable  $(\newrandomMvariable,\randomYvariable)$ having the distribution $\tau$. As discussed above, using the training data $S$ in \eqref{eq: training data S}, our primary aim is to find a recurrent operator network ${\bar f}_\theta:\R^{n \times n} \to \R^n$ that approximates the map  $f:\R^{n\times n} \to \R^n$. {Incorporating the composition with $g$ and finding a neural network with fully connected layers that approximates it, with training data $S' = \{(y_j,z_j)\ ;\ j=1,2,\dots,s\}$, will be addressed at the end of this section.}

For training set $S$ the empirical loss function is given by
\begin{equation} \label{Imploss}
   \mathcal{L}^{(em)}(\theta,S) = \frac{1}{s}\sum_{i=1}^s
          \| {\bar f}_{\theta}(\Lambdamap_i) - \yvariable_i \|^2_{\R^n}
\end{equation}
and the empirical regularized loss function is given by
\begin{equation}
\label{Imploss2}
   \mathcal{L}^{(em)}_r(\theta,S) = \frac{1}{s}\sum_{i=1}^s
          \| {\bar f}_{\theta}(\Lambdamap_i) -{\fmrtext\yvariable}_i \|^2_{\R^{\ktext n}}
                                + \alpha \mathcal{R}(\theta).
\end{equation} 
Here, {${\bar f}_\theta$ are truncated basic recurrent operator networks of depth $L+2$ with truncation parameter $m$, see \eqref{G net parameter}-\eqref{G net}. Below, we assume for simplicity that $m = \|f\|_\infty$.}

Standardly, training is the optimization problem of finding parameters $\theta$, given a training set $S$, such that ${\mathcal{L}^{(em)}_r}(\theta,S)$ is minimized.

\subsection{Optimal neural network for sampled data}

{ In this section we consider neural networks that are truncated.}

As seen above in \eqref{on optimal performance updated},
there are $\theta\in \Theta$ such that 
$\mathcal{L}_{r}(\theta,{\fmrtext \tau}) \le 
    {\fmrtext\mathcal{L}_{r,0}}$
    where { ${\fmrtext\mathcal{L}_{r,0}}$ was defined in \eqref{L0 equation pre-formula1}-\eqref{L0 equation pre-formula2}.
 Thus when we minimize
    ${\fmrtext\mathcal{L}_{r}}(\theta,\tau)$ subject to condition $\theta \in \Theta$,
without loss of generality, we can search only among parameters $\theta \in \Theta$ such that
\begin{equation}
\mathcal{L}_{r}(\theta,{\fmrtext \tau}) \le 
    {\fmrtext\mathcal{L}_{r,0}}.
\end{equation}
We will see later that when the number of training samples, that is, $s$ is large, it is probable that
${\fmtext\mathcal{L}^{(em)}_{r}}(\theta,S)$ is close to ${\fmtext\mathcal{L}_{r}}(\theta,\tau)$. Due to this we 
enforce in the optimization of $\theta$ the constraint ${\fmtext\mathcal{L}^{(em)}_{r}}(\theta,S) \le {\fmtext\mathcal{L}_{r,0}}$ which
automatically enforces the constraint
\beq\label{R0 bound}
   \mathcal{R}(\theta) \le \mathcal{R}_0 , 
\eeq 
where $\mathcal{R}_0$ is defined in \eqref{R0 equation}. This yields the minimization problem with inequality constraint,
\begin{equation} \label{eqn.boundedminproblem}
   \text{find } \theta \text{ minimizing } { \mathcal{L}^{(em)}_{r}}(\theta,S)
   \text{ subject to } \mathcal{R}(\theta) \le \mathcal{R}_0.
\end{equation}
Due to this we introduce

\begin{definition}{\it
{The optimal weights corresponding to the training set $S$, $\theta(S)$, are a solution of the minimization problem
 \begin{eqnarray} \label{eqn.boundedminproblem 2}
   \theta(S) &=&
\underset{\theta\in {\fmtext  \Theta(\mathcal{R}_0})}{\operatorname{argmin}}\,
 \mathcal{L}^{(em)}_{r}(\theta,S),
\eeq
where 
\beq
 \Theta(\mathcal{R}_0)=\{\theta\in  \Theta\ :\  \mathcal{R}(\theta) \le\mathcal{R}_0\}.
\eeq
}}
\end{definition}
We note that when \eqref{R0 bound} holds (see  also \eqref{new L infty bound}), we have 
\beq\label{Key estimate}
{\fmtext\mathcal{L}_{r}}(\theta,\newrandomMvariable,
{\bf y})\leq  (1+n^{1/2})^2\|f\|_\infty^2
+\alpha \mathcal R_0\leq 
\mathcal B_0,\quad \hbox{ for a.e. $(\newrandomMvariable,
{\bf y})\sim \tau$}
\eeq
where
\begin{equation}\label{B0 equation} 
  \mathcal{B}_0 :={9n}\|f\|_\infty^2,
\end{equation}
see formulas \eqref{new L infty bound} and \eqref{R0 equation pre-formula1}-\eqref{R0 equation pre-formula2}.

{We use below the following technical result.

\begin{lemma}\label{lem: R< less thatn 4R0}
If \eqref{R0 condition}-\eqref{deterministic estimate 2} holds and 
\beq \label{alpha bound}
{\color{black}
   \alpha \ge \frac 1{R_0} \Big( \e_0^2 +  L^2 c_0^{2L}   \exp\Big( 4 R_0 \Big) \cdotp n \sigma^2 \Big)}
 \eeq 
then 
\beq \label{Key estimate2}
   \mathcal{R}_0 \leq  4R_0 .
\eeq
\end{lemma}

\begin{proof}
The proof is a straightforward computation: Inequality \eqref{alpha bound} implies that $\alpha\ge \frac 1{R_0}\e_0^2$, 
 so that $\e_0^2/\alpha\le  {R_0}$ and 
$(\varepsilon_0^2 +\alpha R_0)/\alpha\le 2R_0$.
Thus \eqref{alpha bound} implies 
  \ba
4R_0 \alpha &\ge& 2R_0 \alpha +
\bigg(2\e_0^2+2L^2c_0^{2L}   \exp\Big( 2 
(\varepsilon_0^2 +\alpha R_0)/\alpha
 \Big)
\cdotp {\color{black}n^2\sigma^2}  \bigg)\\
&\ge&
\bigg(2\e_0^2+2R_0 \alpha+2L^2c_0^{2L}   \exp\Big( 2 
(\varepsilon_0^2 +\alpha R_0)/\alpha
 \Big)
\cdotp {\color{black}n^2\sigma^2}  \bigg)\ge \mathcal{L}_{r,0}.
 \ea
Hence 
\ba
   \mathcal{R}_0 =\frac 1\alpha  \mathcal{L}_{r,0} \leq {\ytext {\fmrtext 4}R_0}.
\ea
\end{proof}}

For the remainder of this paper, we study (local) minimizers to \eqref{eqn.boundedminproblem} and show how the resulting neural networks, ${\bar f}_{\theta}$, enjoy good approximation properties with respect to the true function $f$. We directly analyze minimizers without studying how to compute them. Typically, the minimization is carried out using variants of stochastic gradient descent. Note that while the architecture of operator recurrent networks differs from that of standard neural networks, gradient descent can still be performed in a straightforward manner with the computation of the gradients via the chain rule. The key difference is that these gradients will contain multiplicative terms with $\Lambdamap$.

\subsubsection*{Selecting the parameter $\alpha$}

We later show that a large value of $\alpha$ leads to greater control over the so-called generalization gap. However, a large value of $\alpha$ also leads to large errors, $\|{\bar f}_{\theta}(\Lambdamap) - f(\Lambdamap)\|^2$, which govern the accuracy of the trained network at approximating the true {\mtext function $f$}. This is due to the fact that with a large $\alpha$, the loss function will be best minimized by reducing the regularization term $\mathcal{R}$ rather than reducing the error.

Cross validation is an established technique for selecting the best predictive model among a set of candidates; see \cite{bishop}. { \color{black}  However, we note that this approach may not be practically applicable for large-scale inverse problems.} It can be employed to choose the smallest value of $\alpha$ that has good generalization properties as follows: Given a training set $S$ and particular $\alpha$, we partition $S$ into equally sized subsets $S_1, S_2, \ldots, S_K$. For each $i = 1, \ldots, K$, we train a network on $S \setminus S_i$ and then evaluate the prediction error on $S_i$. The arithmetic mean of these errors over all $i$ is computed to produce a \emph{cross-validated error}. Then, given a finite set of candidate parameter values $\alpha_1, \alpha_2, \ldots$, the smallest is chosen such that the corresponding cross-validated error is below some tolerance. These techniques have been used, for example, to regularize solutions to linear systems \cite{friedman09}.

\subsection{Sparsity bounds}


Below, we will show that the regularized minimizer will be found in a set $\Theta(N_0,\mathcal R_0)$ that consists of sparse sequences. As the set of sparse sequences is a union of finite dimensional sets, $\Theta(N_0,\mathcal R_0)$ can be covered with a ``relatively small'' number of balls. We will use this and Hoeffding's inequality to obtain improved generalization bounds.

Let $\theta(S)$ be a minimizer that we obtain for
\eqref{eqn.boundedminproblem}.
 We show that $\theta(S)$ enjoys some sparsity bounds which are controlled through the regularizing term $\mathcal{R}(\theta(S))$. We let
\beq\label{cal N estimate}
  & & {\ztext \mathcal N(\theta) := \# \{(\ell,p,{\ztext i{}}) \in P_1\cup P_2:\
         \hbox{vector $\theta_p^{\ell,\ztext i{}}$ is non-zero}\},}\\
         \label{cal N1 estimate}
  & & \mathcal N_1(\theta) := \# \{(\ell,p,{\ztext i{}}) \in P_1:\
         \hbox{vector $\theta_p^{\ell,\ztext i{}}$ is non-zero}\},
\eeq
where $\#A$ denotes the cardinality of the set $A$.

\begin{theorem} \label{thm.sparsity}
{\ytext Let $\theta$  satisfy \eqref{R0 bound}.} Then
\begin{equation} \label{number or elements 2}
   \mathcal N_1(\theta(S)) \le {\ztext \frac{{\qtext {\qtext {}\,{2Lc_0^L}}}\mathcal{R}_0^{3/2}}{{\alpha^{1/2}}}  \exp(2\mathcal{R}_0) \le {\ztext \frac{{\qtext {\qtext {}\,{2Lc_0^L}}}{\fmtext(\mathcal{L}_{r,0})}^{3/2}}{{ \alpha^{2}}}  \exp\left(\frac{2{\fmtext\mathcal{L}_{r,0}}}{\alpha}\right).}}
\end{equation}
\end{theorem}

\begin{proof}
We will use estimates of the directional derivatives to derive
sparsity estimates on the parameters. 
{\fmrtext Let $S = \{ (\Lambdamap_1,{\fmrtext \yvariable}_1),\ldots,
                           (\Lambdamap_s,{\fmrtext \yvariable}_s) \}.$}
At the minimizer $\theta(S)$,
every directional derivative of ${\fmtext\mathcal{L}^{(em)}_{r}}(\theta,S)$ is
non-negative. {\ktext Then we compute the derivative of
  ${\fmtext\mathcal{L}^{(em)}_{r}}(\theta,S)$ with respect to $\theta\in (\R^n)^P$ in
  direction $v$ and obtain}
\begin{equation} \label{eqn.Lgateauxderiv}
   \partial_v {\fmtext\mathcal{L}^{(em)}_{r}}(\theta,S)
        = \frac{2}{s} \sum_{{\ytext j}=1}^s \partial_v {\bar f}_{\theta}{\ytext (\Lambdamap_{\ytext j})} \cdot 
   ({\bar f}_{\theta}(\Lambdamap_{\ytext j}) - {\fmrtext\yvariable_j})
                + \alpha \partial_v \mathcal{R}(\theta).
\end{equation}
At $\theta = \theta(S)$, we must have $\partial_v
{\fmtext\mathcal{L}^{(em)}_{r}}(\theta(S), S) \geq 0$ for every direction $v$, and, hence,
\begin{align} \label{eqn.Rgateauxderivbd1A}
   -\partial_v \left.\mathcal{R}(\theta) \right|_{\theta(S)}
   & \le \frac{2}{s \alpha}
     \sum_{{\ytext j}=1}^s \left. \partial_v {\bar f}_{\theta}{\ytext (\Lambdamap_{\ytext j})} \right|_{\theta(S)}
                  \cdot ({\bar f}_{\theta(S)}(\Lambdamap_{\ytext j}) - {\fmrtext\yvariable_j})
\\
                  & \le \frac{2}{s \alpha}
     \sum_{{\ytext j}=1}^s \left\| \left. \partial_v {\bar f}_{\theta}{\ytext (\Lambdamap_{\ytext j})} \right|_{\theta(S)}\right\|_{\ktext \R^{n}}
                  \, \left\| {\bar f}_{\theta(S)}(\Lambdamap_{\ytext j}) - {\fmrtext\yvariable_j}\right\|_{\ktext \R^{n}}
\nonumber\\
 \label{eqn.Rgateauxderivbd2}
   & \le \frac{2{\fmtext(\mathcal{L}_{r,0})}^{1/2}}{\alpha}
 {\ytext  \left(\frac{1}{s }
     \sum_{{\ytext j}=1}^s
   \left\| \left. \partial_v {\bar f}_{\theta}{\ytext (\Lambdamap_{\ytext j})} \right|_{\theta(S)} \right\|^2_{\ktext \R^{n}}\right)^{1/2}}
\\ \label{eqn.Rgateauxderivbd2B}
   & \le \frac{2\mathcal{R}_0^{1/2}}{\ktext \alpha^{1/2}}
    {\ytext  \left(\frac{1}{s }
     \sum_{{\ytext j}=1}^s
   \left\| \left. \partial_v {\bar f}_{\theta}{\ytext (\Lambdamap_{\ytext j})} \right|_{\theta(S)} \right\|^2_{\ktext \R^{n}}\right)^{1/2}},
\end{align}
where we used H\"{o}lder's inequality.

Next, we derive a relationship between $\partial_v
\left. \mathcal{R}(\theta) \right|_{\theta(S)}$ and the sparsity of
$\theta(S)$. For a given ${\ktext (\ell,{}i,p)\in P}$, for which the
corresponding column vector of $\theta(S)$, denoted by
$\theta(S)^{{\ktext \ell,{}i}}_p$, is nonzero, we consider the
directional derivative with $v = v^{{\ktext \ell,{}i}}_p$ signifying
the unit vector pointing in the direction of $-\theta^{{\ktext
    \ell,{}i}}_p$. Then
\begin{equation} \label{eqn.absvalderiv}
   w^{{\ktext \ell,{}i}}_p := \partial_{v^{{\ktext \ell,{}i}}_p}
         \left. \mathcal{R}(\theta) \right|_{\theta(S)} = -1.
\end{equation}
Therefore, 
\begin{align}
   \mathcal{N}_1(\theta(S)) & \le
   - \sum_{(\ell,{}i,p) \in P_1} w^{{\ktext \ell,{}i}}_p
\nonumber\\
   & \le \frac{2 \mathcal{R}_0^{1/2}}{{\ktext \alpha^{1/2}}} \sum_{\qtext
   i,p:\ (\ell,{}i,p) \in P_1
   }  {\ytext  \bigg(\frac{1}{s}
     \sum_{{\ytext j}=1}^s}
   \big\| \big. \partial_{v^{{\ktext \ell,{}i}}_p} {\bar f}_{\theta} {\ytext (\Lambdamap_{\ytext j})} \big|_{\theta(S)}
          \big\|_{\R^n} \bigg)^{1/2}
\nonumber\\
   & \le \frac{2 \mathcal{R}_0^{1/2}}{{\ktext \alpha^{1/2}}}
                \sum_{ \qtext  i,p:\ (\ell,{}i,p) \in P_1} K^{{\ktext \ell,{}i}}_p,
\end{align} 
where the $K^{{\ktext \ell,{}i}}_p$ are the derivative estimates obtained in Lemma~\ref{lemma.lipf1}. Thus, we have
\begin{align}
   {\ztext\mathcal{N}_1({\ztext\theta(S))}} &
          \le \frac{{\qtext {\qtext {}\,{2Lc_0^L}}}\mathcal{R}_0^{1/2}}{{\ktext \alpha^{1/2}}}
   \exp(\mathcal{R}_0){\qtext \bigg( \sum_{\qtext ( \ell,{}i,p)\in P_1}  \| \theta(S)^{{\ktext \ell,{}i}}_{{\ptext (p)'}} \|_{\R^n}\bigg)}
\nonumber\\
   & \le \frac{{\qtext {\qtext {}\,{2Lc_0^L}}}\mathcal{R}_0^{3/2}}{{\ktext \alpha^{1/2}}}{\qtext   \exp(2\mathcal{R}_0)} .
\end{align}
\end{proof}

Using Theorem~\ref{thm.sparsity} for finding the best parameters $\theta(S)$ given training set $S$, we may solve
\eqref{eqn.boundedminproblem} with a new constraint: When we define
\begin{equation}
   {\ztext N_1} = \left\lfloor 
   {\qtext \frac{{\qtext {}\,{2Lc_0^L}}\mathcal{R}_0^{3/2}}{ \alpha^{1/2} }\exp(2\mathcal{R}_0)}
        \right\rfloor
        =\left\lfloor 
         {\qtext \frac{{\qtext {\qtext {}\,{2Lc_0^L}}}{\fmtext(\mathcal{L}_{r,0})}^{3/2}}{{ \alpha^{2}}}  \exp\left(2\frac{{\fmtext\mathcal{L}_{r,0}}}{\alpha}\right)}  \right\rfloor,
\end{equation}
{without loss of generality, we may consider the minimization problem
\eqref{eqn.boundedminproblem} with the constraint that the parameters $\theta$ satisfy
\ba
   \mathcal{N}_1(\theta) \le N_1 ,
\ea
where $\mathcal{N}_1(\theta)$, defined in \eqref{cal N1 estimate}, is the number of non-zero parameters determining the weight matrices. Thus, we may consider the minimization problem
\eqref{eqn.boundedminproblem} with adding the constraint that the parameters $\theta$ satisfy
\beq\label{eqn.numele3}
     \mathcal{N}(\theta) \le N_0,\quad N_0=N_1+2L+1 ,
\eeq
where $\mathcal{N}(\theta)$, defined in \eqref{cal N estimate}, is the number of non-zero parameters determining the weight matrices and the bias vectors.} Effectively, the size of the set of feasible parameters is further reduced by imposing \eqref{eqn.numele3}. {\ytext We denote by $\Theta(N_0,\mathcal R_0) \subset \Theta$ the set
\beq
{\fmtext \Theta(N_0,\mathcal R_0)}=\{\theta\in {\fmrtext {{\Theta}}}\ :\ \mathcal{N}_1(\theta) \le N_1,\quad \mathcal{R}(\theta) \le \mathcal{R}_0\}
\eeq
Then we {re}define $\theta(S)$ to be a solution of a problem analogous
to \eqref{eqn.boundedminproblem}, where a minimizer is sought in ${\fmtext \Theta(N_0,\mathcal R_0)}$, that is,
\begin{equation} \label{theta star minimization S}
   \theta(S) = \underset{\theta\in {\fmtext \Theta(N_0,\mathcal R_0)}}{\operatorname{argmin}}\,
 {\mathcal{L}^{(em)}_{r}}(\theta,S).
\end{equation}

We now estimate the size of ${\fmtext \Theta(N_0,\mathcal R_0)}$.} We recall that in our original
construction of operator recurrent networks we proposed that there
could be layers with shared parameters. Therefore, we let $L_1 \le L$
represent the number of independently parametrized layers in the
network; in some cases, this quantity may be much smaller than $L$.
Then ${\fmtext \Theta(N_0,\mathcal R_0)} \subset (\R^n)^P$ is given by a finite union of $M_0$
compact subsets of linear subspaces,
\begin{equation} \label{eqn.thetaunion}
   {\fmtext \Theta(N_0,\mathcal R_0)} = \bigcup_{i=1}^{M_0} V_i,
\end{equation}
where
\begin{equation} \label{eqn.M0}
   M_0 = \binom{{\qtext  \# P_1}
   }{N_1}{\ztext \leq  ( \# P_1)^{N_1} \leq ({\corrtext 4nL})^{
       {}\,{2Lc_0^L}\mathcal{R}_0^{3/2} \alpha^{-1/2}  \exp(2\mathcal{R}_0)} },
\end{equation}
where {$\mathcal{R}_0$ was introduced in (\ref{R0 equation})}, and $V_1, V_2, \ldots, V_{M_0}$ are compact subsets of linear subspaces of the full parameter space, such that each $V_i$, $i = 1,2,\ldots,M_0$, has dimension $N_0 n$. Indeed, each $V_i$ consists of those $\theta = (\theta^{{\ktext \ell,{}i}}_p)_{{\ktext \ell,{}i},p}$ for which all such components are zero except those corresponding to {\ztext $N_1$ choices of indices $({\ktext \ell,{}i},p)\in P_1$}, along with the condition that $\mathcal{R}(\theta) \le \mathcal{R}_0$.

We will extensively use the fact that the set $\Theta(N_0,\mathcal R_0) \subset (\R^n)^{P}$ of the form \eqref{eqn.thetaunion} has Hausdorff dimension $nN_0$ which is significantly smaller than $n\,\cdotp (\#P)$. This means that the assumption that $\theta$ is a $nN_0$-sparse vector implies that $\theta$ is in a lower dimensional subset of the parameter space $\R^{nP}$. 
  
In particular, the above means that when regularization parameter $\alpha$ is sufficiently large, we optimize the parameter $\theta$ over a set consisting of sparse vectors. Thus, when $\alpha$ grows, the Hausdorff dimension of the parameter set ${\fmtext \Theta(N_0,\mathcal R_0)}$ (for the optimization problem \eqref{theta star minimization}) becomes smaller. This property is crucial, and we will see below that generalization estimates become stronger when $\alpha$ grows.
  
We have assumed that the parameters $\theta^{\ell,i{}}_p$ with index $({\ell,i{}},p)\in P_1$, that correspond to the weight matrices, are sparse. However, the parameters $\theta^{\ell,i{}}_p$ with index $({\ell,i{}},p)\in P_2$ that correspond to the bias terms, are not assumed to be sparse.
 
We cover the finite union ${\fmtext \Theta(N_0,\mathcal R_0)}$ with a finite set of balls
of radius $\rho$ with respect to the $\mathcal{R}$-norm. This allows
us to further estimate the parameter set ${\fmtext \Theta(N_0,\mathcal R_0)}$ with a discrete,
finite set.

\begin{lemma} \label{lemma.thetafinite}
Let ${\fmtext \Theta(N_0,\mathcal R_0)}$ be the disjoint union of compact sets given in
\eqref{eqn.thetaunion}. Then, 
for every $\rho \in { (0,\mathcal{R}_0)}$, there exists 
a finite set ${\fmtext \Theta_\rho}$ satisfying
\begin{equation} \label{eqn.thetarhosize}
   \#({\fmtext \Theta_\rho}) \le  {\ztext { 3}^{N_0n}} M_0 (\mathcal{R}_0/\rho)^{N_0 n},
\end{equation}
such that for every $\theta \in {\fmtext \Theta(N_0,\mathcal R_0)}$, there exists $\hat{\theta}
\in {\fmtext \Theta_\rho}$ such that
\begin{equation} \label{eqn.thetarhoerror}
   \| {\bar f}_{\theta}(\Lambdamap) - {\bar f}_{\hat{\theta}}(\Lambdamap) \| 
  \le {\ztext {2Lc_0^L}} \rho {\qtext (\mathcal{R}_0+2L) \exp({\qtext 2}\mathcal{R}_0)}
\end{equation}
for any ${\ptext\Lambdamap \in {\fmtext\mathcal B_{n\times n}}}$.
\end{lemma}

\begin{proof}

{
The proof is based on the fact that the set of bounded sparse sequences is a union of bounded finite-dimensional sets that can be covered with a ``relatively small'' number of balls.  

We write $\mathcal I=\{1,2,\dots,N_0\}$. {\ztext For each component $V_i$, $i = 1,2,\ldots,M_0$, in ${\fmtext \Theta(N_0,\mathcal R_0)}$ there is
an isometry $T_i:V_i\to V$, where}
\begin{equation}\label{def V}
   V = \{{\ztext  (x_i)^{N_0}} \in \R^{nN_0 } :
                   \|x\|_{\ell^1(\mathcal I;\R^n)} \le \mathcal{R}_0 \},\quad
                  {\ztext \|x\|_{\ell^1(\mathcal I;\R^n)}=\sum_{i=1}^{ N_0} \|x_i\|_{\R^n} \le \mathcal{R}_0}
\end{equation}
where each $x_i$ is an element of $\R^n$. Let $m=nN_0$.
We call the sets $B_{1,2}^m(x_0,r)=
\{ {\ztext  x \in \R^{m}}: \|x-x_0 \|_{\ell^1(\mathcal I;\R^n)} \le r \}$ the $V$-balls of radius $r$.
Then, $
   V \subset
         B_{1,2}^m(0,{\ytext \mathcal R_0}).
$
Let   $\rho< \mathcal{R}_0$ and $y_{i},$ $i=1,2,\dots,i_0$ be a maximal set of points in 
$
{\ytext 
V}$ such that $ \|y_i-y_{i'} \|_{\ell^1(\mathcal I;\R^n)}{>} \rho $ for $i\not=i'$.
 Then the balls $B_{1,2}^m (y_i,\rho/{2})$ are disjoint and contained in 
$ B_{1,2}^m(0,{\frac 32}{\ytext \mathcal R_0})$. When $v_1$ is the Euclidean volume of the $V$-ball $ B_{1,2}^m(0,1)$ in $\R^m$,
the sum of volumes of the balls $B_{1,2}^m (y_i,\rho/{2})$ is $i_0v_1(\rho/{2})^m$ and this sum
is bounded by $v_1({\frac32}{\ytext \mathcal R_0})^m$. Thus $i_0\leq ({3}{\ytext \mathcal R_0}/\rho)^m$
and $V\subset B_{1,2}^m(0,{\ytext \mathcal R_0})$ can be covered by $i_0$ $V$-balls of radius $\rho$.
Thus the set ${\fmtext \Theta(N_0,\mathcal R_0)}$ can  
be covered by ${ 3}^{N_0n}M_0 (\mathcal{R}_0/\rho)^{N_0n}$  $V$-balls of radius $\rho$ which center are in ${\fmtext \Theta(N_0,\mathcal R_0)}$.
We let
${\fmtext \Theta_\rho}$ be the collection of centers of all such $V$-balls of
radius $\rho$. Then, $$  \#({\fmtext \Theta_\rho}) \le { 3}^{N_0n}M_0 (\mathcal{R}_0/\rho)^{N_0n}.$$

Now, we consider any $\theta = (\theta^{{\ktext \ell,{}i}}_p)_{(\ell,{}i,p)\in P} \in {\fmtext \Theta(N_0,\mathcal R_0)}$,
where $\theta^{{\ktext \ell,{}i}}_p=(\theta^{{\ktext \ell,{}i}}_{p,j})_{j=1}^n\in \R^n$. We see that there exist
$i\in\{1,2,\dots M_0\}$, such that $\theta\in V_i$, and there is $\hat{\theta}
\in {\fmtext \Theta_\rho}\cap V_i$ such that $\|\theta -
\hat{\theta}\|_{V_i}< \rho$. Let ${\theta}^{(q)}$, $q=0,\dots,N_0$
be such that ${\theta}^{(0)}=\theta$, ${\theta}^{(m)}=\hat \theta$,
and when $T_i{\theta}={\eta}=({\eta}_j)_{j=1}^m$, and 
$T_i{\hat \theta}={\hat\eta}=({\hat\eta}_j)_{j=1}^m$,
and $T_i{\theta}^{(q)}={\eta}^{(q)}=({\eta}^{(q)}_j)_{j=1}^m$, we have
\beq\label{extra}
{\eta}^{(q)}_j={\eta}_j\quad\hbox{if }j\le m-q,\\
{\eta}^{(q)}_j={\hat \eta}_j\quad\hbox{if }j> m-q.\nonumber
\eeq 
Let $(\ell_q,i_q,p_q)\in P$ be such that $T_i$  maps the unit vector in $V_i$  corresponding to the coordinate with the index $(\ell_q,i_q,p_q)$ to the unit vector in $V_i$  corresponding to the coordinate with the index $q$. We note that then $\| \theta^{(q+1)} - \theta^{(q)} \|_{\ell^1(\mathcal I;\R^n)}=  \| (\theta^{(q+1)})^{{\ktext \ell_q,i_q}}_{p_q}
   -(\theta^{(q)})^{{\ktext \ell_q,i_q}}_{p_q}\|_{\R^n}$ and
   $$
   {\ytext \sum_{q=0}^{N_0-1} 
   \| \theta^{(q+1)} - \theta^{(q)} \|_{\ell^1(\mathcal I;\R^n)}\leq \rho.}
   $$
We let  $\Lambdamap \in \mathcal B_{n\times n}$ and $J_q = \{s\theta^{(q)}+(1-s)\theta^{(q+1)}\in V_i:\ 0 \le s \le 1\}$. Then, by { Lemma~\ref{lemma.lipf1}},
\begin{align}
   \| {\bar f}_{\theta}(\Lambdamap) & - {\bar f}_{\hat{\theta}}(\Lambdamap) \|_{\R^n}
        \le \sum_{q=0}^{N_0-1}
    \| {\bar f}_{\theta^{(q+1)}}(\Lambdamap) - {\bar f}_{\theta^{(q)}}(\Lambdamap) \| _{\R^n}
    \nonumber\\
   & \le\sum_{q=0}^{N_0-1} \sup_{\theta'\in J_q}\Big\|
     \frac{\partial {\bar f}_{\theta}(\Lambdamap)}{\partial \theta^{ \ell_q{},i_q}_{p_q} }\bigg|_{\theta=\theta'}\Big\| _{\R^n\to \R^n}   
   \,\cdotp   \| (\theta^{(q+1)})^{{\ktext \ell_q,i_q}}_{p_q}
   -(\theta^{(q)})^{{\ktext \ell_q,i_q}}_{p_q}\|_{\R^n}
\nonumber\\
   & \le \bigg(\sup_{\theta'\in {\fmtext \Theta(N_0,\mathcal R_0)}} \sum_{({\ktext \ell,{}i},p)\in P_1\cup P_2} \Big\|
     \frac{\partial {\bar f}_{\theta}(\Lambdamap)}{\partial \theta^{{\ktext \ell,{}i}}_p}\bigg|_{\theta=\theta'} \Big\|_{\R^n\to \R^n} 
   \bigg)\,\cdotp \sum_{q=0}^{N_0-1} 
   \| \theta^{(q+1)} - \theta^{(q)} \|_{\ell^1(\mathcal I;\R^n)}
   \nonumber\\
   &{\ztext \le \sup_{\theta\in \ytext {\fmtext \Theta(N_0,\mathcal R_0)}}\bigg(\sum_{({\ktext \ell,{}i},p)\in P_1}
   {\ptext {\qtext {2Lc_0^L}}} \| \theta^{{\ktext \ell,{}i}}_{{\ptext (p)'}} \|
                              \exp(\mathcal{R}(\theta))+
                              \sum_{({\ktext \ell,{}i},p)\in P_2}
   {\ptext {\qtext {2Lc_0^L}}} \exp(\mathcal{R}(\theta))
         \bigg) \rho}
         \nonumber\\ \label{Lip estimate}
   & \le {\ztext {2Lc_0^L}} \rho {\qtext (\mathcal{R}_0+2L) \exp({\qtext 2}\mathcal{R}_0)}.
\end{align}
This proves the claim.}
\end{proof}

We  point out that selecting the finite set 
$\Theta_\rho$
means selecting $\rho>0$, or conversely, selecting $\rho$ means selecting $\Theta_\rho$. Hence,
selecting a different $\theta$ means using a different 
finite set  $\Theta_\rho$.
Below, we minimize loss functions over $\theta\in  \Theta_\rho$ in the proofs of the relevant lemmas and theorems, but the set $\Theta_\rho$ is used only as an auxiliary tool so that in the proofs the minimization can be done over a finite set. A suitable value for the parameter $\rho$ is later
chosen in formula \eqref{rho def}.

{\ztext \begin{remark} If $L_0$ is the total number of layers and
    $L_1$ the number of independent layers, then in the above estimates
    $L$ is replaced by $L_1$ in Lipschitz estimates and in
    \eqref{eqn.thetarhoerror}, and by $L_0$ in the definition of $N_0$ and
    in \eqref{eqn.thetarhosize}.
\end{remark}}

\subsection{Generalization}
\label{sec.gen}

In this subsection, we study the probability that a neural network optimized under our regularized empirical loss function can approximate the map $f$. Given a training set $S$, we therefore study the {generalization} error
\beq
   \mathcal{G}(S)
     := |
     {\mathcal{L}^{(em)}_{r}}(\theta(S), S) - 
     \mathcal{L}_{r}(\theta(S),{\fmrtext \tau})|
   = |
       {\mathcal{L}^{(em)}}(\theta(S), S) - \mathcal{L}(\theta(S),\tau)|.
\eeq
Given that the parameters $\theta(S)$ have been computed,
$\mathcal{G}(S)$ measures the difference between the expected loss ${\fmtext\mathcal{L}}(\theta(S),{\fmrtext \tau})$ and the empirical loss ${ \mathcal{L}^{(em)}}(\theta(S),S)$. If a model overfits the data, the empirical loss is very small while the expected loss remains large. Thus an upper bound on $\mathcal{G}(S)$ provides some control over the degree of overfitting that is possible.

Considered as a random variable in $S$, we estimate the probability
that $\mathcal{G}(S)$ is small using the following well-known
inequality 

\begin{lemma}[Hoeffding's inequality \cite{hoeffding1}]
\label{Hoeffding inequ}
Let $Z_1,\ldots,Z_N$ be $N$ i.i.d. copies of the random variable $Z$
whose range is $[0,{\fmtext Z_{\mathrm{max}}}]$, ${\fmtext Z_{\mathrm{max}}} > 0$. Then, for $0 <
\delta < \min(\Expec [Z],{\fmtext Z_{\mathrm{max}}} - \Expec [Z])$, we have
\begin{equation}
   \Prob\Big[\Big|\frac{1}{N}
       \Big(\sum_{i=1}^N Z_i\Big) - \Expec[Z]\Big| \leq
    \delta \Big] \geq 1 - 2 \exp(-2 N \delta^2 {\fmtext Z_{\mathrm{max}}}^{-2}) .
\end{equation}
\end{lemma}

To apply Hoeffding's inequality, one requires independent random
variables. However, the optimal parameters $\theta(S)$ depend on every
element of the training set $S$. Thus we use
Lemma~\ref{lemma.thetafinite} to apply Hoeffding's inequality on each
element of ${\fmtext \Theta_\rho}$, and then use the fact that $\theta(S)$ is sufficiently close to at least one element of ${\fmtext \Theta_\rho}$.   
{We recall  that $\rho$ is the radius of the finite set of balls whose union cover $\Theta(N_0,\mathcal R_0)$. At this moment we keep $\rho$ as a free parameter, and will fix
its value later in formula \eqref {rho def}.} This leads to the main generalization
result

\begin{theorem}\label{generalization gap}
Let ${\bar f}_\theta$ be a truncated basic operator recurrent network  with truncation parameter $m= \|f\|_\infty$, of with $n$ and depth $L$. Consider a {\fmtext random} training set ${\fmtext {\bf S} }$ consisting of $s$ independent samples from {\fmrtext distribution $\tau$}, and let $\theta(S)$ be a minimizer for \eqref{theta star minimization}. Then,  

(i) For {\yytext any $\alpha$ and}  every sufficiently small
$\delta > 0$,
\begin{equation} \label{eqn.maingenbd}
   \Prob_{{\fmtext {\bf S} } \sim{\fmrtext \tau^s}} \left[ \mathcal{G}({\fmtext {\bf S} }) \le 2\delta \right]
   \geq 1 - {\qtext C_1}\left( \frac{1}{\delta} \right)^{\qtext C_2} 
         \exp\left(-\frac 2{({\zztext {\zztext {\fmrtext 9n}}})^2\|f\|_\infty^4} s \delta^2\right),
\end{equation}
{{\xxtext where  
\beq
& &\hspace{-1cm}C_1=
\exp\bigg(164n^{3/2}L^2c_0^{L+1}(1+\|f\|_{\infty})  \exp(4{\fmtext\mathcal{L}_{r,0}}\alpha^{-1})\bigg)
,\hspace{-1cm}\\
\label{C1 and C2 in case (i)}
& &\hspace{-1cm}C_2=
4nLc_0^L \exp(3\mathcal{L}_{r,0}\alpha^{-1})
\eeq
and ${\fmtext\mathcal{L}_{r,0}}\leq {\fmrtext {4n\|f\|_\infty^2}}$, cf. \eqref{L0 equation pre-formula1}-\eqref{L0 equation pre-formula2}.

(ii) Let the function $f$ 
be approximated  with accuracy $\e_0$ by some neural network ${\bar f}_{\theta_0}$,
where ${\theta_0}\in {\fmtext \Theta(N_0,\mathcal R_0)}$ has sparsity bound $R_0\ge 1$;
that is, conditions \eqref{R0 condition}-\eqref{deterministic estimate 2}
hold with $R_0$ and $\e_0$.
Then, for all 
 \beq\label{condition: alpha}
 {\alpha\ge \frac 1{R_0} \Big(\e_0^2+{2L^2c_0^{2L}}   \exp\Big( 4R_0
 \Big)
 {\cdotp {\color{black}n^2\sigma^2}}  \Big)}
\eeq 
the inequality \eqref{eqn.maingenbd} holds, { where
\beq\label{Key estimate implication}
C_1
 &\leq &
  2\exp\Big(
n^{3/2} 2^{10} (1+\|f\|_{\infty})
(1+R_0)L^3c_0^{L+1}e^{8R_0}(1+
R_0^2\alpha^{-1/2})\Big),
\\
C_2
 &\leq &
16nLc_0^Le^{8R_0}(1+
R_0^2\alpha^{-1/2}),
 \eeq
}}}
\end{theorem}

\noindent
{\color{black} Note that when the depth $L$ grows, the set of functions that the neural networks can represent has an increasingly richer structure and this is reflected by the growths of $C_1$ and $C_2$. Naturally $s$ has to increase appropriately to mitigate this growth.}

We also observe that in claim (i), making the regularization parameter $\alpha$ larger (that is, forcing the weight matrices to be sparser) makes the probability in \eqref{eqn.maingenbd} larger, but then the error how well the neural network approximates the function $f$ becomes larger.

\begin{proof}
{\color{black} The main lines of the proof are the following: The truncated neural networks are bounded, so for each neural network $\bar f_\theta$ we can use Hoeffding's inequality. Moreover, the empirical optimizer $\theta(S)$ will be in the set $\Theta(N_0,\mathcal R_0)$ with large probability. We use a suitable value $\rho$ which balances approximating an arbitrary element $\theta \in \Theta(N_0,\mathcal R_0)$ by an element in $\Theta_\rho$ and the number of elements in the set  $\Theta_\rho$. Applying Hoeffding's inequality to for all $f_\theta$, $\theta \in \Theta_\rho$ will finalize the proof.} 

Fix $\theta \in \Theta(N_0,\mathcal R_0)$. {When ${\bf S} =((\newrandomMvariable_i,\randomYvariable_i))_{i=1}^s$ is a sequence of $s$  independent random samples from distribution $\tau$, see \eqref{def tau}.}
 We define the random variable
\begin{equation}
   Z_i = {\fmtext\mathcal{L}_{r}}(\theta,{\fmrtext\newrandomMvariable_i,\randomYvariable_i}).
\end{equation}
The set of $Z_i$, $i = 1,\ldots,s$, consists of i.i.d. copies of the
random variable
\begin{equation}
   Z{\fmtext =Z(\newrandomMvariable
   ,{\fmrtext \randomYvariable}
   )}:= {\fmtext\mathcal{L}_{r}}(\theta,\newrandomMvariable  ,{\fmrtext \randomYvariable}),
\end{equation}
where $(\newrandomMvariable,\randomYvariable)$ is distributed according to the {\mtext probability
  distribution $\tau$.} The empirical loss is given by
\begin{equation}
  {\mathcal{L}^{(em)}_r}(\theta,{\fmtext {\bf S} }) = \frac{1}{s} \sum_{i=1}^s Z_i
\end{equation}
and, by definition, the expected loss is
\begin{equation}
\label{eqn.expectedlossZ}
   {\fmtext\mathcal{L}_{r}}(\theta,{\fmrtext \tau}) = \Expec_{(\newrandomMvariable  ,{\fmrtext \randomYvariable})} {\fmtext [Z(\newrandomMvariable  ,{\fmrtext \randomYvariable})].}
\end{equation}
Since we assumed that ${\bar f}_{\theta}$ is a truncated network, we have
by \eqref{Key estimate} that  $0\le Z \le \mathcal{B}_0$; therefore, by Hoeffding's inequality with
${\fmtext Z_{\mathrm{max}}}=\mathcal B_0$, we have that
\begin{equation} \label{eqn.hoeffding1}
   \Prob\left[ |{\mathcal{L}^{(em)}_r}(\theta,{\fmtext {\bf S} })
             - {\fmtext\mathcal{L}_{r}}(\theta,{\fmrtext \tau})| \le \delta \right]
   \ge 1 - 2 \exp(-2 s \delta^2 \alpha^{-2} {\ytext \mathcal{B}_0^{-2}}).
\end{equation}
In particular, \eqref{eqn.hoeffding1} holds for every element of
${\fmtext \Theta_\rho}$. Since
\begin{equation}
\label{eqn.cardofthetarho}
    \#({\fmtext \Theta_\rho}) \le  {\ztext { 3}^{N_0n}} M_0 (\mathcal{R}_0/\rho)^{N_0 n},
\end{equation}
it follows that
\beq  \label{Hoeffding times finite}
& &\hspace*{-1.5cm}\Prob\left[ \forall \theta \in {\fmtext \Theta_\rho} :
   |{\mathcal{L}^{(em)}_r}(\theta, {\fmtext {\bf S} }) - {\fmtext\mathcal{L}_{r}}(\theta,{\fmrtext \tau})| \le \delta \right]
\\[0.25cm] \nonumber
&\ge &{1-\sum_{ \theta \in \Theta_\rho }
   \Prob\left[
   |{\mathcal{L}^{(em)}_r}(\theta, {\fmtext {\bf S} }) - {\fmtext\mathcal{L}_{r}}(\theta,{\fmrtext \tau})| > \delta \right]}
\\ \nonumber
  & \ge &1 - 2\,\cdot  {\ztext { 3}^{N_0n}} M_0 (\mathcal{R}_0/\rho)^{N_0 n}
          \exp(-2s\delta^2 \alpha^{-2}{\ytext \mathcal{B}_0^{-2}}).
\eeq
Furthermore, {in view of \eqref{eqn.thetarhoerror}}, for every $\theta \in {\fmtext \Theta(N_0,\mathcal R_0)}$ there exists $\hat{\theta}
\in {\fmtext \Theta_\rho}$ such that for any \newline ${\ptext\Lambdamap \in B^{n\times
    n}_L(1)}$,
\begin{equation}\label{Q estimate0}
    \| {\bar f}_{\theta}(\Lambdamap) - {\bar f}_{\hat{\theta}}(\Lambdamap) \| 
    \le  {\ztext {}\,{2Lc_0^L}}\rho {\qtext (\mathcal{R}_0+2L) \exp({\qtext 2}\mathcal{R}_0)}.
\end{equation}
{\xxtext Using this estimate and (\ref{new L infty bound}), we find that for any $\Lambdamap \in {\fmtext\mathcal B_{n\times n}}$ and ${\fmtext y\in B_n(1)}$ we have
\beq
& &\hspace*{-2.0cm}
   |\| {\bar f}_{\theta}(\Lambdamap) - y \|^2
   - \| {\bar f}_{\hat{\theta}}(\Lambdamap) - y \|^2|
\nonumber\\
&\leq&  \nonumber
|  ( \| {\bar f}_{\theta}(\Lambdamap) -  {\fmtext y}\|+
     \|   {\bar f}_{\hat{\theta}}(\Lambdamap) - {\fmtext y}\|) \,\cdotp
      ( \| {\bar f}_{\theta}(\Lambdamap) - {\fmtext y}\|-
     \|   {\bar f}_{\hat{\theta}}(\Lambdamap) -  {\fmtext y}\|)| 
    \\ \nonumber
    &  \le& 2(1+n^{1/2})\|f\|_{\infty}  {\ztext {}\,{2Lc_0^L}}\rho {\qtext (\mathcal{R}_0+2L) \exp({\qtext 2}\mathcal{R}_0)}
    \\
    &  \le& 4n^{1/2}(1+\|f\|_{\infty})  {\ztext {}\,{2Lc_0^L}}\rho {\qtext (\mathcal{R}_0+2L) \exp({\qtext 2}\mathcal{R}_0)}.
\label{Q estimate}
\eeq
Below, we denote $Q=4n^{1/2}(1+\|f\|_{\infty})$.} 

We next consider the implications of the above estimates
when $\rho$ has the value
\begin{equation}\label{rho def}
   \rho = \frac{\delta}{2 {\xxtext Q}\,\cdotp{\ztext {}\,{2Lc_0^L}} {\qtext (\mathcal{R}_0+2L) \exp({\qtext 2}\mathcal{R}_0)}}.
\end{equation}
Then, for any $S$, there exists $\hat{\theta}{\color{black}=\hat{\theta}(S)}
\in {\fmtext \Theta_\rho}$ such
that
\begin{equation}
\label{eqn.generalizationineq1}
   \mathcal{G}(S) \le |{\mathcal{L}_r}(\hat{\theta},{\fmrtext \tau})|
                 + {\xxtext Q} \,\cdotp{\ztext {}\,{2Lc_0^L}} \rho {\qtext (\mathcal{R}_0+2L) \exp({\qtext 2}\mathcal{R}_0)}.
\end{equation}
When we apply this observation for randomly chosen samples
${\bf S}$, we obtain that
\begin{multline}
   \Prob\left[ \mathcal{G}({\fmtext {\bf S} }) \le {\xxtext {\delta}}
         + {\xxtext Q} \,\cdotp{\ztext {}\,{2Lc_0^L}} \rho {\qtext (\mathcal{R}_0+2L) \exp({\qtext 2}\mathcal{R}_0)} \right]
\\
    \ge 1 - 2\,\cdot   {\ztext { 3}^{N_0n}} M_0 (\mathcal{R}_0/\rho)^{N_0 n}
        \exp(-2 s  {\xxtext {\delta}^2} \alpha^{-2} {\ytext \mathcal{B}_0^{-2}}),
\end{multline}
We substitute our expressions for {\qtext $\mathcal{R}_0$}, $N_0$ and $M_0$ to obtain 
{\ztext the estimate
\begin{equation}
   \Prob\left[ \mathcal{G}({\fmtext {\bf S} }) \le 2\delta \right]
    \ge 1 - {\yytext C_0}
        \exp(-2 s \delta^2 \alpha^{-2} {\ytext \mathcal{B}_0^{-2}}),
\end{equation}
where 
\ba
{\yytext C_0}&=&2 \,\cdotp{\ztext { 3}^{N_0n}} M_0 (\mathcal{R}_0/\rho)^{nN_0 }\\
\commentedinfinal{&\leq &2 M_0 \bigg(16\mathcal{R}_0 
\frac { {\ztext {}\,{2Lc_0^L}}  {\xxtext Q}(\mathcal{R}_0+2L) \exp( 2\mathcal{R}_0)}{\delta}
\bigg)^{nN_0 }\\}
&\leq &2 M_0 \bigg(
\frac { \,{3 \cdotp 4L c_0^{L}}Q\exp( { 4} (\mathcal{R}_0+2L))}{\delta}\bigg)^{nN_0 }.
\ea
{\yytext As 
\ba
M_0&\leq& (4nL)^{N_0},
\ea
we have
\ba
{\yytext C_0}&\leq &2 M_0 \bigg(
\frac { \,{12L c_0^{L}}Q\exp( { 4} (\mathcal{R}_0+2L))}{\delta}\bigg)^{nN_0 }
\\
&\leq &2  \bigg(4nL
\frac {{(12LQ)^n c_0^{nL}}  \exp( { 4}n (\mathcal{R}_0+2L))}{\delta^n}\bigg)^{N_0}.
\ea
Using that 
\ba
N_0
&\leq&
{2Lc_0^L}
(1+\mathcal{R}_0^{3/2} \alpha^{-1/2})  \exp(2\mathcal{R}_0),
\ea
we obtain the estimate 
\beq\nonumber
{\yytext C_0}&\leq &
2  \bigg(4nL
\frac {{(12LQ)^n c_0^{nL}}  \exp( { 4}n (\mathcal{R}_0+2L))}{\delta^n}\bigg)^{N_0}
\\ \label{good C0 estimate}
&\leq &
2  \bigg(4nL
\frac {{(12LQ)^n c_0^{nL}}  \exp( { 4}n (\mathcal{R}_0+2L))}{\delta^n}\bigg)^{{2Lc_0^L}
(1+\mathcal{R}_0^{3/2} \alpha^{-1/2})  \exp(2\mathcal{R}_0)}.
\eeq

For claim (i), we can use the facts that  $\mathcal{R}_0=\alpha^{-1}{\fmtext\mathcal{L}_{r,0}}$  
and $\mathcal{L}_{r,0} \le 4 n \|f\|_\infty^2$, so that 
\beq \nonumber 
C_0 &\leq &2  \bigg(4nL
\frac { {(12LQ)^n c_0^{nL}} \exp( { 4}n ( {\fmtext\mathcal{L}_{r,0}}\alpha^{-1}+2L))}{\delta^n}\bigg)^{{2Lc_0^L}(1+
{\fmtext(\mathcal{L}_{r,0})}^{3/2}\alpha^{-2})  \exp(2{\fmtext\mathcal{L}_{r,0}}\alpha^{-1})}
\\ \nonumber
   &\leq &2 \bigg(\delta^{-n} \exp(
   3 + n L + n Q + n L c_0
\\ \nonumber & &\hspace{4.0cm}   
   + 4 n ( \mathcal{L}_{r,0} \alpha^{-1} + 2 L)}) \bigg)^{{2Lc_0^L}(1+
{\fmtext(\mathcal{L}_{r,0})}^{3/2}\alpha^{-2})  \exp(2{\fmtext\mathcal{L}_{r,0}}\alpha^{-1})
\\ \nonumber
&\le& C_1\left( \frac{1}{\delta} \right)^{C_2} ,
\eeq
where, {
 using that $ {\fmtext\mathcal{L}_{r,0}}\ge 1$, $Q\ge 1$, $c_0\ge 1$, and $\alpha\le 1$,
\ba
   C_1 &=& 2 \exp[
   (3 + n L + n Q + n L c_0
\\
& &\hspace*{2.2cm}
   + 4n (\mathcal{L}_{r,0}\alpha^{-1}+2L))
\cdot {{2Lc_0^L}(1+
{\fmtext(\mathcal{L}_{r,0})}^{3/2}\alpha^{-2})  \exp(2{\fmtext\mathcal{L}_{r,0}}\alpha^{-1})}]
\\
&\le &
2\exp[
20n( 1+Q+c_0L+\mathcal{L}_{r,0}\alpha^{-1})
\cdot Lc_0^L(1+
(\mathcal{L}_{r,0})^{2}\alpha^{-2})  \exp(2{\fmtext\mathcal{L}_{r,0}}\alpha^{-1})]
\\
&\le &
2\exp[
40nQc_0L(1+\mathcal{L}_{r,0}\alpha^{-1})
\cdot Lc_0^{L}(1+\tfrac 12
(\mathcal{L}_{r,0})^{2}\alpha^{-2})  \exp(2{\fmtext\mathcal{L}_{r,0}}\alpha^{-1})]
\\
&\le &
\exp[
41nQL^2c_0^{L+1}  \exp(4{\fmtext\mathcal{L}_{r,0}}\alpha^{-1})]
\\[0.45cm]
C_2&=&n\cdotp 2Lc_0^L(1+(\mathcal{L}_{r,0})^{3/2}\alpha^{-2})  \exp(2\mathcal{L}_{r,0}\alpha^{-1})
\\
&\le &4nLc_0^L(1+\frac 12 (\mathcal{L}_{r,0})^{2}\alpha^{-2})  \exp(2\mathcal{L}_{r,0}\alpha^{-1})
\\
&\le &4nLc_0^L \exp(3\mathcal{L}_{r,0}\alpha^{-1}).
\ea
}}
This proves the claim (i).

{Now, we consider claim (ii). If \eqref{R0 condition}-\eqref{deterministic estimate 2} hold, and $\alpha$ satisfies the assumption in claim (ii), {we have by Lemma \ref{lem: R< less thatn 4R0} that}
\beq\label{Key estimate 2}
 \mathcal{R}_0 \leq  4R_0.
\eeq 
Hence, we find, using \eqref{good C0 estimate}
and  $2R_0\ge 1$ and $\alpha\leq 1$, that \ba
C_0&\leq&
2 [\delta^{-n}\exp(
3+nL+nQ+nLc_0+4n ( 4R_0+2L))]^{2Lc_0^L(1+
(4R_0)^{3/2}\alpha^{-1/2})  \exp(2\cdot 4R_0)}\\
\nonumber
&\leq&
2 [\delta^{-n}\exp(
3+9nL+nQ+nLc_0+16R_0)]^{16Lc_0^L(1+
R_0^2\alpha^{-1/2})  \exp(8R_0)}\\
&\leq&
2\exp[
n 2^{8}Q(1+R_0)L^3c_0^{L+1}e^{8R_0}(1+
R_0^2\alpha^{-1/2})]
\delta^{-16nLc_0^Le^{8R_0}(1+
R_0^2\alpha^{-1/2}) }
\ea
Thus, we obtain
\ba
C_1
 &\leq &
  2\exp[
n 2^{8}Q(1+R_0)L^3c_0^{L+1}e^{8R_0}(1+
R_0^2\alpha^{-1/2})] , \\
C_2
 &\leq &
16nLc_0^Le^{8R_0}(1+
R_0^2\alpha^{-1/2}) .
\ea
This completes the proof of claim (ii).}}
\end{proof}

Estimate \eqref{eqn.maingenbd} quantifies the effect on the
generalization error from varying the values of the regularization
parameter $\alpha$ and the sample size $s$. 
Note that \eqref{eqn.maingenbd} approaches $1$ exponentially fast with
respect to increasing $s$. On the other hand, with increasing $L$
the expressivity of the network also rapidly increases, so
one may thus expect that the sample size $s$ will need
to increase accordingly in order to maintain a good generalization bound.
Indeed, \eqref{eqn.maingenbd} decreases super-exponentially away from $1$
as $L$ increases.
Similarly, increasing the regularization parameter $\alpha$ also
reduces the generalization error, as it decreases the variance in the
loss function. However, increasing $\alpha$ to improve the
generalization competes with the goal of accurately approximating the
true function. Furthermore, when \eqref{R0
  condition}-\eqref{deterministic estimate 2} hold then the lower
bound $\alpha\ge \e_0^2/R_0$ indicates when there is sufficient
regularization. Additionally, a suitable value for the error lower
bound $\delta$ can also be tuned to apply the bound meaningfully.  If
$s, \alpha, \delta$ are not chosen judiciously, the resulting
probability bound may be potentially meaningless, yielding a
probability value close to, or potentially less than, zero.

\subsection{Trained neural network versus optimal neural network}

The generalization error expresses how efficient the training
is. Here, we discuss how close the  trained network is to an optimal
network. We denote the optimal weights by $\theta^*$ and present a
``generalization gap'' type estimate for the error between networks
with weights $\theta^*$ and weights $\theta(S)$.

We let $\theta^*$ be a solution of 
\begin{equation} \label{theta star minimization N0}
   \theta^* =
\underset{\theta\in {\fmtext \Theta(N_0,\mathcal R_0)}}{\operatorname{argmin}}\,
 {\fmtext\mathcal{L}_{r}}(\theta,{\fmrtext \tau})
   = \underset{\theta\in {\fmtext \Theta(N_0,\mathcal R_0)}}{\operatorname{argmin}}\, \Expec_{\fmrtext (\newrandomMvariable,{\randomYvariable})\sim \tau}  (\|{\bar f}_{\theta}(\newrandomMvariable)-{\fmrtext {\randomYvariable}}\|^2 + \alpha \mathcal R(\theta)),
\end{equation}
and write
\[
   \mathcal{L}^*_{r} =
   \mathcal{L}_{r}(\theta^*,\tau) .
\]
This means that $\Lambdamap\mapsto {\bar f}_{\theta^*}(\Lambdamap)$ is the neural
network having the optimal expected performance for $(\Lambdamap,\randomYvariable)$ sampled
from distribution $\tau$.
{\yytext Note that the optimal parameter $\theta^*$ depends on the regularization parameter $\alpha$,
and to emphasize this we sometimes denote it by $\theta^*(\alpha)$. Clearly,
\beq\label{on optimal performance}
 \Expec_{(\newrandomMvariable,\randomYvariable)\sim \tau}  (\|{\bar f}_{\theta^*(\alpha)}(\newrandomMvariable)-{\fmrtext \randomYvariable}\|^2)\le {\fmtext\mathcal{L}_{r,0}}(\alpha),
\eeq
cf. \eqref{L0 equation pre-formula1}-\eqref{L0 equation pre-formula2}. We observe that when \eqref{R0
  condition}-\eqref{deterministic estimate 2} {\fmrtext holds, then, when} $\alpha$ grows, also the
bound $ {\fmtext\mathcal{L}_{r,0}}(\alpha)$ for the expected error in \eqref{on
  optimal performance} may grow.  }

{\ztext A trivial, but important observation is that {\zztext when
    ${\bar f}_{\theta_0}$ is any neural network, for example, a neural
    network which corresponds to an implementation of the
    approximation of the analytic solution algorithm, we have}
\begin{equation}\label{always optimal}
       \Expec_{(\newrandomMvariable,\randomYvariable)\sim \tau}
                \left[ {\fmtext\mathcal{L}_{r}}(\theta^*, \newrandomMvariable,{\fmrtext \randomYvariable}) \right]\le  \Expec_{(\newrandomMvariable,\randomYvariable)\sim \tau}
                \left[ {\fmtext\mathcal{L}_{r}}(\theta_0, \newrandomMvariable,{\fmrtext \randomYvariable}) \right].
\end{equation}
This means that the optimal neural network ${\bar f}_{\theta^*}$ (or a
network trained with a sufficiently large data set as elucidated
below) has a better expected performance than the deterministic
approximation ${\bar f}_{\theta_0}$ of the analytic solution algorithm.}

Next, we estimate the expected performance gap between the optimal
neural network and the neural network ${\bar f}_{\theta(S)}$ optimized with
the training data $S$, defined by,
\begin{equation}
   {\xxtext \mathcal{G}_{opt}(S)
      := |\Expec_{(\newrandomMvariable,\randomYvariable)\sim \tau}({\fmtext\mathcal{L}_{r}}(\theta(S),
              \newrandomMvariable,{\fmrtext \randomYvariable})-{\fmtext\mathcal{L}_{r}}(\theta^*, \newrandomMvariable,{\fmrtext \randomYvariable}))|.}
\end{equation}
Given that the parameters $\theta(S)$ have been generated using the
training set $S$, $\mathcal{G}_{opt}(S)$ measures the difference
between the expected loss ${\fmtext\mathcal{L}_{r}}(\theta(S),{\fmrtext \tau})$ and the loss of
the optimal neural network, ${\fmtext\mathcal{L}_{r}}(\theta^*,{\fmrtext \tau})$.

Using similar methods to those used to prove Theorem
\ref{generalization gap} we obtain the following

\begin{theorem}\label{optimal gap}
Let ${\bar f}_\theta$ be truncated basic operator recurrent networks  {with truncation parameter $m = \|f\|_\infty$.} Consider {\fmtext a random} training set ${\fmtext {\bf S}}$ consisting of $s$ independent samples from distribution $\tau$ and let $\theta({\fmtext {\bf S}})$ be a minimizer for \eqref{eqn.boundedminproblem} and $\theta^*$ be a minimizer for \eqref{theta star minimization N0}  signifying the best possible weights. Then

(i) For {\yytext any $\alpha>0$ and} every sufficiently small
$\delta > 0$, we have
\begin{equation}\label{eqn.maingenbd 2}
   \Prob_{{\fmtext {\bf S} } \sim \tau^n} \left[   \mathcal{G}_{opt}({\fmtext {\bf S} }) \le {\xxtext 6\delta} \right]
   \geq 1 -   {\yytext 2C_1}\left( \frac{1}{\delta} \right)^{\yytext C_2} 
       \exp\left(-\frac 2{{\ytext ({\zztext {\fmrtext 9n}})^2\|f\|_\infty^4}} s \delta^2\right),
\end{equation}
{\yytext where $C_1$ and $C_2$  are given as in \eqref{C1 and C2 in case (i)}.}

{\zztext (ii) Let the function $f$ be approximated with accuracy
  $\e_0$ by some neural network ${\bar f}_{\theta_0}$ where ${\theta_0}\in {\fmtext \Theta(N_0,\mathcal R_0)}$ has
  sparsity bound $R_0\ge 1$, that is, conditions \eqref{R0
    condition}-\eqref{deterministic estimate 2} hold with ${\zztext
    R_0}$ and $\e_0$.  Then, for { all $\alpha $ satisfying \eqref{condition: alpha},}
  the inequality \eqref{eqn.maingenbd 2} holds with the constants
    $C_1$ and $C_2$ given by \eqref{Key estimate implication}.

\medskip

\noindent
Moreover, then
\begin{multline}\label{eqn.maingenbd 2B}
   \Prob_{{\fmtext {\bf S} } \sim \tau^n}     \left[  \Expec_{\newrandomMvariable ,{\fmrtext \randomYvariable}}
 {\yytext (\|{\bar f}_{\theta({\fmtext {\bf S} })}(\newrandomMvariable)-{\fmrtext \randomYvariable}\|^2) \le {\xxtext 6\delta}+4\e_0^2+2\alpha R_0+{2L^2c_0^{2L}}   \exp\Big( { 8R_0} \Big)
 {\cdotp {\color{black}n^2\sigma^2}}  }
    \right]
\\
   \geq 1 -   2{\qtext C_1}\left( \frac{1}{\delta} \right)^{\qtext C_2} 
    \exp\left(-\frac 2{{\ytext ({\zztext {\fmrtext 9n}})^2\|f\|_\infty^4}} s \delta^2\right).
\end{multline}
}
\end{theorem}

Roughly speaking, Theorem \ref{optimal gap} (i) means that {\yytext
  the trained neural network performs almost as well as the optimal
  neural network with large probability. Theorem \ref{optimal gap}
  (ii) estimates the probability that training yields a neural network
  which output is close to that of the target function. We note that
  the training of the neural network does not require that we know
  $\theta_0$, and thus Theorem \ref{optimal gap} (ii) estimates the
  probability that the trained neural network ${\bar f}_{\theta(S)}$
  approximates the function $f$ when some $\theta_0$ is just known to
  exist.}

\begin{proof}
{\color{black} The main lines of the proof are the following: We will compare the minimization of empirical and non-empirical loss functions when the parameters $\theta$ vary either in the continuous index set $\Theta(N_0,\mathcal R_0)$ or in the finite index set $\Theta_\rho$. Thus, we compare four minimization problems. Finally, the claim follows by applying the results for the generalization gap, that is, Theorem \ref{generalization gap} for the ``best'' and the ``worst'' minimization problem.} 

Let $\rho$ be given by (\ref{rho def}).
As in Theorem \ref{thm.sparsity} above and  \eqref{eqn.numele3}, we find that
$\theta^*$  satisfies the sparsity estimate
\beq \label{number or elements 1 star}
 {\ztext  \mathcal N_1(\theta^*)\leq  N_1.}
\eeq
We will compare the optimal parameter $\theta^*$ with an optimal
parameter $\theta^*_\rho$ in the finite set ${\fmtext \Theta_\rho}$, that is,
$\theta^*_\rho$ is a solution of
\beq\label{theta star minimization finite}
   \theta^*_\rho &=& \underset{\theta_\rho\in {\fmtext \Theta_\rho}} 
          {\operatorname{argmin}} \, {\fmtext\mathcal{L}_{r}}(\theta_\rho,{\fmrtext \tau}),
\\
   {\fmtext\mathcal{L}^*_{r,\rho}}&=& {\fmtext\mathcal{L}_{r}}(\theta^*_\rho,{\fmrtext \tau}).
\eeq
As in \eqref{Lip estimate}, if $\hat{\theta}\in {\fmtext \Theta(N_0,\mathcal R_0)}_\rho$
satisfies {\ztext $\|\hat \theta -
{\theta} \|_{\ell^1(\mathcal I;\R^n)} \le \rho$}, then for any ${\ptext\Lambdamap \in {\fmtext\mathcal B_{n\times n}}}$,
\begin{align}
   \| {\bar f}_{\hat{\theta}}(\Lambdamap) - {\bar f}_{{\theta}}(\Lambdamap) \| _{\R^{\ktext n}}
 \label{Lip estimate finite}
   &\le  {\ztext {}\,{2Lc_0^L}}\rho {\qtext (\mathcal{R}_0+2L) \exp({\qtext 2}\mathcal{R}_0)}
 \end{align}
 and  
 \beq\label{Q estimate 2}
   |    \| {\bar f}_{\hat \theta}(\Lambdamap) - \yvariable\|^2-
     \|   {\bar f}_{{\theta}}(\Lambdamap) -\yvariable\|^2 |
     \le Q {\ztext {}\,{2Lc_0^L}}\rho {\qtext (\mathcal{R}_0+2L) \exp({\qtext 2}\mathcal{R}_0)}
     \le \delta,
\eeq
{\xxtext where $\rho$ is given by (\ref{rho def})}
{\xxtext and  $Q={ 4n^{1/2}}(1+\|f\|_{\infty})$ as before, cf. \eqref{Q estimate0}-\eqref{Q estimate}.}
{\xxtext As $\rho\le \delta$, we have $\|\hat \theta -
{\theta} \|_{\ell^1(\mathcal I;\R^n)} \le \delta$;}
then ${\fmtext\mathcal{L}_{r}}^*(\rho)\leq  {\fmtext\mathcal{L}^*_{r}}+{\xxtext 2\delta}$.
Clearly, $ {\fmtext\mathcal{L}^*_{r}}\leq {\fmtext\mathcal{L}^*_{r}}(\rho)$. Thus,
\beq\label{Lip applied 2}
 {\fmtext\mathcal{L}^*_{r}}\leq  {\fmtext\mathcal{L}^*_{r,\rho}}\leq  {\fmtext\mathcal{L}^*_{r}}+{\xxtext 2\delta},
\eeq
{\ytext or, equivalently,
\beq\label{Lip applied 2B}
 {\fmtext\mathcal{L}_{r}}(\theta^*,{\fmrtext \tau})\leq  {\fmtext\mathcal{L}_{r}}(\theta^*_\rho,{\fmrtext \tau}) \leq  {\fmtext\mathcal{L}_{r}}(\theta^*,{\fmrtext \tau})+{\xxtext 2\delta},.
\eeq}

Let training data $S$ be sampled from $\tau^s$, and let
$\theta_\rho(S)$ be an optimal empirical parameter for $S$ in ${\fmtext \Theta_\rho}$,
that is,
\beq\label{min problem 2}
   \theta_\rho(S) &=& \underset{\theta_\rho\in {\fmtext \Theta_\rho}} 
             {\operatorname{argmin}} \, {\mathcal{L}^{(em)}_r}(\theta_\rho,S),
\\
   {\fmtext\mathcal{L}_{r}}_\rho(S) &=& {\mathcal{L}^{(em)}_r}(\theta_\rho(S),S).
\eeq
We denote, as in the above, an optimal empirical parameter for sample
$S$ in the entire parameter set by
\ba
   \theta(S) &=& \underset{\theta_\rho\in {\fmtext \Theta(N_0,\mathcal R_0)}} 
{\operatorname{argmin}}\,
{\mathcal{L}^{(em)}_r}(\theta,S),
\\
   {\fmtext\mathcal{L}_{r}}(S) &=& {\mathcal{L}^{(em)}_r}(\theta(S),S).
\ea
As in \eqref{Lip applied 2}, we have
\beq\label{Lip applied 3}
 {\fmtext\mathcal{L}_{r}}(S)\leq  {\fmtext\mathcal{L}_{r}}_\rho(S) \leq  {\fmtext\mathcal{L}_{r}}(S)+{\xxtext 2\delta},
\eeq
{\ytext or, equivalently,
\beq\label{Lip applied 3B}
{\mathcal{L}^{(em)}_r}(\theta(S),S)\leq {\mathcal{L}^{(em)}_r}(\theta_\rho(S),S)  \leq  {\mathcal{L}^{(em)}_r}(\theta(S),S)+{\xxtext 2\delta}.
\eeq}

We recall that by  \eqref{Hoeffding times finite},
\begin{multline}\label{Hoeffding times finite recall}
   \Prob_{{\fmtext {\bf S} }\sim \tau^n}\left[ \forall \theta \in {\fmtext \Theta_\rho} :
   |{\mathcal{L}^{(em)}_r}(\theta, {\fmtext {\bf S} }) - {\fmtext\mathcal{L}_{r}}(\theta,{\fmrtext \tau})| \le \delta
               \right]
\\
   \ge 1 - 2\,\cdotp{ 3}^{N_0n} M_0 (\mathcal{R}_0/\rho)^{N_0 n}
          \exp(-2s\delta^2 \alpha^{-2}{\ytext \mathcal{B}_0^{-2}}).
\end{multline}
By applying \eqref{Hoeffding times finite recall} when $\theta$ has
the value $\theta_\rho({\fmtext {\bf S} })\in {\fmtext \Theta_\rho}$, we trivially obtain
\begin{multline}\label{Hoeffding times finite recall applied 1}
   \Prob_{{\fmtext {\bf S} }\sim{\fmrtext \tau^s}} \left[ 
   |{\fmtext\mathcal{L}_{r}}(\theta_\rho({\fmtext {\bf S} }), {\fmtext {\bf S} }) - {\fmtext\mathcal{L}_{r}}(\theta_\rho ({\fmtext {\bf S} }),{\fmrtext \tau})| \le \delta
               \right]
 \\
   \ge 1 - 2\,\cdotp{ 3}^{N_0n} M_0 (\mathcal{R}_0/\rho)^{N_0 n}
          \exp(-2s\delta^2 \alpha^{-2}{\ytext \mathcal{B}_0^{-2}}) \,
\end{multline}
and, by applying \eqref{Hoeffding times finite recall} when $\theta$
has the value $\theta_\rho^*\in {\fmtext \Theta_\rho}$, we trivially obtain
\begin{multline}\label{Hoeffding times finite recall applied 2}
   \Prob_{{\fmtext {\bf S} }\sim{\fmrtext \tau^s}} \left[ 
   |{\fmtext\mathcal{L}_{r}}(\theta_\rho^*, {\fmtext {\bf S} }) - {\fmtext\mathcal{L}_{r}}(\theta_\rho^*,{\fmrtext \tau})| \le \delta
               \right]
\\
   \ge 1 - 2\,\cdotp{ 3}^{N_0n} M_0 (\mathcal{R}_0/\rho)^{N_0 n}
          \exp(-2s\delta^2 \alpha^{-2}{\ytext \mathcal{B}_0^{-2}}).
\end{multline}
{\fmtext We recall that for an arbitrary training data $S$,} $\theta_\rho^*$ and $\theta_\rho(S)$ are defined to be some
solutions of minimization problems \eqref{theta star minimization
  finite} and \eqref{min problem 2}, respectively. Thus we have for all
$S$,
\beq\label{trivial estimates}
{\mathcal{L}^{(em)}_r}(\theta_\rho(S), S)\leq {\mathcal{L}^{(em)}_r}(\theta_\rho^*, S),\quad
 {\fmtext\mathcal{L}_{r}}(\theta_\rho^*,{\fmrtext \tau})\leq {\fmtext\mathcal{L}_{r}}(\theta_\rho(S),{\fmrtext \tau}).
\eeq
By combining \eqref{Hoeffding times finite recall applied 1},
\eqref{Hoeffding times finite recall applied 2}, and \eqref{trivial
  estimates}, we obtain
\begin{multline} \label{Hoeffding times finite recall applied 3}
   \Prob_{{\fmtext {\bf S} }\sim{\fmrtext \tau^s}} \left[ 
   |{\fmtext\mathcal{L}_{r}}(\theta_\rho({\fmtext {\bf S} }),{\fmrtext \tau}) - {\fmtext\mathcal{L}_{r}}(\theta^*_\rho,{\fmrtext \tau})| \le 2 \delta
               \right]
\\
   \ge 1 - 2\,\cdotp 2\,\cdotp{ 3^{N_0n}} M_0 (\mathcal{R}_0/\rho)^{N_0 n}
          \exp(-2s\delta^2 \alpha^{-2}{\ytext \mathcal{B}_0^{-2}}).
\end{multline}
Combining this estimate with \eqref{Lip applied 2B} and \eqref{Lip
  applied 3B}, we conclude that
\begin{multline} \label{Hoeffding times finite recall applied 4}
   \Prob_{{\fmtext {\bf S} }\sim{\fmrtext \tau^s}} \left[ 
   |{\fmtext\mathcal{L}_{r}}(\theta({\fmtext {\bf S} }),{\fmrtext \tau}) - {\fmtext\mathcal{L}_{r}}(\theta^*,{\fmrtext \tau})| \le 2 \delta+2\,\cdotp{\xxtext 2\delta}
   \right]
\\
   \ge 1 - 2\,\cdotp 2\,\cdotp{ 3^{N_0n}} M_0 (\mathcal{R}_0/\rho)^{N_0 n}
          \exp(-2s\delta^2 \alpha^{-2}{\ytext \mathcal{B}_0^{-2}}).
\end{multline}
This yields {\ztext claim (i). 

{\yytext In claim (ii), the fact that inequality \eqref{eqn.maingenbd
    2} holds with constants $C_1$ and $C_2$ given by \eqref{Key
    estimate implication} follows by estimating $C_1$ and $C_2$ as in
  the proof of Theorem \ref{generalization gap}. Finally, using
  inequalities \eqref{R0 condition}, \eqref{deterministic estimate 2},
  \eqref{always optimal}, and \eqref{eqn.maingenbd 2B}, it follows
  that {\fmtext for any $S$}
\begin{multline}
   \Expec_{\newrandomMvariable ,{\fmrtext \randomYvariable}} \|{\bar f}_{\theta(S)}(\newrandomMvariable)-{\fmrtext \randomYvariable}\|^2\leq
{\fmtext\mathcal{L}_{r}}(\theta(S),{\fmrtext \tau})
\\
   \leq ({\fmtext\mathcal{L}_{r}}(\theta(S),{\fmrtext \tau})-{\fmtext\mathcal{L}_{r}}(\theta^*,{\fmrtext \tau})) +{\fmrtext 4\e_0^2+2\alpha R_0+{2L^2c_0^{2L}}   \exp\Big( 2 R_{max} \Big)
 {\cdotp {\color{black}n^2\sigma^2}}  .}
\end{multline}
This inequality together with claim (i) yields claim (ii).}}
\end{proof}

{\corrtext

\begin{remark}
 Above we have considered  a truncated basic operator recurrent network $f_\theta$.
The results can be generalized for a 
  neural network ${\bar f}_{\vec\theta}$,
  ${\vec\theta}=(\theta_{s_1},\dots,\theta_{s_K})$ of the form
\beq \label{F formula 2BB}
f_{\vec\theta} (\Lambdamap)=G({\bar f}_{\theta}^1(\Lambdamap), {\bar f}_{\theta}^{2}(\Lambdamap),\dots,{\bar f}_{\theta}^K(\Lambdamap)),
\eeq
where ${\bar f}^j_{\theta}(\Lambdamap)$, $j=1,2,\dots,K$ are basic operator recurrent networks 
and 
$G:\R^{Kn}\to \R^{d}$, $G(z_1,\dots,z_{Kn})=(G^a(z_1,\dots,z_{Kn}))_{a=1}^d$ is a given Lipschitz function, for example a neural network of
  the form \eqref{eqn.stdNNdef1}-\eqref{eqn.stdNNdef3}. We call ${\bar f}_{\vec\theta}$ in \eqref{F formula 2BB} a combination of 
  basic operator recurrent networks. This type of neural networks are below used to analyze solution algorithms for inverse problems, cf.\ \eqref{F formula 2BB}.

  To obtain the generalizations of the above theorems
 an essential observation is that 
\begin{align}
   \left\| \frac{\partial F_{\vec\theta}(\Lambdamap)}{
                  \partial \theta^{{\ytext  \ell,i}}_{p,s_j}} \right\|
    \le        \left\| \nabla { G}
                \right\|  \,\cdotp      
  \left\| \frac{\partial {\bar f}^j_{\theta}(\Lambdamap)}{
                  \partial \theta^{{\ytext  \ell,i}}_{p,s_j}} \right\|    \le \hbox{Lip}(G)     \left\| \frac{\partial {\bar f}^j_{\theta}(\Lambdamap)}{
                  \partial \theta^{{\ytext  \ell,i}}_{p,s_j}} \right\| .
%
\end{align}
Using this and results of Lemma \ref{lemma.lipf1}, we see that 
if $\hbox{Lip}(G)\le 1$, then
the Lipschitz constants of $F_{\vec\theta}(\Lambdamap)$ with respect to the components of
$\vec\theta$ satisfy the analogous estimates that are given in Lemma \ref{lemma.lipf1} for a basic operator recurrent network $f_\theta$. 
Moreover, if we assume that  $\| G\|_\infty\leq m=\| f\|_\infty$, then the proofs of Theorems 
\ref{generalization gap} and \ref{optimal gap} show that the claims of Theorems 
\ref{generalization gap} and \ref{optimal gap} are valid when the truncated basic operator recurrent network $f_\theta$
is replaced by the  combination of 
 basic operator recurrent networks $F_{\vec\theta}$, when the number $L$ in the claims of these theorems is replaced by the number $KL$, the terms $({\fmrtext 9n})^2$ are replaced by $(5d)^2$, and the terms $n^{3/2}$ are replaced by $nd^{1/2}$.
  The first replacement is needed as the number of components of the parameters 
   ${\vec\theta}=(\theta_{s_1},\dots,\theta_{s_K})$ is increased by a factor $K$ and hence the estimate  in formula \eqref{eqn.M0} changes.
   The second and the third replacements are needed as in the equation \eqref{B0 equation} and \eqref{Q estimate} the factor ${\fmrtext 9n}$  is replaced by $9d$.
   \end{remark}
}

\section{Example: Operator recurrent network for matrix inversion}
\label{sec: Example of matrix inversion}

Before we describe the relationship between operator recurrent neural networks and nonlinear inverse problems for the wave equation, we describe the simpler problem of matrix inversion.
For $n > 0$ an integer, suppose we have a data set
\begin{equation}
\label{eqn.matrixinversiondata}
\{(X_j, y_j) ; j = 1, \ldots, s \},
\end{equation}
where each $X_j \in \R^{n \times n}$ is a nonsingular matrix and $y_j \in \R^n$.
As before, the learning problem is to construct a function $f$ whose graph $\{(X, y = f(X)\}$ closely fits the data set.
However, suppose we also know that the data set comes from an algebraic relationship
\begin{equation}
Xy = h,
\end{equation}
where $h \in \R^n$ is a fixed vector.
Then the problem of constructing $f$ can be solved exactly by $f(X) = X^{-1} h$.
In other words, given a matrix $X$, we are tasked with learning how to apply its inverse to some particular vector $h$.

Developing efficient methods to solve linear systems under special conditions is a central problem in scientific computing.
In the absence of any additional assumptions on the linear system, in practice one must use Gaussian elimination or variations thereof.
However, over the decades, a variety of faster methods have been developed for specific families of matrices, such as those that are sparse, low-rank, oscillatory, arising from differential equations, and so forth.
Of particular note are iterative methods, such as Krylov methods.
Just as deep-learning-driven methods have been shown to be competitive with handmade algorithms in the realm of image processing,
it is of similar interest to see whether deep-learning-driven matrix inversion can be competitive with handmade inversion methods.

It is important to reiterate that this problem is distinct from,
and significantly more challenging than, a linear inverse problem.
In the linear inverse problem, the data set consists of pairs of vectors $\{(x_j, y_j)\}$ which obeys a linear (or approximately linear) relationship $x = Ay$ for a fixed matrix $A$.
In this case, the learning problem is to construct the linear (or approximately linear) map $f(x) = A^{-1}y$.
Traditional rectifier neural networks are well-suited to this task.

We seek to investigate the suitability of operator recurrent networks for learning to solve this problem under certain conditions.
From Theorem~\ref{thm.piecewisepoly}, we know that operator recurrent networks are exactly equal to piecewise matrix polynomials,
and therefore a natural question is how to approximate the matrix inversion problem with piecewise matrix polynomials.
One notable special case is Neumann series, which represents the inverse of $X$ via the matrix power series
\begin{equation}
X^{-1} = \sum_{k=0}^{\infty} (I - X)^k,
\end{equation}
and this equality holds when $\| I - X\| < 1$, in which case the power series converges.
By truncating this power series, we can approximate $X^{-1}$ by a matrix polynomial,
which can in turn be represented by an operator recurrent network.
To apply Neumann series to any matrix $X$, we first rescale the matrix so that $\|I - X\| < 1$ is satisfied,
before applying the series expansion, and then scale back.

Because it comes from a Taylor expansion, Neumann series is a very simplistic construction
and only holds on the disk of convergence given by $\| I - X\| < 1$.
When learning matrix inversion, we may have prior knowledge about additional spectral information of $X$,
and this can allow us to produce a polynomial approximation of the matrix inverse that has better approximation properties
and which also holds for regions other than a disk centered about identity or a multiple of the identity.

To see this, we further assume that the matrices $X_j$ are drawn from a set $U$ consisting of normal
(that is, orthogonally diagonalizable) matrices whose eigenvalues lie in a compact set $K$ that does not contain some open neighborhood of zero.
This guarantees that all $X_j$, as well as their inverses, have uniformly bounded spectral norm.

\begin{lemma}
\label{lemma.matrix_inversion_uniform_apprx}
Let $U$ consist of the set of orthogonally diagonalizable matrices whose eigenvalues lie in a compact set $K \subset \mathbb{C}$
that does not contain $0$, and assume that $\mathbb{C} \setminus K$ is connected. 
Then there exists a sequence of operator polynomials that approximate the function $X \mapsto X^{-1}$ uniformly on $U$.
\end{lemma}

\begin{proof}
Since $K$ does not contain $0$, then the complex function $z \mapsto 1/z$ is holomorphic on some open set containing $K$.
Because $\mathbb{C} \setminus K$ is connected, then we can apply the celebrated theorem of Mergelyan \cite{rudin1987} to construct a sequence of polynomials $\{p_i(z)\}$ that uniformly approximates $z \mapsto 1/z$ on $K$.
Then, by the holomorphic functional calculus, we have a sequence of operator polynomials $\{p_i(X)\}$ that uniformly approximates $X \mapsto X^{-1}$ on $U$.
\end{proof}

This basic result conveys that it is possible to find a polynomial $p$ such that $p(X)h$ well-approximates $X^{-1}h$
under the assumption that $X$ belongs to the set $U$. Next, we construct a toy example that demonstrates how piecewise-linear activation functions $\sigma$ in a operator recurrent network can be used to separate the space of matrices into separate regions, on each of which a different matrix polynomial is defined by the network.

\begin{lemma}
\label{lemma.positive_negative_classification}
Let $U$ consist of real symmetric $n \times n$ matrices of norm at most $1$, which are definite (that is, all eigenvalues share the same sign), and which are diagonally dominant. Furthermore, suppose that all matrices in $U$ have inverses whose norms do not exceed $1/\epsilon$ for some $\epsilon > 0$. Then there exists an operator recurrent network $f$ such that for every $X \in U$, and for some nonzero vector $h$,
\begin{equation}
f(X) = \begin{cases} Xh, & X > 0, \\ 0, & X < 0. \end{cases}
\end{equation}
In particular, there is a network that can distinguish $X$ as either positive definite or negative definite.
\end{lemma}

\begin{proof}
Because each $X \in U$ is diagonally dominant, then $|X_{ii}| \geq \sum_{j \neq i} |X_{ij}|$. Then the discs $D_i$ centered at $X_{ii}$ of radius $\sum_{j \neq i} |X_{ij}|$ must lie either entirely in the left half of the complex plane, or the right half. It follows from the Gershgorin disk theorem (see \cite{golub2013matrix}) that we can determine whether $X$ is positive or negative definite by determining the sign of any of its diagonal entries. 

Let $e_1 = [1,0, \ldots, 0]^T$ be the first standard coordinate vector, and let $E_{11}$ be the $n \times n$ of all zeros except a $1$ in the $(1,1)$ entry. Then we observe that for any matrix $X$,
\begin{equation}
E_{11} X e_1 = \begin{bmatrix} X_{11} \\ 0 \\ \vdots \\ 0 \end{bmatrix}.
\end{equation}
If $\sigma$ is the standard rectifier, then $\sigma(E_{11} X e_1)$ is a nonzero vector if and only if $X$ is positive definite.
Next, we claim that the vector
\begin{equation}
\label{eqn.matrix_inversion_positive_quantity}
(\|h\|_2/\epsilon) E_{11} X e_1 + Xh 
\end{equation}
has a positive number in its first component if $X > 0$.
If $(v)_1$ denotes the first component of any vector $v$, then 
\begin{align}
((\|h\|_2/\epsilon) E_{11} X e_1 + Xh)_1 & = X_{11} \|h\|_2 /\epsilon  + (Xh)_1 \\
& \geq X_{11} \|h\|_2 /\epsilon - \|Xh\|_2 \notag \\
& \geq X_{11} \|h\|_2 /\epsilon - \|h\|_2 \notag \\
& = \|h\|_2 (X_{11}/\epsilon - 1). \notag 
\end{align}
Since the norm of $X^{-1}$ is bounded by $1/\epsilon$, then the smallest eigenvalue of $X$ must be greater than $\epsilon$. Therefore $X_{11} > \epsilon$, so then the first entry of \eqref{eqn.matrix_inversion_positive_quantity} must be positive.
Next, we consider the value of the first entry of \eqref{eqn.matrix_inversion_positive_quantity} when $X < 0$. 
Now we consider the sum $b = \sum_{j=1}^n b_j$, where $b_j =( \|h\|_2/\epsilon) E_{jj} X e_j$. We claim it is negative.
Using similar manipulations,
\begin{align}
((\|h\|_2/\epsilon) E_{11} X e_1 + Xh)_1 & = X_{11} \|h\|_2 /\epsilon  + (Xh)_1 \\
& \leq X_{11} \|h\|_2 /\epsilon + \|Xh\|_2 \notag \\
& \leq X_{11} \|h\|_2 /\epsilon + \|h\|_2 \notag \\
& = \|h\|_2 (X_{11}/\epsilon + 1). \notag 
\end{align}
Since $X < 0$ and its inverse has norm bounded by $1/\epsilon$, then its largest eigenvalue is at most $-\epsilon$. Therefore $X_{11}/\epsilon < -1$, so the result follows. 

Now consider the vector
\begin{equation}
\label{eqn.matrix_inversion_positive_quantity2}
Xh + \frac{\|h\|_2}{\epsilon} \sum_{j=1}^n E_{jj} X e_j.
\end{equation}
From the above, every entry of this vector is either positive or negative, depending on whether $X$ itself is positive or negative definite. Now, applying the standard rectifier $\sigma$ to \eqref{eqn.matrix_inversion_positive_quantity2}, the quantity is unchanged if $X > 0$, and is set to zero if $X < 0$.
Finally, we consider the function
\begin{equation}
f(X) = - \sigma(\frac{\|h\|_2}{\epsilon} \sum_{j=1}^n E_{jj} X e_j) + \sigma(Xh + \frac{\|h\|_2}{\epsilon} \sum_{j=1}^n E_{jj} X e_j).
\end{equation} 
From our above, computations we observe that this $f$ satisfies the property desired for the lemma, and furthermore, $f$ can be constructed using layers of a general operator recurrent network.
\end{proof}

The purpose of this lemma is to produce an example that demonstrates how an operator recurrent network can distinguish between two
sets of matrices, in particular those are that positive definite or negative definite,
in a manner similar to how a standard rectifier network can determine whether a vector lies above
or below a particular hyperplane. 
Next, utilizing the network constructed in the above lemma,
we can show that an operator recurrent network exists that represents
a different matrix polynomial depending on whether the input matrix is positive
definite or negative definite.

\begin{theorem}
\label{thm.ORNpolyforposneg}
Let $U$ be the set of real symmetric matrices satisfying the same properties as
those of Lemma~\ref{lemma.positive_negative_classification}.
Then there exists an operator recurrent network $f$ such that $f(X) = p_1(X)$
when $X > 0$ and $f(X) = p_2(X)$ when $X < 0$, such that $p_1, p_2$ are operator
polynomials applied to the input vector $h_0$.
\end{theorem}

\begin{proof}
First we construct a network $f_1$ representing polynomials of degree $1$; in particular
\begin{equation}
f_1(X) = \begin{cases} A_1 Xh_0 + h_0, & X > 0, \\
A_2 Xh_0 + h_0, & X < 0, \end{cases}
\end{equation}
where $h_0$ is some fixed vector.
Let the operator network constructed in Lemma~\ref{lemma.positive_negative_classification},
using initial vector $h_0$, be relabeled at $g$. Then the above $f_1$ can be constructed by
\begin{equation}
f_1(X) = A_1 g(X) - A_2g(-X) + h_0,
\end{equation}
We observe that $f_1$ is a general operator recurrent network.
To obtain matrix polynomials of higher degree, we perform a similar construction.
By way of example, let us write down a piecewise degree 2 matrix polynomial by
\begin{equation}
f_2(X) = (C_1 X B_1 + A_1) g(X) - (C_2 X B_2 + A_2) g(-X) + h_0.
\end{equation}
Then
\begin{equation}
f_2(X) = \begin{cases} C_1 X B_1 Xh_0 + A_1 X h_0 + h_0, & X > 0 \\
C_2 X B_2 Xh_0 + A_2 X h_0 + h_0, & X < 0.
\end{cases}
\end{equation}
This construction can thus be easily extended to a network $f_n$,
and in each such case, $f_n(X)$ restricted to either $\{ X> 0\}$ or $\{X < 0\}$
yields an $n$-th degree operator polynomial.
\end{proof}

Lastly, we can use the above theorem, combined with Lemma~\ref{lemma.matrix_inversion_uniform_apprx},
to construct an operator recurrent network that represents two different operator polynomials,
each a different approximation to the matrix inverse, applied to the vector $h$.
Note that the construction in the above theorem has no restrictions on the coefficient matrices
$A_1, A_2, B_1, B_2$, etc.
Since the operator polynomials arising from Lemma~\ref{lemma.matrix_inversion_uniform_apprx}
have scalar coefficients,
this is equivalent to the matrix-valued coefficients being multiples of the identity matrix,
in which case they commute with all $X$.
In this case, it is clear that we can arrange for values of $A_1, A_2, B_1, B_2$,
and so forth, as to produce arbitrary scalar coefficients.

We reiterate that the purpose of Theorem~\ref{thm.ORNpolyforposneg}
is not to give an optimal result for how operator recurrent networks can learn
matrix inversion, but to provide a concrete illustration for how such networks 
can leverage its piecewise-polynomial nature, partitioning its input domain into distinct regions.
  
Finally, we note that Theorem \ref{generalization gap}, claim (i) and Theorem \ref{optimal gap}, claim (i) provide generalization estimates for training an operator recurrent neural network to represent matrix inversion. These imply estimates for sample complexity. Claims (ii) of these theorems, that is the improved estimates, however, are not generally applicable as we do not know whether the weight matrices can have rapidly decaying singular values (that is, have small $\ell^p$ norms). Next, we consider an inverse problem for a wave equation in which case  there is a solution algorithm which can be approximated by our neural network with such weight matrices that the improved generalization estimates, Theorem \ref{generalization gap}, claim (ii) and Theorem \ref{optimal gap}, claim (ii) are applicable.

\section{Example: Operator recurrent network for an inverse problem with the wave
  equation}
\label{sec.IP}

Here, we establish a direct relationship between operator recurrent
neural networks and reconstruction pertaining to an inverse boundary
value problems for the wave equation.

\subsection{Analytic solution of inverse problem by boundary control method}

We summarize the boundary control method used to solve an inverse
problem for the wave equation. For the sake of simplicity, we present
the one-dimensional case. We consider the wave equation with an
unknown wave speed $c = c(x)$,
\begin{eqnarray} \label{e:wave-ivp 1}
   (\partial_t^2-c(x)^2\partial_x^2)u(x,t) &=& 0 ,
\quad
   x \in \R_+ ,\ t \in \R_+
\\
\nonumber
   {\partial_x u(x,t)}|_{x=0} &=& h(t) ,\\
\nonumber
   {u(x,t)}|_{t=0} &=& 0 ,\
   {\partial_t u(x,t)}|_{t=0} = 0 ,
\quad
   x \in \R_+ ,
\end{eqnarray}
where we assume that $c$ is a smooth positive function satisfying
$c(0) = 1$. We denote the solutions of the wave equation with Neumann
boundary value $h = h(t)$ by $u = u^h(x,t)$. Function $h$ can be
viewed as a boundary source. We assume that $c$ is unknown, but that
we are given the Neumann-to-Dirichlet map, ${\ytext \ND} = {\ytext
  \ND^c}$,
\begin{equation}
   {\ytext \ND} h = u^h(x,t)\bigg|_{x=0} ,\quad t \in (0,2T) .
\end{equation}
This map is also called a response operator that maps the source to
the boundary value of the produced wave. The Neumann-to-Dirichlet map
is a smoothing operator of order one, that is, it is a bounded linear
operator ${\ytext \ND}:L^2([0,2T])\to H^1([0,2T]),$ where
$H^s([0,2T])$ are Sobolev spaces.  {\ytext An alternative to
  approximate ${\ytext \ND}$ by a matrix would be to choose suitable
  bases in the Hilbert spaces $L^2([0,2T])$ and $H^1([0,2T])$ and
  represent ${\ytext \ND}$ with respect to the relevant basis vectors.
  An alternative that avoids using two different bases, is to consider
  the bounded operator ${\ytext\Lambdamapc}$,
\beq\label{Lambda map def}
   {\ytext \Lambdamapc = \p_t\ND^c} : L^2([0,2T])\to L^2([0,2T]),
\eeq
and approximate this operator in a basis of the Hilbert space $
L^2([0,2T])$. In this paper, we use this option and consider operator
\eqref{Lambda map def} as the given data.}
  
The travel time of the waves from the boundary point $0$ to the point
$x$ is given by
\begin{equation}
   \tau(x) = \int_0^{x} \frac{\dd x'}{c(x')} .
\end{equation}
We consider the set $M = [0,\infty)$ as a manifold with boundary
  endowed with the distance function $d_M(x,y) = |\tau(x) - \tau (y)|$
  that we call the travel time distance. We denote by $M(s) = \{ x \in
  \R_+:\ \tau(x) \leq s\}$ the set of points which travel time to the
  boundary is at most $s$. The set $M(s)$ is called the domain of
  influence. The function $\tau$ is strictly increasing and we denote
  its inverse by
\[
   \chi = \tau^{-1} :\ [0,\infty) \to [0,\infty) ,
\]
that is, $\tau(\chi(s))=s$. The function $\chi(s)$ is called the
travel time coordinate, because for every time $s$ it gives a point $x$ whose
travel time to the boundary is $s$.
{\color{black} The function
\ba
Z(s)=c(\chi(s))
\ea 
is the wave speed in the medium represented in 
the travel time coordinates and by \cite{korpela2016}, formula (22), it uniquely determines the wave speed $c(x)$ in Euclidean coordinates. Thus, it also determines the data
operator $\Lambdamapc$, and thus we can define
a non-linear operator
\beq
\mathcal F:Z\to \Lambdamapc.
\eeq
In the study of the inverse problems, this map is called the direct map.
Below, we approximate the function $Z(s)$
by a finite-dimensional vector $z=(Z(s_j))_{j=1}^m$,
where $s_j$ are points in the interval $[0,T]$.
Also, $\Lambdamapc$ will be approximated by
a finite-dimensional matrix $\Lambdamap=(\bra \Lambdamapc \psi_k\cet)_{j,k=1}^n$, we obtain
a finite-dimensional direct map, see \eqref{gfF sequence 2},
\beq
F:B^m(z^{(0)},\rho_0)\to \R^{n\times n},\quad F(z)= \Lambdamap,
\eeq
where $B^m(z^{(0)},\rho_0)\subset \R^m$ is a ball centered at a vector $z^{(0)}$ having positive elements.

Next, we return to the continuous setting} we explain how the data
operator $\Lambdamapc$ measured on the boundary can
be used to compute the wave speed function in the travel time
coordinate, that is, $c(\chi(s))$, and after that, how this
reconstruction process can be approximated by an algorithm that has
the same form as the neural network in
(\ref{eqn.NN1def1})-(\ref{eqn.NN1def2}).

{\ytext We define
\begin{equation}
   Sf(t) = \int_0^t f(t') \dd t' .
\end{equation}
We observe that $\p_t\ND^c=\ND^c \p_t$, and, hence, we have
$\ND^c=S\Lambdamapc=\Lambdamapc S$.}

We denote 
\begin{eqnarray}
   \bra u^f(T),u^h(T) \cet_{L^2(M)}
                 = \int_{M} u^f(x,T) {u^h(x,T)} c(x)^{-2} \dd x
\end{eqnarray}
and $\|u^f(T)\|_{L^2(M)} = \bra u^f(T),u^f(T)\cet^{1/2}_{L^2(M)}$. By
the Blagovestchenskii identity, see for example \cite{bingham2008,
  korpela2016}, we have
\begin{equation} \label{Blago 1}
   \bra u^f(T),u^h(T)\cet_{L^2(M)}
                     = \int_{ [0,2T]} (Kf)(t)h(t) \, \dd t ,
\end{equation}
while
\begin{equation} \label{Blago 2}
   \bra u^f(T),1\cet_{L^2(M)}
              = \int_{ [0,2T]} f(t)\Phi_T(t) \, \dd t ,
\end{equation}
where
\begin{eqnarray}
   K &=& J{\ytext S\Lambdamapc}-R{\ytext \Lambdamapc S} R J ,
\\[0.25cm]
   R f(t) &=& f(2T-t)\quad\hbox{``time reversal operator''} ,
\\
   J f(t) &=& \tfrac 12 \indicator_{{}_{[0,T]}}(t)
           \int_t^{2T-t} f(s) \dd s\quad\hbox{{``time filter''}} ,
\\
   \Phi_T(t) &=& (T-t){\corrtext \indicator_{{}_{[0,T]}}(t) }.
\end{eqnarray} 
Here, $J :\ L^2([0,2T]) \to L^2([0,2T])$ and $R :\ L^2([0,2T])\to
L^2([0,2T])$.

In the boundary control method the first task is to approximately
solve the following blind control problem: Can we find a
boundary source $f$ such that
\begin{equation}\label{control problem}
   u^f(x,T) \approx \indicator_{{}_{M(s )}}(x)\ \text{?}
\end{equation}
Here, $\indicator_A$ is the indicator function of the set $A$, that is
$\indicator_A(x) = 1$ for $x \in A$, zero otherwise. The problem is
called a blind control problem because we do not know the wave speed $c(x)$
that determines how the waves propagate in the medium, and we aim to
control the value of the wave at the time $t=T$.  This control problem
can be solved via regularized minimization problems. In
\cite{KLO-2019} the problem was solved using Tikhonov regularization,
while in this paper we consider sparse regularization techniques that
are closely related to neural networks.

\subsection{Variational formulation and sparse regularization}

In sparse regularization, we represent the function $f(t)\in
L^2([0,2T])$ in terms of {\xxtext orthogonal} functions $\psi_j(t) \in
L^2([0,2T])$, $j=1,2,\dots,{\xxtext n}$, where ${\xxtext n} \in
\mathbb N_+ \cup \{ \infty \}$, such that
\begin{equation}
   \left\| \sum_{j=1}^{\xxtext n} f_j \psi_j \right\|_{L^2([0,2T])}
              \leq C_0\sum_{j=1}^{\xxtext n} |f_j| .
\end{equation}
{\ytext Here, the case ${\xxtext n}<\infty$ corresponds to numerical
  approximations with a finite set of basis functions, and the case
  ${\xxtext n}=\infty$ corresponds to the ideal continuous model; we
  consider these two cases simultaneously. When ${\xxtext n}=\infty$,
  we assume that the functions $\psi_j(t),$ $j=1,2,\dots$ span a dense
  set in $L^2([0,2T])$.}

For $\mathbf{f} = (f_j)_{j=1}^{\xxtext n}$, we denote 
\begin{equation}\label{B-definition}
   f(t) = (B\mathbf{f})(t) = \sum_{j=1}^{\xxtext n} f_j \psi_j(t) .
\end{equation}
For ${\xxtext n} = \infty$, we denote $\ell_{\xxtext n}^1 = \ell^1$
and $\|\mathbf{f}\|_1 = \sum_{j=1}^\infty |f_j|$. For ${\xxtext n} <
\infty$, we denote $\ell_{\xxtext n}^1 = \R^{\xxtext n}$ and $\|
\mathbf{f} \|_1 = \sum_{j=1}^{\xxtext n} |f_j|$.

We seek solutions for which $\mathbf{f}=(f_j)_{j=1}^{\xxtext n} \in
\ell^1_{\xxtext n} $ is a sparse vector. {\ztext Such sparse vectors
  correspond to sources that are generated by a small number of basis
  functions $\psi_j$.} {\ftext We let $P_s:L^2([0,2T])\to L^2([0,2T])$
  denote the mulitiplication by the indicator function of the interval
  $[0,s]$, that is, $(P_sf)(t)=\indicator_{[0,s]}(t)\,f(t)$.}

To obtain approximate solutions of control problem \eqref{control
  problem}, we consider an \\ $\ell^1_{\xxtext n}-$regularized version
of the minimization problem,
\begin{equation} \label{minimization for ISTA}
   \min_{\mathbf{f} \in  \ell^1_{\xxtext n}}
         \| u^{P_s B \mathbf{f}}(\,\cdotp,T) - 1 \|_{L^2(M)}^2
                  + \alpha \|\mathbf{f}\|_1 ,
\end{equation}
where $\alpha>0$  is a regularization parameter.
This minimization problem is equivalent to finding $\mathbf{f}$ that
solves
\begin{equation} \label{eq189}
   \min_{\mathbf{f} \in  \ell^1_{\xxtext n}} \bra K P_s B {\mathbf{f}},
        P_s B {\mathbf{f}} \cet_{L^2([0,2T])}
     - 2 \bra  P_s B{\mathbf{f}},\Phi_T \cet_{L^2([0,2T])}
                   + \alpha \|\mathbf{f}\|_1 ,
\end{equation}
where $K = J {\ytext S\Lambdamapc} - R {\ytext \Lambdamapc S} R J$ as
before. We denote the solution of this minimization problem by
$\mathbf{f}_{\alpha,s}$.

Minimization problem (\ref{eq189}) can be solved using the Iterated
Soft Thresholding Algorithm (ISTA) \cite{Daubechies2004}. The standard
ISTA algorithm is the iteration
\begin{equation} \label{ISTA}
   \mathbf{f}_{\qtext s}^{(m+1)}
      = \sigma_\alpha (\mathbf{f}_{\qtext s}^{(m)}
      - B^* P_s (J {\ytext S\Lambdamapc}
      - R {\ytext \Lambdamapc S} R J) P_s B \mathbf{f}_{\qtext s}^{(m)}
      + B^* P_s \Phi_T) ,\quad m=1,2,\ldots ,
\end{equation}
where $\mathbf{f}_{\qtext s}^{(m)} \in \ell^1_{\xxtext n}$,
$\mathbf{f}_{\qtext s}^{(0)}=0$ and $\sigma_\alpha$ is the soft
thresholding operator, given by
\begin{equation}
   \sigma_\alpha(x) = \max(0,x-\alpha)
         - \max(0,-x-\alpha) = \relu(x-\alpha) - \relu(-x-\alpha)
\end{equation}
for $x \in \R$; for a vector $x = (x_j)_{j=1}^{\xxtext n}$ it is
defined componentwise.

By \cite{Daubechies2004},
\begin{equation} \label{ISTA LIMIT}
   \mathbf{f}_{\alpha,s}
           = \lim_{m \to \infty} \mathbf{f}_{\qtext s}^{(m)} ,
\end{equation}
where the limit is taken in $ \ell^1_{\xxtext n}$, and the convergence
in this limit is exponential. We denote $f_{\alpha,s} = B
\mathbf{f}_{\alpha,s}$. {\ytext When ${\xxtext n}=\infty$, we have by
  Appendix~A that}
\begin{equation} \label{sparse convergence}
   \lim_{\alpha \to 0} u^{f_{\alpha,s}}(.,T)
                      = \indicator_{{}_{M(s)}}(.)
\end{equation}
in $L^2(M)$.

\subsection{Reconstruction}

When the minimizers $f_{\alpha,s}$ are found for all $s\in [0,T]$ with
small $\alpha>0$, we continue the reconstruction of the wave speed by
computing the volumes of the domains of influence,
\begin{equation} \label{possu}
   V(s) = \|1_{M(s)}\|_{L^2(M)}^2
      = \lim_{\alpha \to 0} \bra u^{f_{\alpha,s}}(T),1 \cet_{L^2(M)}
   = \lim_{\alpha \to 0} \bra f_{\alpha,s},\Phi_T \cet _{L^2([0,2T])} ,
\end{equation}
where $s \in [0,T]$. We note that $M(s)=[0,\chi(s)]$. In particular,
$V(s)$ determines the wave speed in the travel time coordinate,
\begin{align} \label{pikkuvee}
   v(s) = \frac{1}{\partial_s V(s)} ,
\end{align}
That is,
\begin{align} \label{c in travel time}
   v(s) = c(\chi(s)) ,\quad
     \chi(s) = \int_0^s v(t) \, \dd t .
\end{align}
When $v(s)$ is obtained, we can find the wave speed $c(x)$ also in the
Euclidean coordinates using the formula,
\begin{equation} \label{cee}
   c(x) = v(\chi^{-1}(x)) .
\end{equation}
However, in our reconstruction, we consider the function $v(s)$ as the
final result.

\subsection{Identification with operator recurrent networks}

The ISTA algorithm iteration \eqref{ISTA LIMIT} produces
$\mathbf{f}_s^{(n_0)}$ after $n_0$ steps. We observe
that this iteration can be expressed by {\qtext defining}
\begin{eqnarray}
   {\bf h}^{(3m+1)}_s = \mathbf{f}^{(m)}_s ,\quad 
   {\bf h}^{(3m+2)}_s = P_s B \mathbf{f}^{(m)}_s ,\quad
   {\bf h}^{(3m+3)}_s = R J P_s B \mathbf{f}^{(m)}_s
\end{eqnarray}
and viewing it as the operator recurrent neural network $\Lambdamap\mapsto
f_{(\alpha,s)}(\Lambdamap)$, {\color{black} where 
\beq 
f_{(\alpha,s)}(\Lambdamap)={\bf h}^{(3{\xxtext
    \ell_0}+1)}_s,
    \eeq
    in which, for $m=0,1,\dots,{\xxtext \ell_0}$
\begin{eqnarray}\label{final NN1}
   \hspace{-10mm} {\bf h}^{(3m+3+1)}_s &=& \sigma_\alpha (I\,{\bf h}^{(3m+1)}_s
                           - B^* P_s J {\ytext S\Lambdamap} {\bf h}^{(3m+3)}_s
         + B^* P_s R {\ytext \Lambdamap} {\bf h}^{(3m+2)}_s + B^* P_s \Phi_T) ,
\\ \label{final NN2}
 \hspace{-10mm}   {\bf h}^{(3m+3)}_s&=& P_s B {\bf h}^{(3m+1)}_s ,
\\ \label{final NN3}
  \hspace{-10mm}  {\bf h}^{(3m+2)}_s &=& {\ytext S} R J P_s B {\bf h}^{(3m+1)}_s 
\end{eqnarray}
with the initial state $\mathbf{h}^{(1)}_s=0$.}} This is motivated by the notion of
unrolling.  {\ytext As the low-pass filter operator $J$ and the
  integrator $S$ are compact operators in $L^2(0,2T)$, and moreover,
  the operators $S R J P_s$ and $P_s J S$ appearing above are in a
  Schatten class $\mathcal S_p$ with index $p > 1/2$,} we approximate
the above algorithm as a neural network with weight matrices of the
form (\ref{eqn.Cmatrixdecomp}), and
\begin{equation}\label{alternatives for the fixed operators}
   A^{\ell{}} = A^{\ell,{}(0)} + A^{\ell,{}(1)}_\theta ,\quad
   B^{\ell{}} = B^{\ell,{}(0)} + B^{\ell,{}(1)}_\theta ,
\end{equation}
where the $A^{\ell,{}(0)}$ and $B^{\ell,{}(0)}$, considered as fixed
operators in a suitable basis are zero operators, identity operators,
projectors $P_s$ or $P_s R$, and
$A^{\ell,{}(1)}_\theta$ and $B^{\ell,{}(1)}_\theta$ {\xxtext are
  operators ${\ytext S} R J P_sB$ and $B^*P_s J {\ytext S}$ appearing
  in \eqref{final NN1}-\eqref{final NN3}, which are Schatten class
  operators, in $\mathcal S_p$ with index $p > 1/2$.  {\xxtext When
    ${\xxtext n}<\infty$, the generalized H\"older inequality implies
    for a matrix $A\in \R^{n\times n}$ and $p>1/2$ that
\beq\label{Holder rpq}
\|A\|_{\mathcal S_{1/2}(\R^{n\times n})}\leq n^{1/r}\|A\|_{\mathcal S_{p}(\R^{n\times n})},
\eeq
where $r=p/(2p-1)$. Furthermore, $ B^* P_s \Phi_T$ in \eqref{final
  NN1}-\eqref{final NN3} are the bias vectors.}

We have included the fixed operators $A^{\ell,{}(0)}$ and
$B^{\ell,{}(0)}$ in the network architecture, because then {\ytext for any
  given value of $s\in [0,T]$ the computation of
  $\mathbf{f}_s^{({\xxtext \ell_0})}$ in the} discretized boundary
control method can be written as an operator recurrent network $f
_{\theta}^s(\Lambdamap)$ of the form \eqref{eqn.NN1def2}. Here,
parameters $\theta$,
define the operator recurrent networks $f_{\theta}^s(\Lambdamap)$, depend on $s$ and
$\theta$. {\xxtext Also, by
  \eqref{Holder rpq}, when $n<\infty$, it follows that the neural
  network \eqref{final NN1}-\eqref{final NN3} of depth $3\ell_0+1$ has
  the sparsity bound
\beq\label{R function with depth n}
\mathcal R(\theta_s) &\leq &\ell_0(\|B^*S R J P_sB\|_{\mathcal S_{1/2}(\R^{n\times n})}+\|B^*P_s J SB\|_{\mathcal S_{1/2}(\R^{n\times n})})\\
\nonumber 
&\leq &C_r \ell_0n^{1/r},
\eeq
where $r<\infty$ is arbitrary and $C_r$ depends on $r$.}

{\ytext In the discretized boundary control
method we compute the functions   $\mathbf{f}_s^{({\xxtext \ell_0})}=f_{(\alpha,s)}(X)$ 
that approximate
functions $f_{\alpha,s_j}$, 
for parameter values $s = s_j$, $j=1,2,\dots,{\xxtext K}$, given by
\beq \label{sj values}
    s_j = j T/{\xxtext K}\in [0,T].
\eeq
{\color{black}
Note that  $\mathbf{f}_s^{({\xxtext \ell_0})}$ 
converge to the
functions $f_{\alpha,s_j}$,  as the depth of the neural network, $\ell_0$ tends to infinity.
Then, we define analogously to \eqref{layers 1}, we denote
\beq 
& &f^j_{\theta}(\Lambdamap)=
f_{(\alpha,s_j)}(X)
\in \R^n,\quad j=1,2,\dots,K\\
& &{\bf f}(\Lambdamap)=(f^1_{\theta}(\Lambdamap),\dots,f^K_{\theta}(\Lambdamap))\in (\R^{n})^K.
\eeq 
We also denote $s_0=0$ and $f _{\theta}^0(\Lambdamap)=0$.}

We may add one linear layer $G_1$ into the neural network  that 
computes the derivative in (\ref{pikkuvee}) using finite
differences,
\beq \label{pikkuvee approx}
  D_\alpha(s_j)&:=&\frac1{ v_\alpha(s_j)}
           = \frac{V_\alpha(s_{j}) - V_\alpha(s_{j-1})}{s_j - s_{j-1}}
          \\
          \nonumber &=& \frac{1}{s_{j} - s_{j-1}}( \bra f_{\alpha,s_{j}}, \Phi_T\cet _{L^2([0,2T])}- \bra f_{\alpha,s_{j-1}}, \Phi_T\cet _{L^2([0,2T])}),
\eeq
where $j=1,2,\dots,{\xxtext K}$ and
\begin{equation} \label{possu approx}
   V_\alpha(s_j) = \bra f_{\alpha,s_j}, \Phi_T\cet _{L^2([0,2T])},
\end{equation}
cf. \eqref{possu}. We denote
$G_1(f_{\alpha,s_1},f_{\alpha,s_2},\dots,f_{\alpha,s_{\xxtext K}})=(
D_\alpha(s_1),\dots,D_\alpha(s_{\xxtext K}))$.

{\ytext Approximating the componentwise function $s\to s^{-1}$ via a
  standard neural network $G_2:\R^{\xxtext K}\to \R^{\xxtext K}$, of
  the form \eqref{eqn.stdNNdef1}-\eqref{eqn.stdNNdef3}, we obtain a
  neural network $F_{\vec\theta}$,
  ${\vec\theta}=(\theta_{s_1},\dots,\theta_{s_K})$ of the {\color{black} form
\beq \label{F formula 2BBC}
H_{\vec\theta} (\Lambdamap)=G_2(G_1(f _{\theta}^1(\Lambdamap),f _{\theta}^2(\Lambdamap),\dots,
f _{\theta}^K(\Lambdamap))),
\eeq
which} output approximates} the values $v(s_j) = c(\chi(s_j))$,
$j=1,2,\dots,{\xxtext K}$}. 
{\qtext By using \cite{yarotsky2016},
  steps \eqref{possu approx} and \eqref{pikkuvee approx} (see also
  Theorem \ref{thm.apprx1}), and the function $s\to s^{-1}$ can be
  approximated by a neural network $G_2$} of the form
(\ref{eqn.NN1def1})-(\ref{eqn.NN1def2}).  
We observe that formula \eqref{F formula 2BBC} is analogous to  \eqref{layers 2}.
{\xxtext Finally, by
  \eqref{R function with depth n}, the neural network $F_{\vec\theta}$
  in \eqref{F formula 2BBC} can be written as an operator recurrent
  network that has the sparsity bound
\beq\label{R function with depth n A}
\mathcal R(\vec\theta) &\leq &C'_rK \ell_0n^{1/r},\eeq
where $r<\infty$ is arbitrary and $C'_r$ depends on $r$.}}

\subsubsection*{The low-pass filter operator $J$ is in a Schatten
                class} 

Here, we show that the low-pass filter operator $J$ used above is in a
Schatten class with $p{\ytext >}1$. We consider the extension of low
pass filter operator $J:L^2(0,2T)\to L^2(0,2T)$. It can be written as
\begin{equation}
   J = A^{-1/2} \circ (A^{1/2} \circ J) ,
\end{equation}
where $A = -\frac {\dd^2}{\dd x^2}+1$ where $\frac {\dd^2}{\dd x^2}$
is Laplace operator defined as an unbounded self-adjoint operator in
$L^2([0,2T])$ with Neumann boundary condition, $\mathcal D(A) = \{f\in
H^2([0,2T]):\ \frac {\dd f }{\dd x}(0)=0,\ \frac {\dd f }{\dd
  x}(2T)=0\}$, where $H^s([0,2T])$ are Sobolev spaces, $\mathcal
D(A^{1/2}) = H^1([0,2T])$ and $A^{1/2} \circ J^* :\ L^2([0,2T]) \to
L^2([0,2T])$ is a bounded operator. As the eigenvalues of $A$ are of
the form $\lambda_j = c_T j^2+1$, the eigenvalues of $A^{-1/2}
:\ L^2([0,2T]) \to L^2([0,2T])$ are $(c_T j^2+1)^{-1/2}$, and, hence,
$A^{-1/2} :\ L^2([0,2T]) \to L^2([0,2T])$ is in the Schatten class
$\mathcal S_p(L^2(0,2T))$ with $p > 1$. As the Schatten classes are
operator ideals, this implies that
\begin{equation}
   J \in \mathcal S_p(L^2(0,2T)) ,\quad\text{with $p>1$.} 
\end{equation}
In the same way, we observe that $S \in \mathcal S_p(L^2(0,2T))$ with
$p>1$ and hence the operators $S J$ and $S R J$ appearing in
\eqref{ISTA} satisfy $S J$, $S R J \in \mathcal S_p(L^2(0,2T))$ with
$p>1/2$. Thus, when we approximate these operators by matrices
representing operators in a space spanned by finitely many basis
functions $\psi_j$, it is natural to assume that the $\mathcal
S_{1/2}$-norms of these matrices are bounded with some relatively
small constants.

\medskip\medskip

\noindent
Furthermore, we note that the ``bias functions'' $\Phi_T$ are in the
Sobolev space $H^1([0,2T])$, that is, a compact subset of
$L^2([0,2T])$ and therefore $\Phi_T$ can be approximated by a vector
which coordinates are a sparse sequence.

\medskip\medskip

In summary, the boundary control method can be approximated by
an operator recurrent network of the form (\ref{eqn.NN1def1})-(\ref{eqn.NN1def2}),
where the weight operators $A$ and $B$ are either Schatten class
operators (which we can train with sparsity regularization to obtain a
better algorithm), or simple operators, such as the time-reversal
operator $R$ or the projector $P_s$ that we may consider as fixed in
the neural network and that we do not train. The time reversal operator 
is extensively used in imaging applications; see for example
\cite{bal2003time,borcea2003theory}. 
Also, the bias vectors
can be approximated by sparse vectors. Furthermore, we observe that
if we consider sparse regularization leading to activation functions
that are linear combinations of ReLU functions, we do not specify in
the neural network formulation what the basis function $\psi_j$
are. Thus the training of the neural network also leads to finding a
basis that is optimal for sparse regularization.

\subsection{Discretization error versus depth and width of the
            network}
\label{sec: discussion on inverse problems algorithms}

Here, we estimate the error in the point of departure of the network
design in the main body of this paper. By stability and error analyses
of the boundary control method, we can estimate how well the
discretized boundary control method works and what are the error
estimates for all wave speeds $c$ in the set
\begin{equation} \label{nopeudet}
   {\mathcal V^3} = \{ c \in {C^3} (M)\ :\
        C_0 \le c(x) \le C_1 ,\ \|c\|_{C^3(M)} \le M,\,
        \hbox{supp}(c-1) \subset I_0\} ,
\end{equation}
where $I_0 \subset \R_+$ is a compact interval. {\yytext We use $C$ as
  a generic constant which depends on parameters of the space
  ${\mathcal V^3}$ and which value may be different in each
  appearance.}

{\yytext We consider the discretization of analytical algorithms that
  reconstruct $c(x)$, with error $C \delta$ in the $L^\infty(M)$-norm,
  from the map $\Lambdamap$, or from the map $\ND$. To this end, we
  denote $\e = \delta^{m}$, where $m=270$.}  In \cite{korpela2016}, it
was shown for the discretized boundary control method that we can
compute the wave speed with error $C\delta=C \epsilon^{\gamma}$, with
H\"older exponent $\gamma=1/m$, when we discretize the time interval
$[0,T]$ with a grid of $N_0(\epsilon) = C \epsilon^{-4/7}$ points and
measurement operator ${\ytext \ND}$ is given with an error $\epsilon$
{\ytext in the operator norm in $L^2(0,2T)$}. {\yytext In this paper
  we omit the analysis of the measurement errors in the
  Neumann-to-Dirichlet map, and consider only the discretization error,
  that is, the error caused by approximating the infinite dimensional
  operators by finite dimensional matrices.} {The discrete BC-method
  in \cite{korpela2016} requires solving $K \leq C \epsilon^{-1/18}=C
  \delta^{-270/18}$ minimization problems of the form
  \eqref{minimization for ISTA}, that is, for each value of $s_j$ in
  \eqref{sj values}}. Moreover, as by \cite{Daubechies2004} the
iteration in the ISTA algorithm has exponential convergence to the
solution of the minimization problem, we conclude that the linear
system can be solved with accuracy $C \epsilon$ using an iteration of
$C \log(\epsilon^{-1})$ steps that each require a composition of
linear operators and the operator ${\ytext \ND}$.

From the discretization error estimates we may deduce estimates for
the depth and width of the operator recurrent neural network based on
a scenario without training: {\xxtext The upper bound for the depth is
  $L$ and the upper bound for the width $n$ is
\begin{equation}
   L \leq C \log (\delta^{-1}) ,\quad
   n \leq C \epsilon^{-4/7-1/18}
     \le C \epsilon^{-9/14} \leq C \delta^{-175} .
\end{equation}
Moreover, as $K \leq C \epsilon^{-1/18}$, we see 
that this neural network  can be written as $H_{\vec\theta}$ given
in \eqref{F formula 2BBC} that has the sparsity bound {\corrtext $\mathcal R(\vec\theta)$ and accuracy bound $\e_0$} that  given by
\beq \label{R function with depth n B}
  & & \mathcal R(\vec\theta) \leq C' K L n^{1/r}\leq C' \delta^{-270/18}
   \, \cdotp
   \log (\delta^{-1})\,\cdotp \delta^{-175/r}\leq C'' \delta^{-16},\\
   & &\corrtext{\e_0=C\delta},
\eeq
where $r<\infty$ is arbitrary and $C,C'$ and $C''$ depend on $r$. 
Consider now the case when a priori distribution of the data  is supported in the set of the Neumann-to-Dirichlet maps corresponding to the wave speeds
$c\in   {\mathcal V^3}$. Then the above
 implies, in terms of Definition \ref{def:
  approximation}, that the map $\Lambdamapc\to c$, solving the inverse
problem for the wave equation, can be approximated with accuracy  
$\e_0=C\delta$ by a neural network $\Lambdamap\to
F_{\vec\theta}(\Lambdamap)$ where ${\vec\theta}$ has the sparsity bound
$R_0 \leq C''\delta^{-16}$. Note that here we do not require that the absolute values of the components
of the vector ${\vec\theta}$ are bounded by one. However, this happens if $T$ or the parameters of the set ${\mathcal V^3}$ 
are sufficiently small.}

\medskip\medskip

\noindent
The above worst case estimate gives also an upper bound how well an
optimally trained neural network performs. However, if one is
interested in reconstructing a wave speed $c$ in a subset ${\rtext \mathcal W} \subset
{\mathcal V^3}$ and uses training data sampled from the set ${\rtext \mathcal W}$, then
the trained network is by our analysis close to an optimal neural
network that will most likely perform better than the neural network
with a priori determined parameters approximating the boundary control
method for three reasons: First, the optimal neural network is
optimized to the subset ${\rtext \mathcal W}$, not the larger class ${\mathcal
  V^3}$. Second, the neural network is based on theoretical estimates
that prove worst case errors in all substeps. Third, the algorithm with
a priori determined parameters does not estimate the average error in
the reconstruction, but absolute error and thus the optimal neural
network that optimizes the expected error may perform better.

\appendix   

\renewcommand{\theequation}{\thesection.\arabic{equation}}
\renewcommand\thesection{\Alph{section}}
\setcounter{equation}{0}
   
\section{Time reversal algorithm with sparse regularization}
\label{appendix 1}

In this appendix we consider how the results in \cite{bingham2008,
  korpela2016} can be generealized in the case when one regularizes
the $\ell^1$ term of the source term.

Let $B: \ell^1\to L^2(0, T)$ be an operator such that there is $C_0>0$
such that $\|Bf\|_{ L^2(0, T)}\leq C_0\|f\|_{\ell^1}$.  For example,
when $s>1/2$, the Besov space $B^s_{11}(S^1)$ on the unit circle $S^1$
is subset of $L^2(S^1)$ (that is isomorphic to
$L^2([0,T])$). Moreover, there is an isomorphism $B:\ell^1\to
B^s_{11}(S^1)$ of the form \eqref{B-definition}, where $\psi_j$ are
wavelets \cite{triebel2008}.

\begin{theorem} \label{thm_minimization}
Assume that $B(\ell^1)\subset L^2(0, T)$ is a dense subset. Let $r
\in [0, T]$ and $\alpha > 0$.  Let us define
\begin{equation}
\label{muumipappa}
S_r=\{f \in L^2(0,T)\,:\,\hbox{supp}(f) \subset [T - r, T]\}.
\end{equation}
Then the regularized minimization problem 
\begin{equation}
\label{eq_minimization_regularized}
\min_{f \in \ell^1}(\bra Bf, K Bf\cet _{L^2(0, T)} - 2\bra Bf, \Phi_T\cet _{L^2(0, T)} + \alpha \norm{f}_{\ell^1} )
\end{equation}
has a minimizer $ f_{\alpha,r}$. Moreover $u^{Bf_{\alpha,r}}(T)$
converges to the indicator function of the domain of influence,
\begin{align}
&\lim_{\alpha\to 0}\norm{u^{Bf_{\alpha,r}}(T) - 1_{M(r)}}_{L^2(M;dV)} =0.
\end{align}
\end{theorem}

\begin{proof}[Proof of Theorem \ref{thm_minimization}.]
Let $\alpha>0$ and let $f\in \ell^1$. We define the energy function 
\begin{equation}
\label{appe1}
E(f) := \bra P_sBf, KP_sBf\cet _{L^2(0,T)}  - 2 \bra P_sBf, \Phi_T\cet _{L^2(0,T)} + \alpha\norm{f}_{\ell^1}.
\end{equation}
The finite speed of wave propagation
implies $\hbox{supp}\,(u^{P_sBf}(T)) \subset M(r)$. 
Moreover, by the
 Blagovestchenskii formula we have 
\begin{equation}
\label{eq:energy_functional}
E(f) = \norm{u^{P_sBf}(T) - 1_{M(r)}}_{L^2(M; dV)}^2 - \norm{1_{M(r)}}_{L^2(M; dV)}^2 + \alpha \norm{f}_{\ell^1}.
\end{equation}
Let $(f_j)_{j=1}^\infty \subset \ell^1$ be such that 
\begin{equation*}
\lim_{j \to \infty} E(f_j) = \inf_{f \in \ell^1} E(f)=:E^*
\end{equation*}
Then 
\begin{equation*}
\alpha \norm{f_j}_{\ell^1} \le E(f_j) + \norm{1_{M(r)}}_{L^2(M; dV)}^2
\leq E^*+\hbox{vol}(M)=E^{**},
\end{equation*}
and we see that $(f_j)_{j=1}^\infty$ is bounded in $ \ell^1$ and satisfies $\|f_j\|_{\ell^1}\leq \alpha^{-1}E^{**}.$

The space $\ell^1$ is the dual of the space $c_0$ of sequences
converging to zero. Thus by Banach-Alaoglu theorem, 
Hilbert space, there is a subsequence of $(f_j)_{j=1}^\infty$ that
weak${}^*$ converges in $\ell^1$.  Let us denote the limit by
$f_\infty \in \ell^1$ and the subsequence still by
$(f_j)_{j=1}^\infty$.

When $y=(y_i)_{i=1}^\infty\in \ell^1$, we denote
$p_k(y)=(y_i)_{i=1}^k\in \R^k$. Now, we see that as
$(f_j)_{j=1}^\infty$ weak${}^*$ converges to $f_\infty$ in $\ell^1$,
we have for all vectors $g_k=(\delta_{jk})_{j=1}^\infty \in c_0$ that
$(f_j,g_k)_{\ell^1,c_0} \to (f_\infty,g_k)_{\ell^1,c_0}$ as $j\to
\infty$. Hence we see that $p_k(f_j)$ converge to
$p_k(f_\infty)$ and for all $k$ and
\[
   \sum_{i=1}^k |(f_\infty)_i|\leq \lim_{j\to\infty} \sum_{i=1}^k |(f_j)_i|\leq 
 \norm{f_j}_{\ell^1}\leq  \alpha^{-1}E^{**}.
\]
Taking limit $k\to\infty$ we see that $\|f_\infty\|_{\ell^1}\leq
 \alpha^{-1}E^{**}$.

The map $U_T:L^2( 0, T) \to H^{1}(M)$, mapping $U_T:h \mapsto u^h(T)$, is
bounded. The embedding $I:H^{1}(M)\hookrightarrow L^2(M)$ is compact
and thus $U_T$ is a compact operator
\begin{equation*}
   U_T:L^2(0, T) \to L^2(M) .
\end{equation*}
As $P_sBf_j$ is a bounded sequence in $L^2(0, T)$, we see that by
replacing the sequence $(f_j)_{j=1}^\infty$ by its suitable
subsequence, we can assume that $u^{P_sBf_j}(T) \to
u^{P_sBf_\infty}(T)$ in $L^2(M)$ as $j \to \infty$.

The above yields that
\begin{align*}
&E(f_\infty) = \lim_{j \to \infty} \norm{u^{P_sBf_j}(T) - 1_{M(r)}}_{L^2(M; dV)}^2 - \norm{1_{M(r)}}_{L^2(M; dV)}^2 + \alpha \norm{f_\infty}_{\ell^1}
\\\nonumber&\le \lim_{j \to \infty} \norm{u^{P_sBf_j}(T) - 1_{M(r)}}_{L^2(M; dV)}^2 - \norm{1_{M(r)}}_{L^2(M; dV)}^2 + \alpha 
\liminf_{j \to \infty} \norm{f_j}_{\ell^1}
\\\nonumber&= \liminf_{j \to \infty} E(f_j) = \inf_{f \in S_r} E(f),
\end{align*}
and thus $f_\infty \in \ell^1$ is a minimizer for (\ref{appe1}). 

As $B(\ell^1)\subset L^2(0, T)$ is a dense subset, we see by using
Tataru's approximate controllability theorem that
\begin{equation*}
\{ u^{P_rBf}(T) \in L^2(M(r));\ f \in \ell^1 \}
\end{equation*}
is dense in $L^2(M(r))$. 
{Let $\delta > 0$. For $\epsilon=\frac{\delta^2}{2}$, let us choose} 
$f_\epsilon \in \ell^1$,  such that
\begin{equation}
\label{rudiw}
\norm{u^{P_rBf_\epsilon}(T) - 1_{M(r)}}_{L^2(M; dV)}^2 \le \epsilon.
\end{equation}
Using (\ref{eq:energy_functional}) we have
\begin{align*}
\norm{u^{P_rBf_{\alpha,r}}(T) - 1_{M(r)}}_{L^2(M; dV)}^2 \le  
E(f_{\alpha,r})+\norm{1_{M(r)}}_{L^2(M; dV)}^2.
\end{align*}
Because $E(f_{\alpha,r}) \le E(f_\epsilon)$ 
we have
\begin{align*}
\norm{u^{P_rBf_{\alpha,r}}(T) - 1_{M(r)}}_{L^2(M; dV)}^2 & \le  
 \norm{u^{P_rBf_\epsilon}(T) - 1_{M(r)}}_{L^2(M; dV)}^2 + \alpha \norm{f_\epsilon}_{\ell^1}.
\\\nonumber& \le \epsilon + \alpha \norm{f_\epsilon}_{\ell^1}.
\end{align*}
{When $0<\alpha<\alpha_\delta=\frac{\delta^2}{2\norm{f_\epsilon}_{\ell^1}}$,
we see that}
\begin{align*}
{\norm{u^{P_rBf_{\alpha,r}}(T) - 1_{M(r)}}_{L^2(M; dV)} \le
(\epsilon + \alpha \norm{f_\epsilon}_{\ell^1})^{\frac{1}{2}}=\delta}.
\end{align*}
Thus
\begin{align*}
{\lim_{\alpha\to 0}\,\norm{u^{P_rBf_{\alpha,r}}(T) - 1_{M(r)}}_{L^2(M; dV)} 
=0.}
\end{align*}
\end{proof}

\setcounter{equation}{0}
\section{Conditional expectation as a projector}

In this appendix, we recall the definition and the properties of  conditional expectations using $\sigma$-algebras
discussed in detail in \cite[Ch. 5]{Kallenberg} and \cite{Dudley}.

 Let $(\Omega,\Sigma,\Prob)$ be an complete probability space and
 ${Z} :\ \Omega \to \R^m$ be  a random variable. Below, we consider the case when $Z$ is $\R$-valued,
 that is, $m=1$, but the discussion below generalizes in a straight forward way for $m\in \mathbb Z_+$. 
 
 Let
$\mathcal{B}_{Z} \subset \Sigma$ be a $\sigma$-algebra generated by the random variable ${Z}$, that is, the smallest $\sigma$-algebra that contains the sets ${Z}^{-1}(S) \subset \Omega$, where $S\subset \R$ is an open set We recall that when $F:\Omega \to \R$ satisfies $F=F(\omega) \in L^1(\Omega;d\Prob)$, then $\Expec(F|\mathcal B_{Z})(\omega)$ is the $\mathcal B_{Z}$-measurable random variable that satisfies
 \beq\label{def: conditional expectation}
    \int_A \Expec (F|{\mathcal B_{Z}} )(\omega) \, d\Prob(\omega)
    = \int_A F(\omega) \, d\Prob(\omega)
 \eeq
 for all sets $A\in \mathcal B_{Z}$.

Roughly speaking, $\Expec({F}\,| \B_{Z})$ denotes the expectation of  a random variable 
${F}=F(\omega)$ under the condition that  ${Z}$ is known. More precisely, by \cite[Section 10.1 and Thm.\ 4.2.8]{Dudley},
 there is a measureable function $g_{F}:\R\to \R$ such that 
\beq
   \Expec({F}\,| \mathcal{B}_{{Z}}) = g_{F}({Z}) = g_{F}({Z}(\omega)) ,\quad \Prob\hbox{-a.e.}
\eeq
that is, $\Expec({F}\,| \mathcal{B}_{Z})$ can be considered as deterministic function of ${Z}$.
To simplify notations, one uses for the conditional expectation of the random variable ${F}$, under the condition that ${Z}$ is given, the notation
\beq
   \Expec({F}\,| \mathcal{B}_{Z})=\Expec({F}\,| \,{Z}) ,
\eeq
where the right-hand side is in fact equal to $g_{F}({Z})$. We emphasize that as ${Z}$ is a random variable, also $\Expec({F}\,| \B_{Z})=\Expec({F}\,| \,{Z})$ is  a random variable.

Let $H = L^2(\Omega;\mathcal{B}_{Z},d\Prob)$ be the set of $\R$-valued functions $u = u(\omega)$ that satisfy $u \in L^2(\Omega;\Sigma,d\Prob)^{}$ and
are $\B_{Z}$-measurable.
Observe that $H \subset L^2(\Omega;\Sigma,d\Prob)^{}$ is a closed subspace of the Hilbert space $L^2(\Omega;\Sigma,d\Prob)^{}$.

By \cite[Thm.\ 4.2.8]{Dudley}, for any $u \in L^2(\Omega;\mathcal{B}_{Z},d\Prob)^{}$ there is a Borel-measurable function $g$ such that $u(\omega) = g({Z}(\omega))$, that is, $u = g \circ {Z}$, $\Prob$-a.e. in $\Omega$.

By \eqref{def: conditional expectation}, we have 
 \beq\label{properties of conditional expectation}
    \bra \Expec (F|{\mathcal B_{Z}} ),g(\omega)\cet_{L^2(\Omega;\Sigma,d\Prob)} 
    =   \bra F,g(\omega)\cet_{L^2(\Omega;\Sigma,d\Prob)} 
 \eeq
 for indicator functions $g={\bf 1}_A$ with all sets $A\in \mathcal B_{Z}$.
As such indicator functions span a dense set in $H$, we have that 
\eqref{properties of conditional expectation} holds for all $g\in H$. 
As $\Expec (F|{\mathcal B_{Z}} )(\omega) \in H$, this yields that
   \beq\label{conditional expectation is projector}
 \Expec (F|{\mathcal B_{Z}} )=P_HF
  \eeq
where 
 \beq
   P_H :\ L^2(\Omega;\Sigma,d\Prob) \to L^2(\Omega;\Sigma,d\Prob)
\eeq 
is the orthogonal projector onto the set $H=L^2(\Omega;\mathcal{B}_{Z},d\Prob)$, that is,
Ran$(P_H) = H$. 
In the main text we use extensively that fact that 
\beq
   P_H {F} = \hbox{arg} \min \|{F}-u\|_{L^2(\Omega;\Sigma,d\Prob)^{}}^2
\eeq
subject to the condition $u\in H=L^2(\Omega;\B_{Z},d\Prob)$.

\subsection*{Acknowledgements}

MVdH gratefully acknowledges support from the Department of Energy under grant DE-SC0020345, the Simons Foundation under the MATH + X program, and the corporate members of the Geo-Mathematical Imaging Group at Rice University. CW was funded by Total. ML was supported by Finnish Centre of Excellence in Inverse Modelling and Imaging and Academy of Finland grants 284715, 312110. The authors are indebted to an anonymous referee for many suggestions that greatly improved the paper.

\bibliographystyle{siam}
\bibliography{MLIPref}

\end{document}